\documentclass{amsart}
\usepackage{master}
\usepackage[margin=1.3in]{geometry}

\renewcommand{\pdual}[1]{{#1}^*}
\renewcommand{\pt}{*}

\begin{document}
\title{Projective symmetries of three-dimensional TQFTs}
\author{Jackson Van Dyke}
\date{\today}
\begin{abstract}
Quantum field theory has various projective characteristics which are captured by
what are called anomalies.
This paper explores this idea in the context of fully-extended three-dimensional
topological quantum field theories (TQFTs).

Given a three-dimensional TQFT (valued in the Morita $3$-category of fusion categories),
the anomaly identified herein is an obstruction to gauging a
naturally occurring orthogonal group of symmetries.
In other words, the classical symmetry group almost acts: There is a lack of coherence at
the top level.  This lack of coherence is captured by a ``higher (central) extension'' of
the orthogonal group, obtained via a modification of the obstruction theory of
Etingof-Nikshych-Ostrik-Meir \cite{ENO:homotopy}.
This extension tautologically acts on the given TQFT/fusion category, and this 
precisely classifies a \emph{projective} (equivalently \emph{anomalous}) TQFT.
We explain the sense in which this is an analogue of the classical spin representation. 
This is an instance of a phenomenon emphasized by Freed \cite{F:what_is}:
Quantum theory is projective. 

We also establish, under some assumptions, 
a general relationship between the language of projectivity/anomalies and the language of
topological symmetries.
We also identify a universal anomaly associated with any theory which is appropriately
``simple''.
\end{abstract}
\maketitle
\tableofcontents

\section{Introduction}


Quantum field theory is fundamentally projective \cite[\S 3]{F:what_is}.
For instance, the correlation functions of a given QFT are unchanged by
tensoring with an invertible theory.
One example of this phenomenon, in quantum mechanics, 
is the well-known fact that the pure states form the \emph{projectivization} of the
Hilbert space of all states.
From the perspective of symmetries, this projectivity appears when one asks if a certain
group acts coherently on the Hilbert space:
The (symplectic) group of symmetries of the phase space will only act projectively, i.e.
only on the pure states.

This projectivity is captured by what is called an \emph{anomaly}.
Though anomalies often assume various mathematical guises, 
we will work with anomalies which are captured by invertible once-categorified theories.
Anomalous theories are then theories defined relative to these invertible theories
\cite{F:what_is}.

We define the equivalent notion of a \emph{projective} TQFT defined on $\Bord^X_d$,
following suggestions of Freed \cite{F:what_is}. Namely, we formally define a
``projectivization'' construction, $\PP$, of a linear target category (\Cref{def:P}).
For example, $\PP\Vect$ is an avatar of the category of projective spaces (see
\Cref{rmk:PVect}) and serves as an appropriate target for non-extended projective TQFTs of
any dimension.
The definition of the projectivization (\Cref{def:P}) is general enough to apply to nearly
any target category of interest, and manifestly classifies anomalous theories valued in
that category.
For example, by forming the projectivization of the even higher Morita categories
\cite{JFS}, we obtain a concrete target category for projective TQFTs in any dimension. 

We define the notion of a non-anomalous framed theory in any
dimension having an anomaly as an $X$-theory in \Cref{defn:X_anom}. I.e. it is
well-defined as a functor out of $\Bord^\fra_d$, but it is only a relative theory on
$\Bord^X_d$.
In particular, this includes the notion of a framing anomaly (\Cref{exm:framing_anomaly}),
as well as the 't Hooft anomalies discussed in this paper. 
We discuss a target for projective $3$-dimensional TQFTs in detail (\Cref{sec:PFus}).
We also show that $\PP$ commutes with iterated looping (\Cref{thm:P}) meaning, for
example, that \emph{projective} $3$-dimensional TQFTs give rise to \emph{projective}
mapping class group representations of surfaces.

Under some assumptions (see \Cref{hyp:sigma}), anomalous/projective theories 
can be rephrased via the language of topological symmetry \cite{FMT}.
Let $G$ be a finite group.
If the anomaly theory, as a once-categorified TQFT on $\Bord^{BG}_d$, happens to be
trivializable, then this is equivalent to having a ``sandwich structure'' \cite{FMT}
(i.e. an `action') of pure topological $G$-gauge theory of one dimension
higher.
This is still possible if the anomaly theory is nontrivial:
Now the topological $G$-gauge theory is twisted by the cocycle underlying the anomaly theory
(see \Cref{rmk:map_of_spectra}).
This statement identifies certain higher $G$-representations with
certain higher modules over the higher group algebra of $G$.

For general $\pi$-finite $X$, we formalize this in  
\Cref{thm:anom_symm} 
as an equivalence between four different avatars of projectivity.
We also define a universal anomaly (\Cref{cor:anom_symm}) 
for an object of any target satisfying a certain notion of simplicity (\Cref{defn:ss}).


Let us restrict our attention to anomalies of $3$-dimensional theories.
The anomaly theories themselves will be valued in the (even higher \cite{JFS}) Morita
$4$-category\footnote{We will generally use $n$-category to denote 
$\left(\infty , n\right)$-category in the complete Segal space formalism, and otherwise we
specify that we are considering a discrete $n$-category. 
We will sometimes still write $\left(\infty , n\right)$-category to emphasize the
$\infty$-categorical nature of the target categories.
In particular, we work with the even higher Morita category 
of $\EE_2$-algebras in (a suitable subcategory of) the
symmetric-monoidal $2$-category of presentable categories.
See \Cref{rmk:fusion} for the relationship with the abelian category theory literature.}
of braided tensor categories $\BrTens$. 
We will restrict out attention to the subcategory consisting of 
rigid, finite, and semisimple tensor categories, i.e. fusion categories.
One reason for this is that they are fully-dualizable \cite{BJS:dualizable}.

As usual, endomorphisms of the identity forms the associated $3$-dimensional target:
the Morita category of fusion categories, $\Fus$. 
As is discussed in \Cref{sec:P}, the projectivization of $\Fus$, written $\PP\Fus$, is
constructed in such a way that functors:
\begin{equation}
\Bord_3^{\left(X , \z\right)} \to \PP \Fus
\label{eqn:proj_3d}
\end{equation}
correspond precisely to natural transformations between the unit and
any invertible representation of the same bordism category on $\BrFus^\times$.

One justification for restricting to this dimension and setting, is that $\pi_0
\BrFus^\times$ is highly nontrivial, and closely related to the Witt group of braided
fusion categories \cite{BJSS:invertible}.
This means there are more anomaly theories valued in $\BrFus$ than there are in, say, the
category of vector spaces or the Morita category of algebra (or fusion categories)
since they have trivial groups of invertible objects. 

In slightly more detail, $\PP\Fus$ 
consists of all $1$-morphisms from an invertible object to the unit in $\BrFus$.
I.e. these are monoidal categories with a compatible action of a nondegenerate braided
fusion category.
The morphisms are (op)lax squares, and the higher morphisms are (op)lax transfors
as defined in \cite{JFS}.

An important example of a $3$-dimensional TQFT is the Witten-Reshetikhin-Turaev (WRT) TQFT
associated to a modular tensor category \cite{Wit:RT,RT2,T:book}.
Given a nondegenerate braided fusion category $\cA$,
the avatar of this theory which we will consider is the TQFT 
$\WRT_{\cA} \colon \Bord_3^\fra \to \PP \Fus$,
defined by sending the point to the fully-dualizable object given by $\cA$ as a module
over itself \cite{Fre:aspects,Ben:unit}.
We prove that this theory always defines a projective $G$-theory, for any $G$ acting
fully-coherently by braided autoequivalences of $\cA$ in \Cref{thm:WRT}.

WRT theories do not extend to the point as oriented (or even framed) TQFTs valued in
$\Fus$ whenever $\cA$ is in a nontrivial Witt class. 
This illustrates one feature of the projective target category $\PP\Fus$: it allows for
WRT theories (with a nontrivial framing anomaly) to be extended to the point.


Drinfeld centers of fusion categories are examples of braided fusion categories $\cA$
which are in the trivial Witt class \cite{BJSS:invertible}. 
Every fusion category $\cC$ is fully-dualizable \cite{DSPS:dualizable}, so it defines a
fully-extended framed $3$-dimensional TQFT by the Cobordism Hypothesis
\cite{L:CH}.
This theory is sometimes called the Turaev-Viro (TV) theory associated with $\cC$.

The TV theory for $\cC$ is related to the WRT theory (in the above sense)
associated to $\cZ\left(\cC\right)$ as follows:
The anomaly theory of $\WRT_{\cZ\left(\cC\right)}$ is trivializable:
$\cA$ is in the trivial Witt class, and a choice of such a $\cC$ determines a 
trivialization of the anomaly. Composing with the trivialization yields an absolute
$3$-dimensional theory, which sends the point to 
\begin{equation*}
\cC \tp_{\cZ\left(\cC\right)} \cZ\left(\cC\right) \simeq \cC \in \Fus \ ,
\end{equation*}
and therefore agrees with the TV theory for $\cC$.
This is a sense in which ``TV for $\cC$ agrees with WRT for $\cZ\left(\cC\right)$''.

The anomaly/projectivity of the TV theory only appears when we attempt to upgrade the
theory to be defined on a bordism category with different background fields. 
An instance of this is given by the question of upgrading a theory to be equivariant for
the action of a certain group.
The anomaly is then encoded by the Postnikov tower of the higher automorphism groupoid of
the fusion category. 

The homotopy theoretic structure of the groupoid of automorphisms is studied in
\cite{ENO:homotopy}.
We utilize the obstruction theory contained therein to produce certain 
$3$-groups (\Cref{defn:lip,defn:pin,defn:spin}).
These are central\footnote{See \Cref{rmk:central_extension}.} extensions of finite orthogonal
groups by higher scalars (e.g. $B^2\units{\CC}$ or $B^2 \mu_p$).
A detailed analogy with the classical spin representation is developed in \Cref{sec:analogy}.
One can also think of this as being analogous to the Weil representation of the
metaplectic group.

These $3$-group extensions of the finite orthogonal group each act on the
framed TQFT. Equivalently, the orthogonal group acts projectively, with projectivity given
by the cocycle classifying the extension. 
In terms of anomalies, the original framed theory defines an anomalous/projective
theory on $\Bord_3^{BG}$, where $G$ is any group mapping to the finite orthogonal group
(\Cref{thm:pin_anom,cor:spin_anom}).
When the fusion category in question consists of vector spaces graded by $\FF_p^n$,
the obstruction / anomaly is trivializable. 
Therefore we obtain the existence of an $\O \left(2n,\FF_p\right)$-equivariant
\emph{linear} TQFT associated to $\FF_p^n$ (\Cref{cor:Fp}).
Under the analogy with the Weil representation, the fact that this class is trivializable
is analogous to the splitting of the Weil representation over a finite field
\cite{GH:finite,GH:categorical}.

On the other hand, given any finite group $G$ and $4$-cocycle we show in
\Cref{cor:nontriv_O4anomaly}
that there exists some $3$-dimensional TQFT with anomaly classified by this cocycle.
In other words: the classification of projective $3$-dimensional $G$-theories is richer
than the classification of (linear) $G$-theories.


This paper is organized as follows. In the remainder of the introduction, we give a more
detailed summary of our results and give some indication as to how this work is related to
the literature.

In \Cref{sec:projective_TQFTs} we implement suggestions of Freed \cite{F:what_is} to
provide a formal definition of a projective TQFT.
We define the projectivization functor in \Cref{sec:P}. We 
apply it to various (even higher)
Morita categories in \Cref{sec:P_exm}, where we also explain some concrete examples of
projective targets.
The definition of a projective TQFT is in \Cref{sec:proj}, and the relationship with 
extensions of bordism categories is discussed in \Cref{sec:ext_bord}.

In \Cref{sec:anomalies} we discuss anomalies in general. 
In \Cref{sec:proj_anom} we explicitly describe how a projectivity cocycle classifies a
once-categorified invertible TQFT, following \cite{F:what_is}. 
In \Cref{sec:anom_sand,sec:univ_anom} we provide the connection with the sandwich picture
in \Cref{thm:anom_symm} and the ``universal anomaly'' in \Cref{cor:anom_symm}.

In \Cref{sec:fusion_prelim,sec:higher_groupoids,sec:pointed,sec:obstruction} we collect
the necessary preliminaries regarding (braided) fusion categories, mostly from
\cite{ENO:homotopy}.
In \Cref{sec:extensions} we introduce the $3$-groups which are processed into anomaly
theories, and in \Cref{sec:spin,sec:analogy} we explain a detailed analogy between
these $3$-groups and the groups $\Pin$ and $\Spin$.

In \Cref{sec:dual_inv,sec:anom_Z} we review some facts concerning fully-extended
$3$-dimensional TQFTs and Drinfeld centers, before moving onto anomalies.
A projective action of the braided automorphism $2$-group of the center (resp. orthogonal group
of the underlying metric group) is considered in \Cref{sec:anom_k3} (resp.
\Cref{sec:anom_O4}), and their avatars as \emph{projective} theories are discussed in
\Cref{sec:3d_proj}.
The anomalies/projectivity are translated into the language of topological symmetries
in \Cref{sec:3d_symm}.
The functorial assignment of a linear $\O\left(\FF_p\right)$-equivariant
$3$-dimensional TQFT to a vector space over $\FF_p$ is considered in
\Cref{sec:Fp_quant}.

In \Cref{sec:TQFT} we recall general facts about TQFTs, mostly from \cite{L:CH}.
In particular, we discuss the particular version(s) of the Cobordism Hypothesis we will
use in this paper in \Cref{sec:cob_hyp}.
Relative TQFTs and invertible TQFTs are considered in
\Cref{sec:relative_theories,sec:invertible}.

\Cref{sec:top_symm} is dedicated to a discussion of topological symmetry and TQFTs
associated to $\pi$-finite groupoids. 
We review some definitions and facts from \cite{FMT}, and define a notion of
a ``reduction'' of topological symmetry in \Cref{sec:reduction}.


\noindent
\textbf{Acknowledgements.} 
I warmly thank my advisor David Ben-Zvi for suggesting this project, and for his constant
help and guidance.
I would also like to thank Dan Freed, David Jordan, 
Lukas M\"uller, German Stefanich, Rok Gregoric, and 
Will Stewart for helpful discussions. 

Part of this project was completed while the author was supported by 
the Simons Collaboration on Global Categorical Symmetries.

This research was also partially completed while visiting the Perimeter Institute for
Theoretical Physics.
Research at Perimeter is supported by the Government of Canada through the Department of
Innovation, Science, and Economic Development, and by the Province of Ontario through the
Ministry of Colleges and Universities.

Jackson van Dyke is funded/acknowledges support by the Deutsche Forschungsgemeinschaft
(DFG, German Research Foundation) – SFB 1624 – ``Higher structures, moduli spaces and
integrability'' – 506632645.

\subsection{Anomalies}


Anomalies have been studied for many years, under many different guises.
Early appearances include \cite{Stein,Adler,BJ,Hooft:anom}, and
more contemporary resources, which directly inspired this paper, include
\cite{F:what_is,Fr:anomalies,FHLT}.
See \cite{F:what_is} and the references contained therein for more thorough resources.
We take the point of view that an anomaly is a feature, rather than a bug
\cite{Hooft:anom,F:what_is}.
Anomalies can, and often are, discussed for more general classes of QFTs, but we will
restrict our attention to topological anomalies of TQFTs.
The main proposal of this work is that the Postnikov tower of the higher automorphism
groupoid is a source of anomalies/projectivity. 

In \Cref{sec:anomalies}, we define the notion of a non-anomalous framed theory in any
dimension having an anomaly as an $X$-theory in \Cref{defn:X_anom}. I.e. it is
well-defined as a functor out of $\Bord^\fra_d$, but it is only a relative theory on
$\Bord^X_d$.
In particular, this includes the notion of a framing anomaly (\Cref{exm:framing_anomaly}),
as well as the 't Hooft anomalies discussed in this paper. 

This notion of an anomalous theory is closely tied to the notion of a \emph{projective}
TQFT, as outlined in \cite{F:what_is}.
We introduce projective target categories for such theories in \Cref{sec:P}.
Given a linear $d$-dimensional target category, i.e. a symmetric monoidal
$\left(\infty , d\right)$-category, in \Cref{def:P} we give a formal definition for a
\emph{projectivization} of any such target.
A $d$-dimensional projective TQFT is then a functor from the bordism
category of interest to the projectivization of the target category of interest.
In \Cref{prop:proj_anom} we record the nearly tautological statement that 
this notion of a projective theory is equivalent to an anomalous theory.

Given any sufficiently nice (see \Cref{rmk:morita} for the precise condition) 
symmetric monoidal $\left(\infty , k\right)$-category $\cS$, the even higher Morita
category \cite{JFS} $\Alg_n \cS$ is a symmetric-monoidal $\left(\infty ,
k+n\right)$-category. 
The projectivization $\PP \Alg_n \cS$ is a candidate target for projective
$\left(k+n\right)$-dimensional TQFTs. 
This gives a well-defined notion of a (fully-extended) projective TQFT in any dimension.
An explicit example of interest is the projectivization of $\Vect$ (considered in
\Cref{sec:PVect}), which can serve as a target for non-extended TQFTs of any
dimension.

Let $\cT$ denote the $\left(d+1\right)$-dimensional target (i.e. the target for bulk
theories).
Write $\Om \cT = \End_\cT\left(1\right)$, and $\Om^k$ for the $k$-fold endomorphisms of
the unit. 
We show in \Cref{thm:P} that $\PP$ depends naturally on $\cT$, and that $\PP$ commutes
with $\Om^k$. 
In particular, if the linear category $\cT$ satisfies $\Om^{d} \cT \simeq \Vect$, then
we have that:
\begin{equation*}
\Om^{d-1} \PP \Om \cT \simeq \PP \Om^{d} \cT \simeq \PP \Vect \ .
\end{equation*}
Since $\PP\Vect$ is an avatar of the category of projective vector spaces
(see \Cref{rmk:PVect}),
one upshot of this statement is that the mapping class group representations that come
from a \emph{projective} TQFT are indeed \emph{projective} mapping class group
representations.

\begin{rmk}[Framing anomalies]
The results in this paper concerning three-dimensional theories are restricted to the
framed setting, in order to avoid the discussion of
an $\SO$-fixed point structure on (braided) fusion categories.
If a given theory cannot be upgraded from a framed theory to an oriented theory in this
manner, then it is said to have a nontrivial \emph{framing anomaly}.
(See \Cref{exm:framing_anomaly} for a formal definition in terms of invertible
once-categorified theories.)
The anomalies in this paper which correspond to projective actions of (higher) groups are
not entirely of the same sort as a framing anomaly:
The tangential structure on our source $\Bord^X_d$ (borrowing notation from \cite{L:CH})
is the trivial map $X \to B\O\left(d\right)$. 

This is not to say that the discussion in this paper is entirely orthogonal to the
discussion of framing anomalies: The four-dimensional Crane-Yetter theory (see
\Cref{sec:dual_inv}) plays an important role in studying the framing anomaly of
the $3$-dimensional Reshetikhin-Turaev (and possibly Turaev-Viro) theory
\cite{RT1,RT2,Wit:RT,Wal:RT,FHLT}.
\label{rmk:framing_anomaly}
\end{rmk}

\subsubsection{Fully-extended anomalous WRT theory}

Unless otherwise specified, all categories are 
$\bk$-linear, where $\bk$ is algebraically closed of characteristic
zero. 
The anomaly theories of the anomalous $3$-dimensional theories are valued 
in the Morita $4$-category of \emph{braided} fusion categories,\footnote{See
\Cref{sec:fusion_prelim} for detailed definitions.} written $\BrFus$. 
By constructing the projectivization of $\Om\BrFus\simeq \Fus$ (where $\Fus$ is the Morita
$3$-category of fusion categories) we are able to introduce a concrete definition of a \emph{projective
$3$-dimensional TQFT} in \Cref{sec:3d_proj}.
This is of particular interest because the low energy \emph{topological} theory modeling a
gapped system is in fact a \emph{projective}, i.e. anomalous, TQFT \cite[Interlude
(P.12)]{F:what_is}.

The objects of $\PP\Fus$ are $1$-morphisms in $\BrFus$ from an invertible object to the
unit. Spelling this out, an object consists of some invertible object $\cA$ of $\BrFus$, 
along with an $\EE_1$-algebra object of the monoidal $2$-category of $\cA$-modules.
For instance, the regular module $\cA_\cA$ is such a $1$-morphism. 
As an object of the arrow category $\BrTens^\down$, any such regular module defines a
fully-dualizable $1$-morphism \cite{Ben:unit}. 

The upshot of this discussion, is that $\cA$ defines a projective TQFT:
\begin{equation*}
\WRT_{\cA} \colon \Bord_3^\fra \to \PP \Fus \ .
\end{equation*}
The object $\cA\in \BrFus$ is invertible on account of being nondegenerate
\cite{BJSS:invertible}. Therefore the TQFT 
$\CY_{\cA} \colon \Bord_4^\fra \to \BrFus$ sending the point to $\cA$ is also invertible. 
It is a fully-extended framed version of the Crane-Yetter TQFT attached to $\cA$, and it
is the anomaly theory of $\WRT_{\cA}$.

We show that $\WRT_\cA$ can always be given a projective action of the $2$-group of
braided autoequivalences of $\cA$.
The following appears as \Cref{thm:WRT} in the text.

\begin{thm*}
Let $\cA$ be a nondegenerate braided fusion category, and consider a monoidal functor 
$\rho \colon G \to \Aut_{\EqBr}\left(\cA\right)$. 
There is an anomalous $G$-TQFT
\begin{equation*}
\WRT_\cA^G \colon 
\Bord^{B\Aut_{\EqBr}\left(\cA\right)}_3 \to \PP \Fus 
\end{equation*}
which agrees with $\WRT_\cA$ upon restriction to trivial $G$-bundles.
The anomaly theory of $\WRT_{\cA}^G$, as a functor from $\Bord^{BG}_4$ to $\BrFus$, agrees
with $\CY_\cA$ upon restriction to trivial $G$-bundles.
\end{thm*}

In words, this theorem says that we can upgrade WRT and Crane-Yetter to be $G$-theories 
in such a way that WRT still lives relative to Crane-Yetter.

One reason to restrict our attention to this dimension, is that 
there is more ``room'' for anomaly theories than in lower dimensions:
The groups $\pi_0 \units{\BrFus}$ and $\pi_0 \units{\BrTens}$ are closely related to the
Witt group of braided fusion categories \cite{BJSS:invertible} which is, in particular,
infinite.
In lower dimensions, in the characteristic zero non-super case, this does not happen:
Every invertible object is trivializable in the category of vector spaces, the Morita
category of algebras, and the Morita category of fusion categories.

\begin{rmk}
As in \cite{BJS:dualizable,BJSS:invertible}, we will prefer to work in the more general 
$\infty$-categories consisting of objects 
which are $\Ind$-completions of what are usually called fusion categories in the abelian
category literature. 
As explained in \Cref{rmk:fusion}, various existing results allow us to translate facts
back and forth between these two settings
\cite{DSPS:dualizable,BJS:dualizable,BJSS:invertible}.
For example, we make frequent use of facts from \cite{ENO:homotopy}.
Our perspective, is that this is a feature not a bug: We can continue to use the arguments
from the finite abelian category theory literature, however the target categories
considered in this paper contain more general objects. For example, the iterated looping
to a $1$-category is the $\Ind$-completion of the category of finite-dimensional vector
spaces, which in-particular contains objects which represent infinite-dimensional vector
spaces.
\end{rmk}

\subsubsection{An anomaly of a fully-extended framed TV theory}
\label{sec:intro_fus_anom}

Let $\cC$ be a fusion category (see \Cref{rmk:fusion} for the details).
The Morita $3$-category of fusion categories, $\Fus$, has
duals \cite{DSPS:dualizable}, so $\cC$ classifies a (framed) $3$-dimensional TQFT $F$ by
the Cobordism Hypothesis \cite{L:CH}.

The theory $F$ is well-defined as a framed theory. However, $F$ will turn out to be
``anomalous'' or ``projective'' as a theory which is equivariant with respect to a certain
natural group, or in some cases $2$-group, of symmetries. 
In physical terms, we are identifying an obstruction to gauging the symmetry. 

The Drinfeld center $\cZ\left(\cC\right)$ is an object of the $4$-category $\BrFus$, which
also has duals \cite{BJS:dualizable}, so this classifies a $4$-dimensional TQFT $\z$. 
The Drinfeld center $\cZ\left(\cC\right)$ tautologically acts on $\cC$, meaning the 
theory $F$ can be upgraded to a relative theory $F_\z \colon \z \to 1$.
The Drinfeld center turns out to be invertible \cite{BJSS:invertible}, and in fact 
$\cZ\left(\cC\right)$ is trivializable in $\BrFus$: The module $\cC$ in
$\Hom_\BrFus\left(\cZ , 1\right)$ is an equivalence.
I.e. $F_\z \colon \z \to 1$ is an equivalence of theories.
The anomalous symmetries of the original theory $F$ appear when we consider the symmetries
of $\z$, or equivalently $\cZ\left(\cC\right)$. 

In \cite{ENO:homotopy}, the homotopy type of the $3$-type
$B\Aut_\Fus\left(\cC\right)$ is studied.
There is a natural map:
\begin{equation*}
\Aut_{\Fus}\left(\cC\right) \to \Aut_{\EqBr}\left(\cZ\left(\cC\right)\right)
\ ,
\end{equation*}
where $\Aut_{\EqBr}$ denotes the $2$-group of braided autoequivalences.
As it turns out, this map is an equivalence when restricted to the truncation of
$\Aut_\Fus\left(\cC\right)$ to a $2$-group.
The top nontrivial homotopy group of $B\Aut_\Fus\left(\cC\right)$ is $\pi_3 =
\units{\bk}$, so the upshot of this discussion is that the higher automorphism groupoid
defines a bundle:
\begin{equation}
\begin{cd}
B^3 \units{\bk} \ar{r}\ar{d}&
B\Aut_\Fus\left(\cC\right)\ar{d} \ar{r}& *\ar{d}\\
* \ar{r}&
B\Aut_{\EqBr}\left(\cZ\left(\cC\right)\right)\ar{r}{k} & B^4 \units{\bk}
\end{cd}
\label{eqn:BAut_bundle}
\end{equation}
The map $k$ defines a TQFT:
\begin{equation*}
\al_k \colon \Bord_4^{B\Aut_{\EqBr}\left(\cZ\left(\cC\right)\right)} \to \BrFus
\ .
\end{equation*}
The original theory $F$ canonically defines a relative theory:
\begin{equation*}
F_k \colon \al_k \to 1 \ ,
\end{equation*}
where now $1$ denotes the trivial theory with source
$\Bord_4^{B\Aut_{\EqBr}\left(\cZ\left(\cC\right)\right)}$.
(See \Cref{prop:alpha_c}.)

I.e. we have seen that the ordinary framed theory $F$ is well-defined as an anomalous
$\Aut_{\EqBr}\left(\cZ\left(\cC\right)\right)$-equivariant theory.
The following is stated as \Cref{thm:k3} in the body of this paper.

\begin{thm*}
The framed TQFT $F$ has an $X= B\Aut_{\EqBr} \left(\cZ\left(\cC\right)\right)$-anomaly as
in \Cref{defn:X_anom}. 
In particular, the fusion category $\cC$ itself defines an anomalous theory:
\begin{equation*}
F_{k_3} \colon \alpha_{k_3} \to 1_{B\Aut\left(\cZ\left(\cC\right)\right)} \ .
\end{equation*}
Furthermore, if $\cZ\left(\cC\right)$ is pointed, and the cohomology class classifying the braiding of
$\cZ\left(\cC\right)$ (as in \Cref{rmk:q}) is nontrivial, then the anomaly $\al_{k_3}$ is nontrivial.
\end{thm*}

If $\cZ\left(\cC\right)$ is pointed (and the group of isomorphism classes of simple
objects $A$ has odd order) then the $2$-group
$\Aut_{\EqBr}\left(\cZ\left(\cC\right)\right)$ splits, and in fact there is a canonical
splitting. I.e. we can (canonically) write the classifying space as a semidirect product of
Eilenberg-MacLane spaces $B\pi_1$ and $B^2 \pi_2$.
The fundamental group is $\pi_1 = \O\left(A , q\right)$, where $\left(A , q\right)$ is the
pre-metric group classifying the pointed braided fusion category $\cZ\left(\cC\right)$
(see \Cref{sec:pointed}).
The upshot of this is that we can pull the class $k$ back to $B\O\left(A\right)$, along
this canonical splitting, to obtain a class:
\begin{equation*}
O_4\left(A\right)
\in H^4\left(B\O\left(A\right) , \units{\bk}\right)
\ .
\end{equation*}
This defines an anomaly theory:
\begin{equation*}
\al_{O_4} \colon  \Bord_{4}^{B\O\left(A\right)} \to \BrFus \ ,
\end{equation*}
and $F$ defines a relative theory:
\begin{equation*}
F_{O_4} \colon \al_{O_4} \to 1 \ .
\end{equation*}

The homotopy fiber of $O_4\left(A\right)$ is a $B^3 \units{\bk}$-bundle
over $B\O\left(A\right)$.
This is the classifying space of a $3$-group we call $\LLip\left(A\right)$
(\Cref{defn:lip}).
This is an analogue of the Lipschitz (a.k.a. Clifford) group, reviewed in \Cref{sec:spin}.
As in the classical story, $\LLip\left(A\right)$ 
can be cut down to an extension by a finite group, rather than all scalars. 
We call this $\PPin$ (\Cref{defn:pin}), and it is analogous to the group $\Pin$.
We can restrict to $\SO\left(A\right) \inj \O\left(A\right)$ to obtain $\SSpin$
(\Cref{defn:spin}), an analogue of $\Spin$.
This is a part of an analogy in \Cref{tab:analogy}, which is fleshed out in detail in
\Cref{sec:analogy}.

\begin{table}[ht]
\centering
\caption{A detailed analogy between the Clifford algebra and spin representation
associated to (a Lagrangian in a) quadratic vector space, 
and the braided fusion category and fusion module category associated to (a Lagrangian
in a) finite metric group. See \Cref{sec:spin} for the former, and \Cref{sec:fusion} for
the latter.
This analogy is explained in detail in \Cref{sec:analogy}.}
\label{tab:analogy}
\mathtabular{
\begin{tabular}{|c|c|}
\hline
$1$-dimensional & $3$-dimensional
\\ \hline\hline
$\left(V , q\right)$
& $\left(A , q\right)$ 
\\ \hline
$\SO\left(V , q\right) \subset \O\left(V , q\right)$ 
& $\SO\left(A , q\right) \subset \O\left(A , q\right)$
\\ \hline
$\units{\bk}$ 
& $B^2 \units{\bk}$
\\ \hline\hline
$\cliff\left(V\right)$
& $\cA = \left(\Vect\left[A\right] , * , \beta_q\right)$
\\ \hline
$\left\{x,y\right\} = b_q\left(x,y\right)$ 
& $\b_q \colon \bk_a * \bk_b \lto{b_q\left(a,b\right) \id} \bk_b * \bk_a$
\\ \hline
$V\rtimes \O\left(V,q\right)$
& $\Aut_{\EqBr}\left(\cA\right)$
\\ \hline\hline
$\Gamma = \Lip$
& $\LLip\left(A , q\right)$
\\ \hline
spinor norm 
& $N_{\left(A , q\right)}$
\\ \hline
$\Pin\left(V , q \right)$ 
& $\PPin \left(A , q \right)$
\\ \hline
$\Spin\left(V , q \right)$
& $\SSpin \left(A , q \right)$
\\ \hline
$\left\{\pm 1\right\} \inj \units{\bk}$ & 
$B^2 \mu_{l^4} \to B^2 \units{\bk}$
\\ \hline
See \Cref{rmk:heis}
& $B^2 \units{\cA}$
\\ \hline
See \Cref{rmk:heis}
& $\Pic\left(\cA\right)$ 
\\ \hline\hline
$V\simeq L \dsum \pdual{L}$
&$A\simeq L \dsum \pdual{L}$
\\ \hline
$\ext^\dott \pdual{L}$
&
$\cC = \left(\Vect\left[\pdual{L}\right] , *\right)$
\\ \hline
$\End\left(\ext^\dott \pdual{L}\right) \simeq \cliff$
&
$\Aut_\Fus\left(\cC\right) \simeq \Pic\left(\cA\right)$
\\ \hline
\end{tabular}}
\end{table}

Along the way, we establish a classification of pointed braided fusion categories which
are Drinfeld centers 
(\Cref{cor:class_Z}).
This result is a restriction of the familiar classification of pointed
braided fusion categories in \cite{JS,EM:2}.
In particular, we describe the braiding on $\cZ\left(\Vect\left[L\right]^\tau\right)$
explicitly in terms of the braiding on $\cZ\left(\Vect\left[L\right]\right)$ in
\Cref{prop:beta_tau}.
The upshot of this is that we see that the ``polarizable'' metric groups are not all of
the form $\left(L\dsum \pdual{L} , \ev\right)$: in general the quadratic form obtains a
factor coming from the twist $\tau$ on $L$.
Note however that many such twists do not result in a pointed center, so $\tau$ must
satisfy a restrictive hypothesis. 

The following results are stated as \Cref{thm:pin_anom,cor:spin_anom} 
in the body of this paper.

\begin{thm*}
Consider a pointed Drinfeld center $\cZ\left(\cC\right)$ classified by a polarized metric
group $\left(A , q , L\right)$.
The nonanomalous framed theory 
\begin{equation*}
F \colon \Bord^{\fra}_3 \to \Fus 
\end{equation*}
sending the point to $\cC$ has an $\O\left(A , q\right)$-anomaly in the sense of
\Cref{defn:X_anom}.
I.e. there is an anomaly theory
\begin{equation*}
\al_{c\left(\PPin\right)} \colon\Bord^{B\O\left(A , q\right)}_3 \to \BrFus \ ,
\end{equation*}
and $\cC$ canonically defines an anomalous $\O\left(A , q\right)$-TQFT:
\begin{equation*}
F_{c\left(\PPin\right)} \colon \al_{c\left(\PPin\right)} \to 1_{B\O\left(A , q\right)} \ .
\end{equation*}
\end{thm*}

\begin{cor*}
Restricting the $\O\left(A,q\right)$-anomaly of \Cref{thm:pin_anom} to $\SO\left(A , q\right)$ we obtain
an $\SO\left(A , q\right)$-anomaly of $F$ 
\begin{equation*}
\al_{c\left(\SSpin\right)} \colon \Bord^{B\SO\left(A , q\right)}_4 \to \BrFus \ ,
\end{equation*}
and an anomalous $\SO\left(A , q\right)$-equivariant theory 
\begin{equation*}
F_{c\left(\SSpin\right)} \colon \al_{c\left(\SSpin\right)} \to 1_{B\SO\left(A , q\right)} \ .
\end{equation*}
\end{cor*}

If one must trivialize the anomaly\footnote{An anomaly is really ``part'' of the 
theory, and only acts as an obstruction when you are ``quantizing'', see \cite[\S
4]{F:what_is}.}, and it is not trivializable on all of $\pi_1$, it can be pulled back
along a map $f\colon G\to \pi_1$. 
This results in an anomaly:
\begin{equation*}
\al \colon \Bord_4^{BG} \to \BrFus \ ,
\end{equation*}
and $F$ still defines a relative theory $\al \to 1$.
A \emph{trivialization} is an equivalence:
$1\lto{\sim} \al$, and the \emph{trivialized theory} is the composition:
\begin{equation*}
1\lto{\sim} \al \to 1 \ .
\end{equation*}
Note that this is an endomorphism of the trivial theory defined on $\Bord^{BG}_4$, and
therefore equivalent to a theory
\begin{equation*}
F_G \colon \Bord^{BG}_3 \to \Fus \ .
\end{equation*}

One should think that we started with a theory $F$ (or fusion category $\cC$) and obtained
a $G$-equivariant theory $F_G$ (or fusion category $\cC$ along with fully
coherent\footnote{Fully coherent means that the assignment of a
bimodule to each group element is not only defined up to isomorphism: 
it is a functor from the discrete $3$-category with objects $G$ to the full
$3$-group of $\cC$-bimodules. See \Cref{sec:triv_O4} for more.}
action of $G$ on $\cC$ via bimodules).

\begin{rmk}
By construction, the direct sum of these bimodules for each $g\in G$ is precisely the
$G$-extension associated to the same data in \cite[Theorem 1.3]{ENO:homotopy}.
\end{rmk}

\begin{rmk}
Sometimes the entire cohomology group containing the obstruction vanishes, meaning the
anomaly is trivializable. 
One example of this is $\cC = \Vect\left[L\right]$, for $L$ a vector space over $\FF_p$.
In this case, the (cohomology group and therefore the)
obstruction is shown to vanish in \cite{EG:reflection}. 
In the language of this paper, this means that there is an $\O\left(L\dsum \pdual{L} ,
\ev\right)$-equivariant theory sending the point to $\cC$.
This is stated as \Cref{cor:Fp} in the text.
\end{rmk}

\subsubsection{Characteristic classes}

Anomalies are most interesting when they cannot be trivialized as in \Cref{sec:triv_O4}. 
This happens, for example when we consider the full symplectic group acting on the quantum
harmonic oscillator: The Hilbert space $L^2\left(\RR\right)$ is a nontrivially projective
representation of $\Sp_2\left(\RR\right)$, and the projectivity classifies the universal
cover $\Mp_2\left(\RR\right) \to \Sp_2\left(\RR\right)$.
The corresponding cocycle classifies (a multiple of) 
$w_2 \in H^2\left(B\Sp_2\left(\RR\right)\right)$.

Similarly, finite orthogonal groups have characteristic classes 
\cite{FP:homology_of_classical_groups_over_finite_fields}
and one might wonder if these characteristic classes match the anomaly theories discussed
in this paper. 

More generally, given a $4$-cocycle $\pi \colon BG \to B^4 \units{\bk}$ on a group, one
might wonder what conditions guarantee the existence of an anomalous $3$-dimensional TQFT
(thought of as a nondegenerate braided fusion category $\cB$)
such that the anomaly is described by $\pi$.
\Cref{cor:nontriv_O4anomaly} guarantees the existence of such a theory for finite $G$. 

\subsubsection{Relationship with the literature}

The anomalies studied in this paper are 't Hooft anomalies for discrete internal
symmetries of Dijkgraaf-Witten theories.
These were first studied (in various dimensions) in
\cite{KT:discrete_anomalies,KT:3discrete_anomalies},
and (in all dimensions) in the once-extended functorial language in
\cite{Mul:thesis,MS:anomalies}.

In the $3$-dimensional and pointed/quasi-trivial case, 
the obstruction classes in \cite{KT:discrete_anomalies,Mul:thesis,MS:anomalies}
are shown (in \cite[\S 11.8]{ENO:homotopy}) to be equivalent to the
obstruction classes in \cite{ENO:homotopy}.

The higher groupoids studied in \cite{ENO:homotopy} were studied as symmetries of
$3$-dimensional theories in \cite{FPSV:BrPic,FS:3d_symm}.
The obstruction theory of \cite{ENO:homotopy} is used in \Cref{sec:obstruction} 
to define the $3$-groups which characterize the anomaly theories.
The same obstruction theory used in this paper from \cite{ENO:homotopy} has been used
extensively in the literature to study MTCs in the unitary setting
\cite{CGPW,DGPRZ,DGPRZ:2}.

A partial description of the Brauer-Picard $3$-group of the Asaeda-Haagerup fusion
categories was given in \cite{GJS:BrPic}.
The missing information was precisely this $k$-invariant which classifies the anomaly
studied herein, which was shown to vanish for these examples in \cite{GIS:AH}.

The finiteness results of \cite{DSPS:dualizable,BJS:dualizable,BJSS:invertible,Ben:unit} are referenced in
\Cref{sec:dual_inv} to construct fully-extended, and sometimes invertible, TQFTs which
agree with the theories associated with certain $\pi$-finite spaces
(in the sense of \cite[\S A.2]{FMT}).

The original Crane-Yetter and Reshetikhin-Turaev invariants/TQFTs
\cite{RT1,RT2,TV,CY:4d_TQFT,Del:CS} 
are believed to agree with (oriented upgrades) of the theories discussed in \Cref{sec:WRT}
and \Cref{sec:anom_Z}. 
In particular, \Cref{thm:WRT} is consistent with the idea that 
Crane-Yetter theory encodes an anomaly of the Reshetikhin-Turaev theory
associated to the MTC \cite{Wal:RT,Ben:WRT_CY}.

The notion of a projective $3$-dimensional TQFT is 
presumably closely related to the notion of a ``modular functor''
\cite{S:MF,MS,T:book,Til,BK,BW}, which is roughly a system of projective representations
of mapping class groups. 
Indeed, any projective $3$-dimensional TQFT gives rise to a projective mapping class group
representation for any surface. 
A formal comparison would be quite interesting, but we do not pursue this here.

\subsection{Chern-Simons theory}

Let $G$ be a compact Lie group, and $\lam \in H^4\left(BG , \ZZ\right)$.
Consider the Reshetikhin-Turaev TQFT attached to $G$ at level $\lam$
\cite{RT1,RT2,BM:CS,Stir:CS,FT:gapped}.
Recall this sends the (bounding) circle to the (semisimplificiation of the) category of
representations of the quantum group at fixed root of unity determined by $\lam$.

When $G = T$ is a torus, write $\Pi = \Hom\left(\U\left(1\right), T\right)$ for the
associated lattice.
Then $\lam$ determines a nondegenerate symmetric bihomomorphism 
\begin{equation*}
\lr{\, , \,} \colon \Pi \times \Pi \to \ZZ \ .
\end{equation*}
This induces a homomorphism $T\to \pdual{T}$ with kernel given by a finite group $A$.
This finite group inherits the quadratic form associated to $\lr{\, , \,}$, and therefore
becomes a nondegenerate metric group $\left(A , q\right)$ as in \Cref{sec:metric_groups}.
See \cite[\S 9.3]{FHLT} for the role played by this finite group in identifying the
framing anomaly of Chern-Simons.

The main upshot of this, for us, is that the braided fusion category $\cB =
\Vect\left[A\right]$ with convolution and braiding from $q$ (see
\Cref{sec:metric_groups}) generates the same $1$-$2$-$3$ theory as the 
Reshetikhin-Turaev $1$-$2$-$3$ theory attached to the category of
representations of the quantum group for $T$ at level $\lam$ \cite{FHLT}.

The problem of extending Chern-Simons theory to the point has received considerable
attention \cite{FHLT,Hen:CS,FT:gapped}.
One way to extend this Reshetikhin-Turaev $1$-$2$-$3$ theory to the point is by asking for 
a fusion category $\cC$ such that $\cB \cong \cZ\left(\cC\right)$ as braided categories.
Given that one such fusion category $\cC$ exists, one might also wonder how uniquely it is
determined. 
It is shown, in \cite{ENO:weakly}, that two fusion categories are Morita equivalent if and
only if the Drinfeld centers are equivalent as braided categories. 
This is strengthened in \cite{ENO:homotopy} to an equivalence of $2$-groups:
\begin{equation*}
\tau_{\leq 1} \Aut_{\Fus}\left(\cC\right) \simeq
\Aut_{\EqBr}\left(\cZ\left(\cC\right)\right) \ .
\end{equation*}
I.e. the isomorphism class of $\cC$ in $\Fus$ is almost determined by the isomorphism class of 
$\cZ\left(\cC\right)$ in $\BrFus$:
There is a lack of coherence at the top level,
captured by an obstruction class originally studied in \cite{ENO:homotopy} and surveyed in
\Cref{sec:O4}.
This same lack of coherence is captured by the anomaly theory studied in
\Cref{sec:anom_O4}.

\subsection{Anomalies in the Langlands program}
\label{sec:langlands}

An analogous anomaly to the one studied here appears in the Langlands program \cite{AV:mp,BZSV}.
Indeed, this was the main source of motivation for the project.
More specifically, upgrading \emph{Rozansky-Witten theory} \cite{RW,RobW,KRS,KR}
to a boundary/relative theory for certain $4$-dimensional gauge theories would produce a
version of the relevant relative theory on the $B$-side (i.e. spectral side). 
The author hopes to return to these questions in the future.

\subsection{Gapped systems}

It is well-known that gapped phases of matter are described at low energy/long range
by (unitary) TQFTs \cite{Fr:sre, FH:reflection, FT:gapped}.
See also \cite{RW:QC,KZ:top_orders}.
In particular, it is pointed out in \cite[Interlude (P.12)]{F:what_is}, that the 
low energy linear theory describing the system is in fact 
\emph{not} topological (it has a metric dependence \cite{Wit:RT}), however its
projectivization is indeed topological. 
In particular, this means \emph{projective} $3$-dimensional TQFTs model gapped systems in
$\left(2+1\right)$-dimensions.
We discuss such theories in \Cref{sec:3d_proj}.

't Hooft anomalies for discrete internal symmetries
are particularly important for the study of SPT phases
\cite{Kap:SPT,KT:discrete_anomalies,KT:3discrete_anomalies,Mul:thesis,MS:anomalies}.

A related subject is the study of anyon systems.
One way to model anyon systems, spelled out e.g. in \cite{RW:QC}, is using unitary modular categories (UMC).
The question of whether or not the UMC is a Drinfeld center, 
which plays a big role in this paper, is also important in the study of anyon systems.
For example, \cite[Conjecture 4.2]{RW:QC} states that a UMC being realized as a
topological phase of matter is equivalent to being a Drinfeld center of a (unitary) fusion
category.

When the fusion category $\cC$ is the category of vector spaces graded by a finite group
$L$, possibly twisted by a cocycle, the corresponding $3$-dimensional theory is
Dijkgraaf-Witten theory \cite{DW:90}.
Kitaev introduced a Hamiltonian realization \cite{K:anyons}, and for $L = \ZZ / 2$ 
this is the famous toric code.

The structure of the $3$-type $B\Aut_\Fus\left(\cC\right) \simeq
B\Pic\left(\cZ\left(\cC\right)\right)$, as identified in \cite{ENO:homotopy}, is the
source of the projectivity/anomalies identified in \Cref{sec:3d_theories}.
The homotopy type of this groupoid is of much interest in the literature.
For example, let $\cB$ be an arbitrary unitary modular category.
It is conjectured in \cite{AWH:paths} that the homotopy type of $B\Pic\left(\cB\right)$
is related to a space of gapped Hamiltonians which give rise to the topological order
described by $\cB$.

\subsection{Sandwiches and anomalies}
\label{sec:sandwich_anomaly}

The discussion in \Cref{sec:intro_fus_anom} fits into a more general picture, which
relates to the language of topological symmetries \cite{FMT}.
An upshot of this discussion, in the context of \Cref{sec:intro_fus_anom}, 
will be that the anomalous theory in \Cref{thm:pin_anom} (resp. \Cref{cor:spin_anom}) 
will equivalently define a
$\sigma_{B\O\left(A , q\right)}$-module structure (in the sense of \cite{FMT}) twisted by
the class which classifies the $3$-group $\PPin$ (resp. $\SSpin$) over $\O\left(A ,
q\right)$.

Let $d\in \ZZ^{\geq 0}$, and let $\cT$ be an $\left(\infty , d+1\right)$-category with
duals. Given any TQFT
\begin{equation*}
F \colon \Bord_{d} \to \Om \cT 
\end{equation*}
we can consider the higher automorphism group of the assignment to the point:
$\Aut_{\Om \cT} \left(F\left(\pt\right)\right)$. 
Note that this is a $d$-group.

Assume that $B\Aut_{\Om \cT}\left(F\left(\pt\right)\right)$
fibers over some space $X$ with fiber $B^{d+1} \units{\bk}$, classified by some map $c$:
\begin{equation*}
\begin{cd}
B^{d} \units{\bk} \ar{r}\ar{d}&
B\Aut_{\Om \cT} \left(F\left(\pt\right)\right)
\ar{d} \ar{r}& *\ar{d}\\
* \ar{r}&
X
\ar{r}{c} & B^{d+1} \units{\bk}
\end{cd}
\end{equation*}
This assumption is the stand-in for \eqref{eqn:BAut_bundle}.
For ease of exposition, write:
\begin{equation*}
\widetilde{X} = B\Aut_{\Om \cT}\left(F\left(\pt\right)\right)
\ .
\end{equation*}

Tautologically we have a theory:
\begin{equation*}
\widetilde{F} \colon \Bord^{\widetilde{X}}_d \to \Om \cT
\ ,
\end{equation*}
classified by the inclusion of $\widetilde{X}$ into $\Om \cT$.
Now we ask if the theory $\widetilde{F}$ factors through/descends to a theory on $\Bord^X_d$:
\begin{equation}
\begin{tikzcd}
\Bord_{d}^{\widetilde{X}} \ar{r}\ar{d}&
\Om \cT \\
\Bord_{d}^X \ar[dashed]{ur} & 
\end{tikzcd}
\label{eqn:factor}
\end{equation}

We will now rephrase this question in terms of twisted topological symmetry as in
\cite{FMT}.
In order to do this, we will assume \Cref{hyp:sigma}, which asserts that there is a
suitable association of TQFTs to $\pi$-finite groupoids. 

We will now assume that $X$ is $\pi$-finite. 
Recall we are assuming $\Om^{d+1 } \cT \simeq B^{d+1} \units{\bk}$. Then we can consider 
the composition of the cocycle $c$ with the inclusion of $B^{d+1}\units{\bk}$
into $\cT$:
\begin{equation*}
X\lto{c}B^{d+1}\units{\bk} \to \cT \ ,
\end{equation*}
which classifies an invertible theory 
\begin{equation}
\Bord^X_{d+1} \to \cT \ .
\label{eqn:pre_alpha}
\end{equation}
If we truncate the invertible theory \eqref{eqn:pre_alpha}, we obtain a once-categorified
invertible theory:
\begin{equation*}
\al \colon \Bord_{d+1}^{X} \to \cT
\ .
\end{equation*}
As explained in \cite[\S A.2]{FMT}, we regard the map from any given bordism to $X$ as 
a fluctuating field, which can be integrated over to obtain a new 
$\left(d+1\right)$-dimensional theory $\sigma_{X , c}^{d+1}$.

Now the upshot is that a module structure (in the sense of \cite{FMT}, see
\Cref{sec:top_symm}) on the original theory $F$ 
over $\sigma_{X , c}^{d+1}$ is equivalent to a projective
TQFT with projectivity $\al$.

Formally, in this paper, a relative theory is a lax natural transformation in the sense of 
\cite{JFS}.
I.e. it is a functor into the arrow category of $\cT$.
A \emph{projective TQFT} is meant to be a TQFT valued in the subcategory of the (op)lax
arrow category of $\cT$ which consists of arrows between the unit and an invertible
object. 
Then by \cite[Theorem 7.15]{JFS}, projective theories are equivalent to anomalous ones,
and a linearization of the projectivity is precisely a trivialization of the anomaly.

The following, stated as \Cref{thm:anom_symm} 
in the body of this paper, summarizes this discussion.

\begin{thm*}
An $X$-anomaly $F_c \colon \al_c \to 1_X$ of $F$ (as in \Cref{defn:X_anom}) naturally 
defines:
\begin{enumerate}[label = (\roman*)]
\item A projective $X$-theory $\overline{F} \colon \Bord_d^X \to \PP\Om \cT$
with projectivity $\al_c$ and underlying theory $F$ (as in \Cref{defn:projective_TQFT}).

\item $\widetilde{F} \colon \Bord^{\widetilde{X}}_d \to \Om \cT$ with underlying framed
theory, written $F$.

\item Assuming \Cref{hyp:sigma}, a $\left(\sigma_{X , c} , \rho_{X , c}\right)$-module structure 
on $F$ (\Cref{sec:mod_str}).
\end{enumerate}
\end{thm*}

\begin{exm}
We have discussed this formalism in general, but it is most interesting when the
space $X$ has a novel interpretation.
E.g. in \Cref{sec:anom_k3,sec:anom_O4} we use the canonical identification of the
truncation to a $2$-type with the braided autoequivalences of the center
\cite{ENO:homotopy} to interpret this as an anomaly of $F$ as a theory which is
equivariant with respect to the braided autoequivalences of the center.

Slightly more generally, consider any nondegenerate braided fusion category $\cA$. 
There is a fibration:
\begin{equation*}
\begin{tikzcd}
B^3 \units{\bk} \ar{r}\ar{d} &
B\units{\left(\lMod{A}\right)} \ar{d}{\tau_{\leq 2}}\\
* \ar{r}&
B\Aut_{\EqBr}\left(\cA\right)
\end{tikzcd}
\end{equation*}
by \cite[Theorem 5.2]{ENO:homotopy}, which is classified by some map 
$c \colon B\Aut_{\EqBr}\left(\cA\right) \to B^4 \units{\bk}$. 

For any monoidal functor $\rho \colon G \to \Aut_{\EqBr}\left(\cA\right)$, $\WRT_\cA$
obtains a $\left(\sigma^4_{BG , \rho^* c} , \rho\right)$-module structure on $\WRT_\cA$.
The question of trivializing the anomaly is the question of
$\sigma_{BG , \rho^* c}^4$ having a Neumann boundary theory, which is equivalent to
whether or not $c$ pulls back to define the trivial cohomology class on $BG$.

So there are many different four-dimensional theories in the story:
the untwisted pure topological $G$-gauge theory
(which we \emph{want} to be defined relative to, in order to gauge),
and then the Crane-Yetter theory and the twisted topological $G$-gauge theory
(which we are \emph{always} on the boundary of).

One way to think of the universal such theory $\sigma_{B\Aut_{\EqBr},c }$ is as 
the result of gauging the canonical anomaly-free $G$-action on $\CY_\cA$ defined by
$\rho$.
\end{exm}

After trivializing the projectivity, we have the following compatibility, stated as 
\Cref{thm:triv_anom_symm} in the body of the paper:
Let $f \colon Y\to X$ be a map of $\pi$-finite  spaces.
A trivialization of the class $f^* c$ (i.e. splitting of $\widetilde{X}$ over $Y$) determines 
\begin{enumerate}[label = (\roman*)]
\item a factorization of \eqref{eqn:factor} pulled back along $f$,
\item a trivialization $1\lto{\sim} \al_c$ of the anomaly theory, and 
\item a linearization of the projective theory (as in \Cref{defn:projective_TQFT}), 
\item a reduction (\Cref{defn:reduction})
from the $\left(\sigma_{X , c} , \rho\right)$-module structure to a $\left(\sigma_X ,
\rho\right)$-module structure.
\end{enumerate}
These all determine theories defined on $\Bord^Y_d$, which agree.

\subsection{Gauging and anomalies}
\label{sec:gauging}

Given a quantum field theory $F$ and a (compact Lie) group $G$, 
under certain conditions, 
a new theory, the \emph{$G$-gauged theory}, can be produced via a procedure known as
\emph{gauging}.
This has received various mathematical formulations, e.g. for topological theories see
\cite[\S 2.3]{Tel:ICM}. Also see \cite[\S 3.4]{FMT}. For non-extended versions of some of
the $3$-dimensional theories considered in this paper, see \cite{Mul:thesis,MS:anomalies}.
This procedure cannot always be carried out, and the obstructions 
are known as \emph{'t Hooft anomalies} \cite{Hooft:anom}.

Gauging is often phrased as a two step process: 
\begin{enumerate}
\item couple $F$ to a background $G$-field (principal $G$-bundle), 
\item integrate over all $G$-bundles to obtain a new ordinary theory: the gauged theory.
\end{enumerate}
In the context of topological field theories, 
the first step is to extend\footnote{There is always a trivial extension
of $F$ to $\Bord_{d}^{BG}$. The pullback of any extension along 
$\Bord_{d} \to \Bord_{d}^{BG}$ (sending every bordism to the same bordism equipped with
the product principal $G$-bundle) is an ordinary theory with internal $G$-symmetry.
Usually one has a given (internal) action of $G$ on $F$, and one insists that the 
extension agrees with this action after pulling back.}
$F$ from $\Bord_{d}^\fra$ (or e.g. $\Bord_{d}^{\ori}$) to $\Bord_{d}^{BG}$.

\begin{rmk}
Even before proceeding to the second step, one can see how anomalies
of the sort defined in \Cref{sec:proj_anom} can act as obstructions to gauging.
One should imagine that we cannot quite compatibly extend $F$ to $\Bord^{BG}_{d}$: We can only 
extend it as a relative theory. I.e. we have a (once-categorified) theory $\al$ (which
\emph{is} defined on $\Bord^{BG}_{d}$)
and a relative theory $F \colon \al \to 1$.
This becomes a TQFT defined on $\Bord^{BG}_{d}$
only once we provide a trivialization $1\lto{\sim} \al$.
\end{rmk}

As in \Cref{sec:sandwich_anomaly}, a theory defined on $\Bord_d^{BG}$ defines
a boundary theory $\widetilde{F}\colon 1 \to \sigma_{BG}^{d+1}$ .
(see \Cref{sec:sigma} for notation, and 
\Cref{sec:proj_anom} for a more detailed discussion). 

The second step is easier to formulate in this language. 
Namely, as in \cite[Definition 3,22, Example 3.24]{FMT}, 
the process of integrating over all $G$-bundles is pairing with the \emph{Neumann boundary
theory}:
\begin{equation}
F /_\e \sigma \ceqq \e \tp_{\sigma_{BG}^d} \widetilde{F} \ ,
\label{eqn:gauge_symm}
\end{equation}
where the Neumann boundary theory $\e$ is the morphism induced (as in
\Cref{prop:corr_mor}) by the augmentation:
\begin{equation*}
\begin{cd}
& BG \ar{dr}\ar{dl} & \\
BG && \pt 
\end{cd}
\end{equation*}

Note that if the theory $\sigma_{BG}^{d+1}$ is twisted by a cocycle $\tau$, then we will not
necessarily have an augmentation map: We need to trivialize $\tau$.
I.e. if we have a boundary theory
\begin{equation*}
\widetilde{F} \colon  1 \to \sigma_{BG , \tau}^{d+1} \ ,
\end{equation*}
in order to gauge the $G$-action, we need to 
pair with an augmentation for $BG$. One obtains an augmentation from any correspondence of
the form
\begin{equation*}
\begin{cd}
& \left(BH , \mu \right) \ar{dl}\ar{dr} & \\
\left(BG , \tau\right) && \pt
\end{cd}
\end{equation*}
where $H\subset G$ is a subgroup, and $\mu$ is a trivialization of $\tau\vert_H$.
This is, for example, the form of the classification of simple, fully extended
$2$-dimensional topological theories relative to gauge theory in \cite{FT:Ising}:
subgroups equipped with central extensions.

As is explained in \Cref{sec:sandwich_anomaly}, the cocycle $\tau$ classifies an anomaly
theory $\al_\tau$. 
So a trivialization of $\al_\tau$ (i.e. of $\tau$) determines an augmentation, which induces
the Neumann boundary condition, which allows us to gauge as in \eqref{eqn:gauge_symm}.
In other words:
$\al_\tau$ obstructs gauging.

\section{Projective TQFTs}
\label{sec:projective_TQFTs}

\subsection{Preliminaries: arrow categories}

Recall the arrow category $\cT^\down$ defined in \cite{JFS}.
An anomalous theory $\al \to 1$ is, by definition (see \Cref{rmk:lax}), a functor
\begin{equation}
\overline{Z}\colon \Bord_{d} \to \cT^{\down}
\label{eqn:barZ}
\end{equation}
such that
\begin{align*}
s \comp \overline{Z} = \al
&&
t\comp \overline{Z} = 1
\ .
\end{align*}

\begin{rmk}
As explained in \Cref{rmk:lax}, the relative theories in this paper are lax natural
transformations. 
Everything can be repeated to produce an oplax version by replacing
$\cT^\down$ with $\cT^\to$ everywhere.
\end{rmk}

Recall that \cite{JFS} define the arrow category `terminating at the unit' to be:
\begin{equation*}
\cT^{\down 1} \ceqq 
\cT^{\down} 
\htimes{\cT} 
\cT\left[0\right] 
\qquad 
\qquad
\qquad
\begin{tikzcd}
\cT^{1\down} 
\ar{r}\ar{d}&
\cT\left[0\right]\ar{d} \\
\cT^{\down} \ar{r}{s}&
\cT 
\end{tikzcd}
\ .
\end{equation*}
Similarly, we will define the subcategory `originating at an invertible object' to be:
\begin{equation*}
\cT^{\times \down} \ceqq 
\cT^{\times}
\htimes{\cT} 
\cT^\down 
\qquad
\qquad
\qquad
\begin{tikzcd}
\cT^{\down \times} 
\ar{r}\ar{d}
&
\cT^{\down} \ar{d}{t} \\
\cT^{\times} \ar{r}&
\cT 
\end{tikzcd}
\ .
\end{equation*}

\subsection{The projectivization of a target category}
\label{sec:P}

The projectivization of the category $\Om \cT$, written $\PP\Om \cT$, consists of arrows
in $\cT$ between an invertible object of $\cT$ and the unit in $\cT$.

\begin{defn}
Define the \emph{projectivization} of a symmetric monoidal $\left(\infty ,
n\right)$-category 
(thought of as an $\EE_\infty$-algebra among complete $n$-fold Segal spaces) $\cT$ to be:
\begin{equation*}
\PP\left(\Om \cT \right) 
\ceqq  
\cT\left[0\right] \htimes{\cT} \cT^{\down} \htimes{\cT} \cT^\times
\end{equation*}
Recall from \cite[Def. 6.10]{JFS} that $\cT\left[0\right]\simeq *$.
\label{def:P}
\end{defn}

\begin{wrn}
This construction does \emph{not} only depend on $\Om \cT$, as the notation might suggest.
It depends on the group of isomorphism classes of invertible objects of $\cT$.
For example, many categories are of the form $\Om \cT_1 \simeq \Om \cT_2$ where $\pi_0
\units{\cT_1}$ is trivial and $\pi_0\units{\cT_2}$ is not.

This becomes especially important when we study $3$-dimensional projective targets in
\Cref{sec:PFus}. See \Cref{rmk:PFus}. 
\label{wrn:P}
\end{wrn}

\begin{theorem}
There is a well-defined functor from the $\infty$-groupoid of symmetric-monoidal
$\left(\infty , n\right)$-categories (and equivalences) to itself,
\begin{equation*}
\PP \colon 
\EE_\infty\left(
n\text{-}\cat{cat}
\right)
\to
\EE_\infty\left(
n\text{-}\cat{cat}
\right)
\end{equation*}
which sends a symmetric-monoidal $\left(\infty , n\right)$-category $\cT$ to $\PP \cT$, 
thought of as $\EE_\infty$-algebras among complete $n$-fold Segal spaces.

Furthermore, projectivization commutes with iterated looping, i.e. the following diagram
commutes for all integers $k\in \ZZ$ satisfying $1\leq k \leq n$:
\begin{equation}
\begin{tikzcd}
\EE_\infty\left(
n\text{-}\cat{cat}
\right)
\ar{d}{\Om^{k}}
\ar{r}{\PP}
&
\EE_\infty\left(
n\text{-}\cat{cat}
\right)
\ar{d}{\Om^{k}}
\\
\EE_\infty\left(
\left(n-k\right)\text{-}\cat{cat}
\right)
\ar{r}{\PP}
&
\EE_\infty\left(
\left(n-k\right)\text{-}\cat{cat}
\right)
\end{tikzcd}
\end{equation}
\label{thm:P}
\end{theorem}
\begin{proof}
The functor $\left(-\right)^\down$ is defined in 
\cite[Definition 5.14]{JFS}, and shown to depend naturally on both arguments in
\cite[Corollary 5.19]{JFS}, which is extended to the symmetric monoidal case in
\cite[Corollary 6.9]{JFS}. 
Rephrasing this slightly, this says that $\left(-\right)^\down$ defines a functor from the
$\infty$-groupoid of symmetric monoidal $\left(\infty , n\right)$-categories and
equivalences to itself.

There is a commuting diagram of three homotopy fiber squares:
\begin{equation}
\begin{tikzcd}
\PP \Om \cT \ar{r}\ar{d}&
\cT^{1\down}\ar{d} \ar{r}{s}&
\cT\left[0\right]\ar{d} \\
\cT^{\down\times} \ar{r} \ar{d}{t}& 
\cT^{\down} \ar{r}{s}\ar{d}{t}&
\cT \\
\cT^{\times} \ar{r}&
\cT  &
\end{tikzcd}
\label{eqn:P_diagram}
\end{equation}
because 
\begin{equation*}
\left(\cT^{1\down}\right) \htimes{\cT^\down} \left(\cT^{\down\times}\right)
\simeq
\cT\left[0\right] \htimes{\cT} \cT^{\down} \htimes{\cT} \cT^\times \ .
\end{equation*}
Therefore the functoriality of $\PP$ follows from the functoriality of
$\left(-\right)^\down$ \cite{JFS} and the functoriality of $\left(-\right)^\times$.

For fixed non-negative integer $n$, we will first show the statement for $k = 1$:
\begin{equation}
\Om \PP \Om \left(-\right)  \simeq \PP \Om^2 \left(-\right)  \ .
\label{eqn:base}
\end{equation}
Recall that, by definition:
\begin{equation*}
\Om \PP \Om \cT  = \Om\left(
\cT\left[0\right] \htimes{\cT} \cT^{\down} \htimes{\cT} \cT^\times
\right)
\end{equation*}
and
\begin{equation*}
\PP \Om^2 \cT  = 
\Om\cT\left[0\right] \htimes{\Om\cT} \left(\Om \cT\right)^{\down} \htimes{\Om\cT} 
\left(\Om \cT\right)^\times \ .
\end{equation*}
There is a canonical equivalence of 
symmetric monoidal complete $n$-fold Segal spaces
(\cite[Proposition 6.12]{JFS}):
\begin{equation*}
\Om \cT \simeq \cT\left[0\right] \htimes{\cT} \cT^{\down} \htimes{\cT} 
\cT\left[0\right] \ ,
\end{equation*}
which means that \eqref{eqn:base} 
follows from the fact that homotopy colimits commute.

To see the statement for general $k\in \left\{1 , \ldots , n\right\}$,
apply \eqref{eqn:base} to $\Om^{k-1} \cT$.
\end{proof}

\begin{rmk}
The $\PP$ in \Cref{def:P} can be thought of as standing for ``projective'', as
$\PP\Vect$ should be thought of as some version of the category of projective spaces.
Therefore it can also be thought of as standing for ``pure states'', as these form the
projectivization of the Hilbert space of mixed states. 
\end{rmk}

As is conjectured by Freed in \cite{F:what_is}, the projectivization should have a
description as a directed pullback of symmetric monoidal $\left(\infty ,
n\right)$-categories
\begin{equation}
\begin{tikzcd}
\PP \Om \cT \ar[dashed]{r}{F} \ar[dashed]{d}&
\units{\cT} \ar{d} \\
* \ar{r}&
\cT
\end{tikzcd}
\label{eqn:P_con}
\end{equation}
One way to interpret \eqref{eqn:P_con}, is as expressing ``exactness'' at a term of an 
``exact sequence'' of symmetric-monoidal $\infty$-categories.
We do not attempt to make this notion precise here, however consider the following diagram
of functors:
\begin{equation}
\begin{cd}
\cdots &
\units{\Om \cT} \ar{r}&
\Om \cT \ar{r}
\arrow[d, phantom, ""{coordinate, name=Z}]&
\PP\Om \cT
\arrow[dll,swap,
rounded corners,
to path={ -- ([xshift=2ex]\tikztostart.east)
|- (Z) [near end]\tikztonodes
-| ([xshift=-2ex]\tikztotarget.west)
-- (\tikztotarget)}]
&
\\
&\units{\cT} \ar{r}&
\cT \ar{r}&
\PP \cT & \cdots
\end{cd}
\end{equation}
where the connecting homomorphism is $F$ from \eqref{eqn:P_con}.
The functors $\Om \cT \to \PP \Om \cT$ and $\cT \to \PP \cT$ in \eqref{eqn:P_con} are the
canonically defined functors which send an object to the same object regarded as an
endomorphism of the unit.
We hope to return to this in future work. 

\begin{rmk}
Note that, if every invertible object of $\cT$ is trivializable, then the functor $\Om \cT \to \PP
\Om \cT$ is essentially surjective. 
Otherwise, there are objects in $\PP \Om \cT$ outside of the image of $\Om \cT$: 
for example, if $\cT$ consists of algebra objects in $\Om \cT$, consider any nontrivial
invertible object of $\cT$. Then the regular module defines an object of $\PP \Om \cT$
which is not in the image of $\Om \cT$. 
\label{rmk:ess_surj}
\end{rmk}

\subsection{Fully-extended projective target categories}
\label{sec:P_exm}

\subsubsection{Projective Morita categories}

As discussed in \Cref{wrn:P}, the construction of $\PP\left(\Om \cT\right)$ actually
depends on $\cT$, rather than merely $\Om \cT$. 
However, if we are given a symmetric monoidal
($\tp$-sifted-cocomplete) $\left(\infty , n\right)$-category $\cS$, we can form a
projectivization of $\cS$ as follows.
If $\cS$ is sufficiently nice\footnote{Specifically we need $\cS$ to be
$\tp$-sifted-cocomplete. See \Cref{rmk:morita} for a discussion of different models of the
higher Morita category.} then we can form the 
symmetric monoidal $\left(\infty ,n+1\right)$-category 
$\cT = \Alg\left(\cS\right)$, the ``even higher'' Morita category of 
$\EE_1$-algebras \cite{JFS}.

Recall that $\Om \cT \simeq \cS$, so in this case we obtain a motto about projective objects:
\textit{An object of $\PP \left(\cS\right)$ can be thought of as an invertible (e.g. Azumaya)
algebra in $\cS$ along with a module over this algebra (also internal to $\cS$).}

More generally, we obtain a projective target category by taking the projectivization of
thee Morita $\left(\infty , n+k\right)$-category $\Alg_n\left(\cS\right)$ for a
(sufficiently nice) symmetric monoidal $\left(\infty , k\right)$-category $\cS$.

\begin{rmk}[Models for Morita categories]
The higher Morita $\left(\infty , n\right)$-category of $\EE_n$-algebras in a 
$\tp$-sifted-cocomplete (resp. $\tp$-GR-cocomplete)
symmetric monoidal $\left(\infty , k\right)$-category $\cS$ was constructed in \cite{S:En}
(resp. \cite{H:En}).
This was extended to the ``even higher'' Morita $\left(\infty , n+k\right)$-category,
written $\Alg_n \cS$, in \cite{JFS}.

There are expected relationships between the models in \cite{S:En} and \cite{H:En}, but
the author is unaware of a theorem relating them. 
One benefit of the former is that $n$-dualizability is known
\cite[Theorem 5.1]{GS}\footnote{Note however that the only $\left(n+1\right)$-dualizable object is
the unit \cite[Theorem 6.1]{GS}.},
but the arguments do not translate to the model in
\cite{H:En}.

On the other hand, the results used in \Cref{sec:dual_inv} concerning the dualizability
and invertibility of (braided) fusion categories are from
\cite{BJS:dualizable,BJSS:invertible}, which use the
model from \cite{H:En}. 
\label{rmk:morita}
\end{rmk}

\subsubsection{The projectivization of the category of vector spaces}
\label{sec:PVect}

Let $\cS = \Vect$ be the symmetric monoidal category of finite-dimensional complex
vector spaces. Setting $\cT = \Alg$ to be the symmetric monoidal Morita $2$-category of
associative algebras, we can calculate $\PP\Vect$ to be the $2$-category given by ``all
modules over all Azumaya algebras''. The $1$-morphisms are (op)lax squares
\begin{equation*}
\begin{tikzcd}
1\ar[equal]{r}\ar[swap]{d}{ {}_A M}&
1\ar{d}{ {}_B N}\\
A \ar{r}{\sim}\ar[Rightarrow]{ur}{\phi} &
B
\end{tikzcd}
\end{equation*}
where $A$ and $B$ are Morita-invertible algebras with modules $M$ and $N$. The data of the
filling $\phi$ is a linear map between $M$ and $N$ which intertwines the module structures
according to an invertible Morita morphism $A\lto{\sim} B$.
The following are examples of objects of $\PP \Vect$. 
\begin{itemize}
\item Any fully-dualizable (i.e. separable) algebra as a module over itself.
\item If there is an algebra map to the trivial algebra (an augmentation map) then this
defines a module structure on the one-dimensional vector space. 
\item The rank $n$ vector space as a module over the matrix algebra $M_n$.
\end{itemize}

\begin{rmk}
Note that $\PP \Vect$ serves as a target for non-extended projective TQFTs in any
dimension. 
The category $\PP \Vect$ is used in \cite{Ben:WRT_CY} to formalize the notion of anomalous 
non-extended non-semisimple Witten-Reshetikhin-Turaev theory.
\end{rmk}

\begin{rmk}
There is a related $2$-category, written $\cat{Proj}$, which is defined in \cite[Appendix
A]{F:what_is}.
The objects are finite-dimensional vector spaces, the $1$-morphisms are linear maps, and
there is a $2$-morphism between any two linear maps $\phi, \psi \colon V \to W$ for every
$\lam \in \units{\CC}$ such that 
$\phi\left(v\right) = \lam \psi\left(v\right)$ for all $v\in V$.

There is an equivalence $\cat{Proj} \to \PP \Vect$, e.g. as symmetric-monoidal
bicategories (\cite[Remark 7.2]{Ben:WRT_CY}), which sends a vector space $V$ to itself
regarded as an endomorphism of the unit in $\Alg$, i.e. as a $\CC$-bimodule.
\label{rmk:PVect}
\end{rmk}

\subsubsection{The projectivization of the Morita category of monoidal categories}
\label{sec:PFus}

Let $\cT = \BrFus$ be the Morita $4$-category of braided fusion
categories (see \Cref{sec:fusion_prelim}) and recall:
\begin{equation*}
\Fus \simeq \Om \cT = \End_\cT \left(1\right)
\ .
\end{equation*}
Applying $\PP$ from \Cref{def:P} to $\BrFus$, we obtain the following.

\begin{cor}
$\PP\Fus$ is a well-defined symmetric monoidal $4$-category. 
The objects of $\PP\Fus$ are arrows in $\BrFus$ which are between an invertible object and
the unit (i.e. central modules over invertible braided fusion categories).
The $1$-morphisms in $\PP \Fus$ are (op)lax squares:
\begin{equation*}
\begin{tikzcd}
1 
\ar[equal]{r}
\ar[swap]{d}{{}_{\cA} \cM }
&
1
\ar{d}{{}_{\cB} \cN }
\\
\cA
\ar[Rightarrow]{ur}{\Phi}
\ar{r}{\sim}
&
\cB
\end{tikzcd}
\end{equation*}
where $\cA$ and $\cB$ are invertible objects of $\BrFus$, $\cM$ and $\cN$ are (possibly
non-invertible) $1$-morphisms in $\BrFus$, and the filling $\Phi$ is a $2$-morphisms in
$\BrFus$, i.e. a `twisted' intertwining bimodule between $\cM$ and $\cN$.
The higher morphisms are defined by general principles in \cite{JFS}. 
\label{cor:PFus}
\end{cor}

\begin{rmk}
Recall from \Cref{wrn:P} that the construction of $\PP\Om \cT$ actually depends on the
group of isomorphism classes of invertible objects of $\cT$, rather than just $\Om \cT$.
In low dimensions this didn't appear, unless we were considering super-vector spaces,
because the category of finite-dimensional vector spaces, the Morita category of algebras,
and the Morita category of fusion categories all have trivial group of isomorphism classes
of invertible objects, i.e. all invertible objects are trivializable. 

This is not true in this dimension: The Morita $4$-category $\BrFus$ has a highly
nontrivial group $\pi_0 \units{\BrFus}$, which is known to be equivalent to the so-called
Witt group of braided fusion categories 
\cite[Theorem 4.2]{BJSS:invertible}.
Note that this contains the Witt group of $\QQ$, and is therefore certainly nontrivial. 
This clearly produces a different projectivization in comparison with performing the same
construction, say, in the arrow category of the $4$-category $B\Fus$, which would clearly
not detect any of the non-trivializable invertible objects of $\BrFus$.
\label{rmk:PFus}
\end{rmk}

\subsection{Projective theories}
\label{sec:proj}

\begin{defn}
A $d$-dimensional \emph{projective $\left(X ,
\z\right)$-TQFT} is a non-zero symmetric-monoidal functor:
\begin{equation*}
\overline{F} \colon \Bord^{\left(X , \z\right)}_{d} \to \PP\Om \cT
\ .
\end{equation*}
Write $s \colon \PP \Om \cT \to \units{\cT}$ for the composition of the vertical functors
on the left in \eqref{eqn:P_diagram}.
Given such a theory $\overline{F}$, composing with the functor $s$
results in an invertible theory:
\begin{equation*}
\al \colon \Bord^{\left(X , \z\right)}_{d} \lto{\overline{F}}
\PP\Om \cT \lto{s}
\units{\cT}
\ .
\end{equation*}
This theory is the \emph{projectivity} of the projective theory $\overline{F}$.
Sometimes to emphasize that a theory is \emph{not} projective, we will call it
\emph{linear}.
\label{defn:projective_TQFT}
\end{defn}

\begin{prop}
A trivialization of the projectivity $\al = s\comp\overline{F}$ 
determines a linear theory $F_X\colon\Bord_d^X \to \Om \cT$.
We will call $F_X$ a \emph{linearization} of $\overline{F}$.
\label{prop:linearize}
\end{prop}
\begin{proof}
A projective TQFT is equivalently a natural transformation from an invertible TQFT to the
trivial TQFT. Composing this with the trivialization results in an endomorphism of the
trivial theory valued in $\cT$ which, by \Cref{prop:1_X}, is equivalent to a TQFT valued
in $\Om \cT$.
\end{proof}

\begin{defn}
Assume a projective TQFT $\overline{F}$ has a linearization upon restriction to framed
bordisms:
\begin{equation}
\begin{tikzcd}
\Bord_d^\fra \ar[hook]{r}\ar{dr}{\overline{F}^\fra}
\ar{d}{F^\fra}
&
\Bord_d^X \ar{d}{\overline{F}}\ar{dr}{\al}&
\\
\Om \cT \ar{r}&
\PP \Om \cT \ar{r}&
\units{\cT} 
\end{tikzcd}
\end{equation}
Then we call $F^\fra$ the \emph{underlying framed theory of $\overline{F}$}.
\label{defn:underlying}
\end{defn}

\begin{rmk}
A previous version of this preprint included ``Hypothesis P'' which asserted the existence
of a functor $\PP$ as in \Cref{def:P} satisfying two conditions, the first of which is
guaranteed by the diagram \eqref{eqn:P_diagram}, the second of which is
\Cref{prop:proj_anom}.
\end{rmk}

\subsection{Extensions of the bordism category}
\label{sec:ext_bord}

Just as in classical representation theory, a projective representation is equivalent to a
linear representation of an extension.
The notion of an extension of a bordism category by a modular functor already appeared in
Segal's original paper \cite[Definition 5.2]{S:CFT}.

The projectivity of a projective TQFT in the sense of \Cref{defn:projective_TQFT} 
is also captured by an extension of a bordism category as follows.
Namely, the extension is the universal object 
over which the projectivity can be trivialized. In slightly more detail, 
it is the pullback:
\begin{equation}
\begin{tikzcd}
\widetilde{\Bord_d^X} \ar{r}\ar{d}{\widetilde{F}}&
\Bord_d^X \ar{d}{\overline{F}}\ar{dr}{\al}&
\\
\Om \cT \ar{r}&
\PP \Om \cT \ar{r}&
\units{\cT} 
\end{tikzcd}
\label{eqn:extension}
\end{equation}
and since the composition $\Om \cT \to \PP \Om \cT \to \units{\cT}$ is canonically
trivialized, so is the anomaly of $\widetilde{F}$, i.e. the extension is the universal
object over which the anomaly is trivializable.

This description as a pullback of symmetric monoidal $\infty$-categories is not
completely precise. A more formal description is as follows:
Given a projective TQFT $\overline{F} \colon \Bord_d^{X} \to \PP \Om
\cT$, the anomaly theory $\al = \overline{F} \comp s$ is an invertible TQFT, and therefore
classifies a map of groupoids $\widetilde{\al}$ \eqref{eqn:invertible} as explained in
\Cref{sec:invertible}.
The homotopy fiber $\cat{E}$ of $\widetilde{\al}$ is some extension of the (groupoid given by
invertible all morphisms in the) original bordism category:
\begin{equation*}
\begin{tikzcd}
\Om \cT^\times \ar{r}\ar{d}&
\cat{E} \ar{r}\ar{d}&
\pt \ar{d} \\
\pt \ar{r}&
\abs{\Bord_d^X} \ar{r}{\al}&
\units{\cT}
\end{tikzcd}
\end{equation*}
and the claim is that $\cat{E}$ is the result of inverting all morphisms in some extension
of the original bordism category.

\begin{rmk}
Recall from \Cref{rmk:ess_surj} that the map $\Om \cT \to \PP \Om \cT$ might not be
essentially surjective, and therefore the map from $\widetilde{\Bord^X_d}$ to $\Bord^X_d$
might not be essentially surjective either. 

One instance of this appearing in examples is the filled bordism category, where bordisms
(of any codimension) are required to bound a bordism of one higher dimension. For example,
the point is not an object of the fully-extended filled bordism category, but $S^0$ is. 
\label{rmk:filled}
\end{rmk}

The bordism category $\Bord^{\widetilde{X}}_{d+1}$ provides a description for (a
fully-extended version of) the extension of the source discussed in
\cite[\S 3]{F:what_is}. See also \Cref{rmk:resolve_anom}.

\section{Anomalies}
\label{sec:anomalies}

See \cite{F:what_is} and the references therein for more details on anomalies.

\subsection{Definitions}

Fix the target $\cT$ from \Cref{sec:cob_hyp}.
Let $\left(X , \z\right)$ be as in \Cref{sec:X_theories}, and recall we have defined
$\left(X , \z\right)$-theories as functors out of $\Bord^{\left(X , \z\right)}$.

\begin{defn*}
An \emph{anomaly} $\al$ is an invertible, once-categorified TQFT:
\begin{equation*}
\al \colon \Bord_d^{\left(X , \z\right)} \to \cT \ .
\end{equation*}
A \emph{$d$-dimensional anomalous TQFT $F$ with anomaly $\al$} is a TQFT defined relative
to $\al$:
\begin{equation*}
F_\al \colon \al \to 1 \ ,
\end{equation*}
where $1$ denotes the trivial theory on $\Bord^{X,\z}_d$.
\end{defn*}

A \emph{trivialization of the anomaly $\al$} is an equivalence $1 \lto{\sim} \al$. 
The \emph{trivialized theory} is the composition:
\begin{equation*}
1 \lto{\sim} \al \lto{F} 1 
\end{equation*}
which defines a $\left(X , \z\right)$-TQFT of dimension $d$ by \Cref{prop:1_X}.

\begin{rmk}
Sometimes anomalies are defined to be boundary theories rather than relative theories 
(see \Cref{rmk:boundary}). 

Sometimes anomalies are defined as left boundary/relative theories $1\to \al$, rather than 
right boundary/relative theories: $\al \to 1$.
In this case trivializations would be $\al\lto{\sim} 1$. 
\label{rmk:alt_anom}
\end{rmk}

Given an $\left(X , \z\right)$-theory $Z$, write $Z_\fra$ for the underlying framed
theory given by restricting $Z$ to bordisms with trivial $\left(X ,
\z\right)$-structure. 

\begin{defn}
Consider two tangential structures $\left(X_0 , \z_0\right)$ and $\left(X , \z\right)$
such that there is a map $X \to X_0$ compatible with $\z_0$ and $\z$. 
We say a TQFT:
\begin{equation*}
F \colon \Bord_{d}^{X_0} \to \Om\cT
\end{equation*}
\emph{has an anomaly as an $\left(X , \z\right)$-theory} 
if there is an anomaly theory
\begin{equation*}
\al \colon\Bord_{d}^{\left(X , \z\right)} \to \cT 
\end{equation*}
and a trivialization $t \colon 1_\fra \lto{\sim}\al_\fra$
such that there exists an $\left(\al , t\right)$-module structure on $F$
(as in \Cref{sec:mod_str}).
\label{defn:anom_of_F}
\end{defn}

Spelling out \Cref{defn:anom_of_F}, we see that an anomaly of $F$ as an $\left(X ,
\z\right)$-theory consists of 
an anomalous theory
\begin{equation*}
F_\al \colon \al \to 1_{\left(X , \z\right)}
\end{equation*}
(where $1_{\left(X , \z\right)}$ denotes the trivial $\left(X , \z\right)$-theory)
along with a trivialization $t \colon 1_\fra \lto{\sim} \al_\fra$, 
and an equivalence:
\begin{equation*}
\theta \colon F \simeq F_\al \comp t \ .
\end{equation*}

\begin{defn}
An \emph{$\left(X , \z\right)$-anomaly of the theory $F$} is such a quadruple $\left(\al ,
F_\al , t , \theta\right)$.
The trivialization $t$ and equivalence $\theta$ are often canonically defined, in which
case we will write that $\al \lto{F_\al} 1_X$ is the $\left(X , \z\right)$-anomaly of $F$.
\label{defn:X_anom}
\end{defn}

\begin{prop}
A projective theory with projectivity $\al$ (in the sense of \Cref{defn:projective_TQFT})
is naturally equivalent to an anomalous
TQFT with anomaly $\al$ (in the sense of \Cref{defn:anom_of_F}).
\label{prop:proj_anom}
\end{prop}
\begin{proof}
An anomalous TQFT is a natural transformation from the unit to an invertible TQFT. 
This is, by definition functor from the bordism category to the arrow category such that
composition with the source and target functors give you an invertible object and the unit
respectively.
This is precisely the same condition as factoring through $\PP \Om \cT \to \cT^\down$,
i.e. defining a projective TQFT. 
\end{proof}

\begin{exm}[Framing anomaly]
Let $\left(X , \z\right) = \left(B\SO\left(d\right) , \z_{\io}\right)$ be as in \Cref{exm:Bord_SO}, i.e. 
$\Bord_d^{\left(X , \z \right)} \simeq \Bord_d^\ori$. 
A framed theory $F$ is often said to have a nontrivial \emph{framing anomaly}, if it
cannot be upgraded to an oriented theory, i.e. a $\left(B\SO\left(d\right) ,
\z_{\io}\right)$-theory. 
This is equivalent to having a $\left(B\SO\left(d\right) , \z_{\io}\right)$-anomaly as
in \Cref{defn:X_anom} such that the associated anomaly theory $\al$ is a nontrivial
theory. 
\label{exm:framing_anomaly}
\end{exm}

\begin{exm}
Let $\left(X , \z\right) = \left(BG , \z_{\triv} \right)$ be as in \Cref{exm:Bord_BG}. 
A $BG$-anomaly as in \Cref{defn:X_anom} (or $G$-anomaly) for the theory $F$ is an
anomalous action of $G$ on $F$.
We will see in \Cref{thm:anom_symm} that this is equivalent to $F$ having a 
$\left(\sigma_{BG , \tau}^{d+1} , \rho\right)$-module structure (\Cref{sec:mod_str}) where
the twist $\tau$ is closely related to the anomaly theory.
\label{exm:G_anom}
\end{exm}

\subsection{The anomaly associated to a projectivity class}
\label{sec:proj_anom}

Let $X$ be a space (higher groupoid) equipped with a cocycle $c$ classifying
$\widetilde{X}$:
\begin{equation*}
\begin{cd}
\widetilde{X}
\ar{d}{\pi}&\\
X \ar{r}{c} & B^{d+1} \units{\bk}
\end{cd}
\end{equation*}

Let $\Om \cT= \End_{\cT}\left(1\right)$ denote the looping of the fixed target $\cT$ from
the beginning of \Cref{sec:cob_hyp}.
The anomaly theory $\al_c$ is defined as follows.
By the Cobordism Hypothesis \eqref{eqn:CH_X}, an $X$-theory is determined by a functor from $X$
to $\cT^\sim$, so we can define $\al_c$ to be the theory classified by the following
composition of functors:
\begin{equation*}
X \lto{c}B^{d+1} \units{\bk} \inj \cT \ .
\end{equation*}

By definition, $\al_c$ factors through ``scalars'' in the target, which is sufficient to be
an invertible TQFT. I.e. a relative theory $\al_c\to 1_X$ is an anomalous $X$-theory with
anomaly $\al_c$.

\begin{prop}
Anomalous $X$-theories with anomaly $\al_c$ give rise to $\widetilde{X}$-theories.
\label{prop:alpha_c}
\end{prop}

\begin{proof}
Given a relative theory $F_\al$, this is classified by a map of bundles over $X$:
\begin{equation*}
\begin{cd}
\widetilde{X} \ar{rr}{\abs{F_\al}}\ar{dr}{\pi} &&
X \times \Om \cT^\sim \ar{dl}{p_1}\\
& X &
\end{cd}
\end{equation*}
Composing with the projection to the second factor we obtain a map
\begin{equation*}
\widetilde{X} \lto{\abs{F_\al}} X \times \Om \cT^\sim \lto{p_2} \Om \cT^\sim \ . 
\end{equation*}
The result follows from the Cobordism Hypothesis \eqref{eqn:CH_X}:
Every TQFT $\widetilde{F} \colon \Bord^{\widetilde{X}}_{d} \to \Om \cT$ is classified by
such a map $\widetilde{X} \to \Om \cT^\sim$.
\end{proof}

\begin{exm}
Fix $d = 1$ and let $c$ be a cocycle classifying a central extension $\widetilde{G}$ of a group 
$G$ by $\units{\bk}$. 
If $X = BG$ and $\widetilde{X} = B\widetilde{G}$, then a projective representation of $G$
with projectivity cocycle $c$ (i.e. a linear representation of $\widetilde{G}$) defines a
such a theory $\widetilde{F}$.
\label{exm:proj_repr}
\end{exm}

\begin{rmk}
Recall that the upshot of the discussion in \Cref{sec:ext_bord} is that every projective 
TQFT is an honest (i.e. non-anomalous) functor out of an extension of the original bordism
category. According to \Cref{prop:proj_anom}, the same is true for anomalous theories. 
This is the familiar statement that equipping your bordisms with extra structure allows
one to trivialize the anomaly.
See \Cref{rmk:WRT_p1} for more on this extended category 
in the context of Reshetikhin-Turaev and Crane-Yetter TQFTs.

The bordism category $\Bord^{\widetilde{X}}_{d+1}$ provides a description for (a
fully-extended version of) the extended bordism category for 't Hooft anomalies described
by such a cocycle $c$.

Note that, as discussed in \Cref{rmk:filled}, the functor from the extended bordism
category to the original one is not necessarily essentially surjective, e.g. it might
consist e.g. of ``filled'' bordisms if the anomalous theory in question does not extend to
the point.
For instance, in this case the point is not an object of the extended category, but $S^0$
is.
\label{rmk:resolve_anom}
\end{rmk}

\subsubsection{Trivializing the anomaly}
\label{sec:triv_anom}

We have a symmetric-monoidal functor
$\pi_* \colon \Bord_{d}^{\widetilde{X}} \to \Bord_{d}^X$ given by composing 
the map defining the $\widetilde{X}$-structure on a bordism with the map $\pi$.

\begin{qn}
Does $\widetilde{F}$ factor as follows?
\begin{equation*}
\begin{tikzcd}
\Bord^{\widetilde{X}}_{d}\ar{d}{\pi_*}\ar{r}{\widetilde{F}}&
\Om \cT \\
\Bord^X_{d} \ar[swap,dashed]{ur}{F_X}
\end{tikzcd}
\end{equation*}
\label{qn:anomaly}
\end{qn}

\Cref{qn:anomaly} is equivalent to the question of splitting $\widetilde{X}\to X$
(i.e. trivializing $c$).
We can ask a weaker question, by pulling back along a map $f \colon Y \to X$. 
Write $f^* \widetilde{X}$ for the pullback:
\begin{equation*}
\begin{tikzcd}
f^* \widetilde{X} \ar{r}\ar{d}{f^* \pi}&
\widetilde{X} \ar{d}{\pi} \\
Y \ar{r}{f}&
X
\end{tikzcd}
\end{equation*}
Write $f^* \widetilde{F}$ for the corresponding $f^* \widetilde{X}$-theory:
\begin{equation*}
f^* \widetilde{F}  \colon \Bord_d^{f^*\widetilde{X}} \to \Bord_d^{\widetilde{X}}
\lto{\widetilde{F}} \Om \cT
\ .
\end{equation*}
Now we can state a refinement of \Cref{qn:anomaly}.

\begin{qn}
Does $f^* \widetilde{F}$ factor as follows?
\begin{equation*}
\begin{tikzcd}
\Bord^{f^*\widetilde{X}}_{d}\ar[swap]{d}{\left(f^* \pi\right)_*}\ar{r}{f^* \widetilde{F}}&
\Om \cT \\
\Bord^Y_{d} \ar[swap,dashed]{ur}{F_Y}
\end{tikzcd}
\end{equation*}
\label{qn:anomaly_Y}
\end{qn}

\Cref{qn:anomaly_Y} is similarly equivalent to splitting $f^* \widetilde{X} \to Y$, or
equivalently trivializing $f^* c$.

A trivialization $1 \lto{\sim} \al_c$ would provide us with an $X$-theory by
\Cref{prop:1_X}. 
However if $c$ is not trivializable, the best we will be able to do is to trivialize the
pullback along some map $f \colon Y \to X$.
This allows us to regard the anomaly theory $\al_c$ as a $Y$-theory:
\begin{equation}
\begin{tikzcd}
\Bord^Y_{d}\ar{d}{f} \ar{dr}{\al_{f^* c}} \ar{dr} & \\
\Bord^X_{d} \ar{r}{\al_c} &
\cT
\end{tikzcd}
\label{eqn:Y_anom}
\end{equation}

A trivialization of $\widetilde{X}$ over $Y$ (i.e. of $f^* c$) 
uniquely determines a trivialization $1_Y \lto{\sim} \al_c^Y$.
Composing with the anomalous theory gives $1_Y\lto{\sim} \al_c^Y \to 1_Y$, which is a
$Y$-theory by \Cref{prop:1_X}.

\subsection{Anomalies and sandwiches}
\label{sec:anom_sand}

Let $\cT$ be the $\left(d+1\right)$-dimensional target fixed in \Cref{sec:cob_hyp}.
The $d$-dimensional TQFTs we will consider will be valued in the looping $\Om \cT$. 
Let $X$ be a (pointed, connected) $\pi$-finite space.
Consider a cocycle $c \colon X \to B^{d+1}\units{\bk}$ classifying $\widetilde{X}$:
\begin{equation*}
\begin{cd}
\widetilde{X}
\ar{d}{\pi}&\\
X \ar{r}{c} & B^{d+1} \units{\bk} \ .
\end{cd}
\end{equation*}
Let $\al_c$ denote the anomaly theory from \Cref{sec:proj_anom}, and let
$\sigma_{\widetilde{X}}$ denote the theory from \Cref{sec:sigma}.

\begin{theorem}
An $X$-anomaly $F_c \colon \al_c \to 1_X$ of $F$ (as in \Cref{defn:X_anom}) naturally 
defines:
\begin{enumerate}[label = (\roman*)]
\item A projective $X$-theory $\overline{F} \colon \Bord_d^X \to \PP\Om \cT$
with projectivity $\al_c$ and underlying theory $F$ (as in \Cref{defn:projective_TQFT}).
\label{proj}

\item $\widetilde{F} \colon \Bord^{\widetilde{X}}_d \to \Om \cT$ with underlying framed
theory, written $F$.
\label{gerbe}

\item Assuming \Cref{hyp:sigma}, a $\left(\sigma_{X , c} , \rho_{X , c}\right)$-module structure 
on $F$ (\Cref{sec:mod_str}).
\label{symm}
\end{enumerate}
\label{thm:anom_symm}
\end{theorem}

\begin{proof}
The fact that $F_c$ determines \ref{proj} follows from \Cref{prop:proj_anom}. 
Similarly, \Cref{prop:alpha_c} shows that $F_c$ determines \ref{gerbe}.
Finally, the fact that $F_c$ determines \ref{symm} follows from \Cref{prop:sigma_mod}
(which depends on \Cref{hyp:sigma}). 
\end{proof}

\begin{theorem}
Consider the setup of \Cref{thm:anom_symm}, and assume that the cocycle $c$ is
trivializable after being pulled back along a map $f \colon Y \to X$. A trivialization of
$f^* c$ naturally defines a trivialization of the anomaly theory $\al_{f^* c}$ (as in
\eqref{eqn:Y_anom}) which in turn defines:
\begin{enumerate}[label = (\roman*)]
\item A linearization of $\overline{F}$ (as in \Cref{prop:linearize}).
\label{triv_proj}

\item A $Y$-theory $F_Y$ which factors $f^*\widetilde{F}$ (as in \Cref{qn:anomaly_Y}).
\label{triv_gerbe}

\item Assuming \Cref{hyp:sigma}, a reduction (\Cref{defn:reduction})
from the $\left(\sigma_{X , c} ,
\rho_{X , c}\right)$-module structure in \Cref{thm:anom_symm} \ref{symm}
to a $\left(\sigma_Y , \rho_Y\right)$-module structure (as in \Cref{defn:reduction}).
\label{triv_symm}
\end{enumerate}
Furthermore, the absolute $Y$-theories obtained from 
\ref{triv_proj}, \ref{triv_gerbe}, and \ref{triv_symm}
agree with the one obtained from the trivialization of the anomaly.

Note that if one such trivialization exists, then the collection of trivializations forms
a torsor over $H^{d-1}\left(X , \units{\bk}\right)$.
\label{thm:triv_anom_symm}
\end{theorem}

\begin{proof}
In terms of \Cref{qn:anomaly_Y}, 
a trivialization of $\widetilde{X}$ over $Y$ (i.e. a trivialization of $f^* c$) explicitly
defines a section $\Bord^Y_{d} \to \Bord^{f^*\widetilde{X}}_{d}$, which we can compose
with $f^*\widetilde{F}$ to obtain the desired theory $F_Y$.

The theory $\al_{f^* c}$ (from \eqref{eqn:Y_anom}) is defined by the map $f^* c$, so a
trivialization of $f^* c$ automatically provides a trivialization 
$1_Y \lto{\sim} \al_{f^* c}$. 
This produces an automorphism of the trivial $Y$-theory $1_Y \lto{\sim} \al_{f^* c} \to
1_Y$, which is a theory on $\Bord^Y_{d}$ by \Cref{prop:1_X}.
This $Y$-theory is classified (via the Cobordism Hypothesis \eqref{eqn:CH_X})
by the functor $Y \to \Om \cT$ given by the composition of the section of $f^*
\widetilde{X}$ with the functor $\widetilde{X} \to \Om \cT^\sim$ classifying the anomalous
theory:
\begin{equation*}
Y \to f^* \widetilde{X} \to \Om \cT^\sim \ .
\end{equation*}
This agrees with the functor $Y\to \Om \cT^\sim$ classifying $F_Y$ by
\Cref{thm:anom_symm}.
The agreement between the trivialization of the anomaly and the linearization
\ref{triv_proj} is \Cref{prop:linearize}.

A trivialization of $f^*c$ determines a reduction to $\left(\sigma_Y^{d+1} ,
\rho_Y\right)$ by \Cref{prop:reduction_Y}.
In particular, the boundary theory $1 \to \sigma_{\widetilde{X}}^{d+1}$ becomes reduced to a
boundary theory $1\to \sigma_Y^{d+1}$ which 
defines a $Y$-theory by \Cref{prop:sigma_mod}.
This agrees with $F_Y$ by the commuting square \eqref{eqn:commuting_square}
in \Cref{rmk:hyp:sigma} associated to $X_1
= X$, $X_2 = \widetilde{X}$, and the morphism in \eqref{eqn:triv_corr}.
\end{proof}

\subsection{A universal anomaly}
\label{sec:univ_anom}

Let 
\begin{equation*}
F \colon \Bord^\fra_{d} \to \Om \cT
\end{equation*}
be a TQFT sending the point to an object $\cS$ of $\Om \cT$.
Let $\widetilde{X} = B\Aut_{\Om \cT}\left(\cS\right)$. 
Because we are assuming $\Om^{d+1} \cT \cong \bk$ (see the beginning of \Cref{sec:cob_hyp}),
we know that $\units{\bk}$ maps to $\Om^{d} \Aut\left(\cS\right)$, however it is only an
equivalence when $\cS$ satisfies a sort of simplicity/irreducibility. 

\begin{defn}
We say an object $\cS$ of $\Om \cT$ is a \emph{Schur object} if $\Om^{d} \Aut
\left(\cS\right) \cong \units{\bk}$.
\label{defn:ss}
\end{defn}

If $\cS$ is a Schur object, then the top level of the Postnikov tower provides a
projectivity class $k$ on the truncation to a $d$-type:
\begin{equation}
\begin{cd}
B\Aut_{\Om \cT}\left(\cS\right)
\ar{d}{\pi} & \\
\pi_{\leq d} B\Aut_{\Om \cT}\left(\cS\right) \ar{r}{k}&
B^{d+1} \units{\bk}
\ .
\end{cd}
\label{eqn:k}
\end{equation}

The following follows from \Cref{thm:anom_symm,thm:triv_anom_symm} by setting
\begin{align}
\widetilde{X} = 
B\Aut_{\Om \cT}\left(\cS\right)
\ ,
&&
X = \pi_{\leq d} B\Aut_{\Om \cT}\left(\cS\right) \ ,
&&
c = k \ , \text{ and }
&&
Y = BG
\end{align}
for some (possible higher) group $G$.

\begin{cor}
Let $\cS$ be a Schur object (\Cref{defn:ss}) of an $\left(\infty , d+1\right)$-category
$\cT$ with duals which satisfies $\Om^{d+1} \cT = \bk$.
\begin{itemize}
\item The class $k$ from \eqref{eqn:k} uniquely determines the following.
\begin{enumerate}
\item The gerbe $B\Aut_{\Om \cT}\left(\cS\right) \to X$ of objects of $\Om \cT$ equivalent
to $\cS$.

\item The anomaly theory $\al_k$ (as in \Cref{sec:proj_anom}).

\item The abstract projectivity $\al_k$ (as in \Cref{defn:projective_TQFT}).

\item The quiche $\left( \sigma_{X , k} , \rho_{X , k} \right)$ (as in \Cref{sec:sigma}).
\end{enumerate}

\item The object $\cS$ uniquely determines the following.
\begin{enumerate}[label = (\roman*)]
\item An $\widetilde{X}$-theory $\widetilde{F} \colon \Bord^{B\Aut
\left(\cS\right)}_d \to \cT$.

\item An $X$-anomaly $F_k \colon \al_k \to 1$ (as in \Cref{defn:X_anom}).
\label{cor_anom}

\item A projective theory $\overline{F}$ with projectivity $\al_k$
(as in \Cref{defn:projective_TQFT}).
\label{cor_proj}

\item Assuming \Cref{hyp:sigma}, a $\left(\sigma_{X , k} , \rho_{X , k} \right)$-module structure on $F$
(as in \Cref{sec:mod_str}).
\label{cor_symm}
\end{enumerate}

\item Given $f\colon BG \to X$ for some (possibly higher) group $G$, 
a trivialization of $f^* k$ determines the following.
\begin{enumerate}[label = (\alph*)]
\item A theory $F_{BG}$ which factors $\widetilde{F}$ (as in \Cref{qn:anomaly_Y}).

\item A trivialization of the anomaly $\al_{f^* k}$ (as in \eqref{eqn:Y_anom}).
\label{triv_cor_anom}

\item A linearization of the projective theory $\overline{F}$ (as in
\Cref{prop:linearize}).
\label{triv_cor_proj}

\item Assuming \Cref{hyp:sigma}, a reduction to a $\left(\sigma_{BG} , \rho_{BG}\right)$-structure 
on $F$ (as in \Cref{defn:reduction}).
\label{triv_cor_symm}
\end{enumerate}
Furthermore, the trivialization in \ref{triv_cor_anom}, the linearization in
\ref{triv_cor_proj},
and the reduction in \ref{triv_cor_symm} determine three $BG$-theories, which all agree with
$F_{BG}$.
If such a trivialization exists,  then the trivializations of $f^* k$ form a torsor over
$H^{d-1}\left(BG , \units{\bk}\right)$.
\end{itemize}
\label{cor:anom_symm}
\end{cor}

\section{Fusion categories}
\label{sec:fusion}

\subsection{Preliminaries}
\label{sec:fusion_prelim}

Fix an algebraically closed field $\bk$ of characteristic zero.
Let $\Pre$ be the symmetric monoidal $2$-category of presentable $\bk$-linear categories with
colimit preserving functors (and natural transformations).\footnote{A category is
presentable (sometimes called locally presentable) if it is accessible (generated under
colimits by a small subcategory) and cocomplete (closed under small colimits).}

As in \cite[\S 3]{BJS:dualizable}, we will consider the 
``even higher'' Morita category \cite{JFS} of $\EE_2$-categories.
This Morita category of braided monoidal categories was proposed in \cite[\S 9]{Wal:TQFT},
\cite{DSPS:dualizable}, and \cite{BZBJ}.
See \Cref{rmk:morita} for more details on the specific model of the Morita category used
here.

The Morita $3$-category of tensor categories, $\Tens =
\Alg_1\left(\Pre\right)$, consists of the following:
\begin{itemize}
\item the objects are tensor categories, 
\item the $1$-morphisms are bimodule categories, 
\item the $2$-morphisms are functors between bimodule categories, and 
\item the $3$-morphisms are natural transformations.
\end{itemize}
We will focus on the subcategory $\Fus$ consisting of fusion categories, semisimple
bimodule categories, compact-preserving cocontinuous bimodule functors, and natural
transformations. 
It was shown in \cite{BJS:dualizable} that this forms a subcategory. 

Following \cite{BJS:dualizable}, we define ``fusion'' as follows.
Given an $\EE_1$-algebra object of $\Pre$, there is an increasingly strict list of
finiteness conditions one can insist upon:
\begin{itemize}
\item cp-rigid: every compact projective object is dualizable.
\item compact-rigid: every compact object is dualizable.
\item finite: compact rigid and the underlying category is finite.
\item fusion: finite, semisimple,\footnote{By semisimple, we mean \cite[Definition
2.27]{BJS:dualizable}: Every object is a (possibly infinite) direct sum
of simple objects.} with simple unit.
\end{itemize}

\begin{rmk}
Recall the original notions of rigid, finite, and fusion
abelian categories as in \cite{EO:finite,ENO:fusion,EGNO:TC}.
As explained in \cite{BJS:dualizable,BJSS:invertible} 
this setting is related to the one defined above by an $\Ind$-completion, 
in the following sense.
The ambient category of abelian categories is not suitable for TQFT constructions, because
it is not closed under the (Deligne) tensor product. Instead, as in \cite{Shim}, 
we can work in the ambient category of finitely co-complete $\bk$-linear categories with
right exact functors and natural equivalences, written $\Rex$.

As is explained in \cite[\S 3]{BZBJ}, taking the $\Ind$-completion defines an equivalence
between $\Rex$ and the subcategory $\Pre_c \subset \Pre$ consisting of 
compactly generated $\bk$-linear presentable categories, compact functors, and natural
equivalences.\footnote{A functor is compact if it preserves compact objects.}
As in \cite[\S 3]{BZBJ}, $\Ind$ and the functor taking compact objects of a presentable
category define an equivalence of $\left(2,1\right)$-categories $\Rex \simeq \Pre_c$.

As it turns out, if an abelian category is rigid, then its $\Ind$-completion is
compact-rigid in the above sense. 
Similarly, if an abelian category is finite (in the usual sense \cite{EO:finite}) then its
$\Ind$-completion is finite (compact-rigid and finite underlying category). 

The upshot of this discussion is that the $\Ind$-completion of a fusion category in the
traditional sense lands in the category $\Fus\subset \Alg_1\left(\Pre\right)$ as it is
defined above.
Note that the Deligne-Kelly tensor product coincides with the Deligne tensor
product of finite abelian categories (and commutes with taking ind-completions).
Therefore we can apply theorems in the setting of finite abelian categories to the objects
of $\Fus \subset \Alg_1\left(\Pre\right)$. 
In particular, we will use various facts from \cite{DSPS:dualizable,ENO:homotopy} (which
are in the abelian setting) throughout. 
\label{rmk:fusion}
\end{rmk}

\subsubsection{Braided fusion categories}

Similarly, braided tensor categories form a $4$-category $\BrTens = \Alg_2\left(\Pre\right)$:
\begin{itemize}
\item the objects are braided tensor categories, 
\item the $1$-morphisms are associative algebra objects
in the category of bimodule categories, 
\item the $2$-morphisms are bimodule categories between bimodule categories, 
\item the $3$-morphisms are functors between bimodule categories, and 
\item the $4$-morphisms are natural transformations.
\end{itemize}
The subcategory $\BrFus$ consists of braided fusion categories, fusion categories equipped
with central structures, finite semisimple bimodule categories, compact-preserving
cocontinuous bimodule functors, and bimodule natural transformations.
It was shown in \cite{BJS:dualizable} that this forms a subcategory. 

Braided fusion categories also form a $2$-category:
\begin{itemize}
\item the objects are braided fusion categories,
\item the $1$-morphisms are functors preserving the braided structure, and
\item the $2$-morphisms are natural transformations.
\end{itemize}
The $2$-groupoid given by the invertible part of this $2$-category is written $\EqBr$, as in
\cite{ENO:homotopy}.

\subsection{Higher groupoids attached to a braided fusion category}
\label{sec:higher_groupoids}

Let $\cA$ be a braided fusion category.
The braided equivalences of $\cA$ form a $2$-group
$\Aut_{\EqBr}\left(\cA\right)$.
From any $2$-group, we can construct its classifying space
$B\Aut_{\EqBr}\left(\cA\right)$, which is a $2$-type. I.e. it has two nontrivial
homotopy groups.

To $\cA$ we can also attach the $2$-category of $\cA$-module categories. 
The braiding defines an embedding from the $2$-category of $\cA$-modules to
the $2$-category of $\cA$-bimodules. Therefore the category of $\cA$-modules inherits a
monoidal structure from the natural one on the category of $\cA$-bimodules.

\begin{defn*}
The invertible $\cA$-modules, written $\Pic\left(\cA\right)$, comprise 
the \emph{Picard $3$-group of $\cA$}.
\end{defn*}

From any $3$-group, we can construct its classifying space $B\Pic\left(\cA\right)$, which
is a $3$-type. I.e. it has three nontrivial homotopy groups. 

\begin{rmk}
Some authors write the full $3$-group/type as
$B\under{\under{\Pic}}\left(\cA\right)$, and write the truncation to a $2$-type as
$B\under{\Pic}\left(\cA\right)$ and to a $1$-type as $B\Pic\left(\cA\right)$. 

Instead, we write $B\Pic\left(\cA\right)$ for the full $3$-group, and write 
$\tau_{\leq 2} B\Pic\left(\cA\right)$  and $\tau_{\leq 1} B\Pic\left(\cA\right)$ for the
truncations. 
\end{rmk}

\begin{thm*}[\cite{ENO:homotopy}]
For a nondegenerate braided fusion category $\cA$ we have an equivalence:
\begin{equation}
\pi_{\leq 2}B\Pic\left(\cA\right)\simeq B\Aut_{\EqBr}\left(\cA\right)\ .
\label{eqn:Pic_Eq}
\end{equation}
\end{thm*}

The homotopy groups of $B\Pic\left(\cA\right)$,
e.g. from \cite[Proposition 7.5]{ENO:homotopy} are as follows:
$\pi_1$ is given by the ordinary group of isomorphism classes of
braided equivalences, and $\pi_2$ is given by the group of tensor isomorphisms of the
identity functor on $\cA$.
The top homotopy group, which is not involved in \eqref{eqn:Pic_Eq}, is
$\pi_3B\Pic\left(\cA\right) = \units{\bk}$.

\subsubsection{Braided fusion categories which are Drinfeld centers}

Let $\cC$ be a fusion category.
To this we can attach the monoidal $2$-category $\End_{\Fus}\left(\cC\right)$.

\begin{defn*}
The \emph{Brauer-Picard $3$-group} of $\cC$ 
is $\Aut_{\Fus}\left(\cC\right)$.
\end{defn*}

Recall the Drinfeld center of a monoidal category $\cC$ is a braided monoidal category,
written $\cZ\left(\cC\right)$.

\begin{thm*}[\cite{EO:finite}]
There is an equivalence of $2$-categories:
\begin{equation*}
\End_{\Fus}\left(\cC\right) \lto{\sim} \lMod{\cZ\left(\cC\right)} \ .
\end{equation*}
\end{thm*}

Passing to the invertible part, we obtain:
\begin{equation}
\Aut_{\Fus}\left(\cC\right) \simeq \Pic\left(\cZ\left(\cC\right)\right) \ .
\label{eqn:BrPic_Pic}
\end{equation}

\begin{rmk}
In \cite{ENO:weakly}, it is shown that two fusion categories are Morita equivalent if and
only if their Drinfeld centers are braided equivalent.
This was strengthened in \cite{ENO:homotopy}: they prove that there is a fully-faithful
embedding of groupoids $\EqBr \to \units{\Fus}$.
Indeed, combining \eqref{eqn:Pic_Eq} and \eqref{eqn:BrPic_Pic} we obtain
\begin{equation}
\pi_{\leq 2} B\Aut_{\Fus}\left(\cC\right) \simeq B\Aut_{\EqBr}\left(\cZ\left(\cC\right)\right) \ .
\label{eqn:BrPic_Eq}
\end{equation}

So the center determines the Morita class of the fusion category itself.
However, since the equivalence \eqref{eqn:Pic_Eq} is with the truncated part of
$\Aut_\Fus\left(\cC\right)$, there is a subtle lack of ``coherence'', captured by an
obstruction class originally studied in \cite{ENO:homotopy} and surveyed in
\Cref{sec:O4}.

Related results were shown in \cite{KZ:center} and \cite[\S 5]{KK:gapped}.
\label{rmk:functorial_Z}
\end{rmk}

In summary: given a braided monoidal category $\cA$, we have a $3$-type
$B\Pic\left(\cA\right)$ attached to it.
If it happens to be the case that $\cA = \cZ\left(\cC\right)$ for some fusion category
$\cC$, then this $3$-type is the classifying space of the Brauer-Picard $3$-group of
$\cC$. All together we have:
\begin{equation*}
\begin{tikzcd}[column sep=tiny]
\End_{\Fus}\left(\cC\right) 
\ar["\text{\cite{EO:finite}}"']{rr}{\sim}& &
\lMod{\cZ\left(\cC\right)}  \\
B\Aut_{\Fus}\left(\cC\right) \ar{rr}{\sim}\ar{d}&&
B\Pic\left(\cZ\left(\cC\right)\right) \ar{d} \\
\pi_{\leq 2} B\Aut_{\Fus}\left(\cC\right) 
\ar{rr}{\sim}
\ar{dr}{\sim}&&
\pi_{\leq 2} B\Pic\left(\cZ\left(\cC\right)\right) \ar[swap,"\text{\cite{ENO:homotopy}}"']{dl}{\sim}\\
& B\Aut_{\EqBr}\left(\cZ\left(\cC\right)\right) & 
\end{tikzcd}
\end{equation*}

\subsection{Pointed braided fusion categories}
\label{sec:pointed}

A fusion category is \emph{pointed} if all simple objects are invertible.

\begin{exm}
Let $L$ be a finite group. 
Consider the category of $L$-graded vector spaces, $\Vect\left[L\right]$,
with convolution, i.e. for $a,b,c \in L$:
\begin{equation*}
\left(W * W'\right)_a =  \bdsum_{a=bc} W_b \tp W_c \ .
\end{equation*}
This category has simple objects $\bk_a$ for $a\in L$, and is therefore pointed.
Given a $3$-cocycle $\tau$ on $G$, we can define a variant of this category
$\Vect\left[L\right]^\tau$, which is still pointed.
\label{exm:pointed}
\end{exm}

As it turns out, all pointed fusion categories are of the form $\Vect\left[L\right]^\tau$. 
(Note however that the twisted version may have no fiber functor. This is a higher analogue
of the fact that not all algebras have an augmentation map.)

\subsubsection{Metric groups}
\label{sec:metric_groups}

Metric groups \cite{DGNO,ENO:homotopy} play an important role in the theory of pointed
braided fusion categories.
Let $\bk$ be an algebraically closed field of characteristic zero, and 
let $A$ be a finite abelian group. Write 
$\pdual{A} \ceqq \Hom \left(A , \units{\bk}\right)$ 
for the character dual of $A$.

\begin{defn*}
A biadditive map 
\begin{equation*}
\lr{\cdot , \cdot} \colon A \times A \to \units{\bk}
\end{equation*}
is a \emph{symmetric bicharacter} if for all $a,b\in A$ we have
$\lr{a,b} = \lr{b,a}$. 
By acting on the first argument, such a bicharacter $\lr{\cdot , \cdot}$ defines a
homomorphism $A \to \pdual{A}$. If this is an isomorphism, then we say that $\lr{\cdot ,
\cdot}$ is \emph{nondegenerate}.
\end{defn*}

\begin{defn*}
A function
\begin{equation*}
q \colon A \to \units{\bk}
\end{equation*}
is a \emph{quadratic form} on $A$ if 
\begin{enumerate}[label = (\alph*)]
\item for all $a\in A$ we have 
$q\left(a\right) = q\left(-a\right)$, and 
\item the following expression defines a symmetric bicharacter:
\begin{equation*}
\lr{a , b}_q \ceqq \frac{q\left(a+b\right)}{q\left(a\right)q\left(b\right)}
\ .
\end{equation*}
\end{enumerate}
We say $q$ is \emph{nondegenerate} if the bicharacter
$\lr{\cdot , \cdot}_q$ is nondegenerate.
\end{defn*}

\begin{rmk}
If $\abs{A}$ is odd, then 
sending $q$ to $\lr{\cdot , \cdot}_q$ defines a bijection between quadratic forms and
bicharacters.
\end{rmk}

\begin{defn*}
Any pair $\left(A , q\right)$ is a \emph{pre-metric group}.
A pair $\left(A , q\right)$ is called a \emph{metric group} if $q$
is nondegenerate. 
\end{defn*}

\begin{defn*}
The \emph{orthogonal group} of a metric group $\left(A, q\right)$ is:
\begin{equation*}
\O\left(A , q\right) \ceqq \left\{
f\in \Hom \left(A , A\right) \st
q\comp f = q
\right\}
\ .
\end{equation*}
\end{defn*}

\begin{defn*}
Let $\left(A , q\right)$ be a metric group. 
Define the \emph{determinant} 
\begin{equation*}
\det \colon \O\left(A ,q\right) \to \units{\QQ}_{>0} / \left(\units{\QQ}_{>0}^2\right)
\end{equation*}
by sending $g\in \O\left(A , q\right)$ to the image of 
$\abs{\left(g-1\right)A}\in\NN$.

The \emph{special orthogonal group} is 
\begin{equation}
\SO\left(A , q\right) \ceqq \ker\left(\det\right) \ .
\label{eqn:SO}
\end{equation}
\end{defn*}

Given an abelian group $A$, a quadratic quadratic form allows us to define a non-symmetric
braided structure on the (symmetric) monoidal category of $A$-graded vector spaces.
We can specify this braiding on simple objects:
\begin{equation}
\begin{tikzcd}[column sep=huge]
\bk_a * \bk_b = \bk_{ab} \ar["\sim"']{r}{\lr{a,b} \cdot \id_{\bk_{ab}}}&
\bk_{ab} =  \bk_b * \bk_a
\ .
\end{tikzcd}
\label{eqn:braiding}
\end{equation}

\subsubsection{Classification of pointed braided fusion categories}
\label{sec:brfus_class}

It is shown in \cite{JS} that pointed braided fusion categories are classified by finite abelian
groups $A$ (the group of simple objects) equipped with a quadratic form $q \colon A \to
\units{\bk}$ which is not necessarily nondegenerate. 
See also \cite[\S 8.4]{EGNO:TC}.

\begin{exm}
If $L$ is any finite abelian group, 
then the finite abelian group $L\dsum \pdual{L}$ has a nondegenerate quadratic form given
by evaluation, so 
$\left(L\dsum \pdual{L} , \ev\right)$ is a metric group.
The symmetric bicharacter corresponding to evaluation is:
\begin{equation*}
\lr{ \left(\ell , \chi\right) , \left(\ell' , \chi'\right)}
=\chi\left(\ell'\right)  \chi'\left(\ell\right)
\ .
\end{equation*}
We obtain a braiding on $\Vect\left[L\dsum \pdual{L}\right]$ via \eqref{eqn:braiding}.
\label{exm:ev}
\end{exm}

Given an abelian group $A$, it is shown in 
\cite[Theorem 26.1]{EM:2} 
that  $H^4\left(B^2 A , \units{\bk}\right)$ is isomorphic to the group of quadratic forms 
$A \to \units{\bk}$.

Eilenberg and MacLane defined an explicit chain complex which computes this cohomology
\cite{Mac:ab,EM:ab_1,EM:ab_2,EM:ab_3}.
They call this abelian group cohomology, and write it as $H^*_\ab$:
\begin{equation*}
H^n_\ab\left(A , B\right) = H^{n+1}\left(K\left(A , 2\right) , B \right)
\end{equation*}
for any two abelian groups $A$ and $B$.
I.e. we have
\begin{equation}
H^3_\ab\left(A , \units{\bk}\right) = H^4\left(B^2 A , \units{\bk}\right) \cong 
\Quad\left(A , \units{\bk}\right) \ .
\label{eqn:q}
\end{equation}
Cocycle representatives in $H^3_{\ab} \left( A , \units{\bk}\right)$ 
are pairs $\left(\tau , b\right)$ where $\tau$ is an ordinary group cocycle on $A$, and
$b$ is a map $A^2 \to \units{\bk}$ satisfies certain compatibility relations with $\tau$.
Explicitly, the identification \eqref{eqn:q} sends a pair $\left(\tau , b\right)$ to the quadratic form
defined by $q\left(a\right) \ceqq b\left(a , a\right)$.

\begin{exm}
This is a continuation of \Cref{exm:ev}.
Recall the pointed braided fusion category $\cZ\left(\Vect\left[L\right]\right)$ was
classified by the finite abelian group $A = L\dsum \pdual{L}$ with quadratic form $q =
\ev$. 

A degree-$3$ cocycle in abelian group cohomology is a pair $\left(\tau , b\right)$ where 
$\tau$ is an ordinary group $3$-cocycle on $A$ 
and $b \colon A^2 \to \units{\bk}$ (satisfying some conditions). 

The abelian group $3$-cocycle classifying $\cZ\left(\Vect\left[L\right]\right)$ has
trivial $\tau$, and $b = b_0$ is a map from $A^2$ to $\units{\bk}$ such that it agrees
with $q$ when restricted to the diagonal, i.e. for all $a\in A$:
\begin{equation*}
b_0\left(a , a\right) = \ev\left(a\right) \ .
\end{equation*}
For example, $b_0$ might be taken to be:
\begin{equation*}
b_0\left(
\left(\ell_1 , \chi_1\right), \left(\ell_2 , \chi_2\right)
\right)
= \chi_1\left(\ell_2\right) \ .
\end{equation*}
Other choices for $b_0$ differ from this by an abelian group
coboundary in degree $3$.
\label{exm:b0}
\end{exm}

\subsubsection{Higher groupoids in the pointed case}
\label{sec:pointed_groupoids}

When the (braided) fusion categories in question are pointed, the groupoids introduced in 
\Cref{sec:higher_groupoids} have a more concrete description. 
Let $\left(A , q\right)$ be a metric group, and let 
$\cA = \Vect\left[A\right]$ be the associated braided category of $A$-graded vector
spaces.
The Picard $3$-type of $\cA$ has the following homotopy groups:
\begin{align}
\pi_1 = \O\left(A , q\right) && 
\pi_2 = A && \pi_3 = \units{\bk} \ .
\label{eqn:pointed_homotopy}
\end{align}

\begin{rmk}
Following \cite[Remark 10.7]{ENO:homotopy}, these homotopy groups appear as follows.
First of all, we know that $\pi_2 = A$ and $\pi_3 = \units{\bk}$,
since this is the underlying groupoid of $\cA$ itself (delooped twice). 

Now we identify $\pi_1$. 
The Whitehead half-square (see \cite[\S 7.3]{ENO:homotopy}) is a homotopy invariant, which 
is a map $\pi_2 \to \pi_3$. 
As it turns out, it is given by the (square of the) braiding on $\cA$, which
is defined by $q$ in this case. 
The action of $\pi_1$ on $\pi_2$ by conjugation must preserve this invariant, and hence
this defines a map
$\pi_1 \to \O\left(A , q\right)$, which turns out to be an equivalence. 

This conjugation action is a shadow of a more general way to understand these groupoids:
The invertible part of any monoidal $2$-category acts on the endomorphism of the
identify object by conjugation. 
As is explained in \cite[Remark 5.4]{ENO:homotopy}, by setting the monoidal $2$-category to be
$\End_\Fus\left(\cC\right)$, we obtain a map:
\begin{equation*}
\Aut_\Fus\left(\cC\right) \to \Aut_{\EqBr}\left(\cZ\left(\cC\right)\right) \ .
\end{equation*}
When we restrict this to the truncation of $\Aut_\Fus\left(\cC\right)$ to a $1$-category
(i.e. $2$-group), we obtain the equivalence \eqref{eqn:BrPic_Eq}.
When we set the $2$-category to be the category of modules over $\cZ\left(\cC\right)$, we
obtain the equivalence \eqref{eqn:Pic_Eq}.
\label{rmk:Drinfeld_map}
\end{rmk}

\subsubsection{Lagrangian subgroups}
\label{sec:lagrangian}

\begin{defn*}
A subgroup $L\subset A$ of a metric group $\left(A , q\right)$ is \emph{isotropic} if
$q\left(\ell\right) = 1$ for all $\ell \in L$. 
The subgroup is \emph{Lagrangian} if $\abs{L}^2 = \abs{A}$.
\end{defn*}

\begin{exm}
If $A = L\dsum \pdual{L}$ and $q = \ev$, then $L\dsum \left\{0\right\}$ and 
$\left\{0\right\}\dsum \pdual{L}$ are both Lagrangian subgroups.
\end{exm}

Let $L$ be a finite abelian group. 
The Drinfeld center of the fusion category $\left(\Vect\left[L\right] , *\right)$ is:
\begin{equation*}
\cZ\left(\Vect\left[L\right] , *\right) \cong \Vect\left[L\dsum \pdual{L}\right] \ .
\end{equation*}
The tensor structure is convolution, and the braiding is induced by evaluation,
as in \Cref{exm:ev}.

Recall from \Cref{sec:brfus_class}, that pointed braided fusion categories are classified
by finite abelian groups equipped with a quadratic form. 
In this case, the finite abelian group associated to $\cZ\left(\Vect\left[L\right]\right)$
is $L\dsum \pdual{L}$, and the quadratic form is $\ev \colon L\dsum \pdual{L} \to
\units{\bk}$.
Lagrangians in the metric group give rise to fusion categories $\cC$ such that $\cA
\cong \cZ\left(\cC\right)$. 

\subsubsection{Pointed Drinfeld centers}
\label{sec:pointed_Z}

Let $\cC = \Vect\left[L\right]^\tau$ for $L$ a finite group and $\tau$ a $3$-cocycle on
$L$. Following \cite[\S 4]{MN:BrPic}, for any $\ell$ in the center of $L$, we can define:
\begin{equation}
T_{\ell} \left(- , -\right) = 
\frac{\tau\left(\ell , - , -\right) \tau\left(- , - , \ell\right)}{
\tau\left(- , \ell , -\right)} \ .
\label{eqn:T}
\end{equation}
As it turns out, $T_\ell$ is a $2$-cocycle on the group $L$, and
$T_{\left(-\right)}$ defines a group homomorphism from the center of $L$ to
$H^2\left(L , \units{\bk}\right)$.
Then, as in \cite[Corollary 4.3]{MN:BrPic}, $\cZ\left(\cC\right)$ is pointed if and only if 
$L$ is abelian and $T_{\left(-\right)}$ is the trivial group homomorphism. 

\begin{exm}
When $L$ is a vector space over $\FF_p$, 
by \cite[Corollary 5.3]{MN:BrPic} $\cZ\left(\cC\right)$ is pointed if and only if 
$\tau$ has trivial alternating component. (Recall $H^3\left(L , \units{\bk}\right)$ decomposes as 
a direct sum of the alternating factor $\ext^3 \pdual{L}$ and the symmetric factor $\Sym^2
\pdual{L}$.)
\end{exm}

\subsubsection{Polarizability}
\label{sec:polarizability}

As we saw in \Cref{sec:lagrangian}, $\cZ\left(\Vect\left[L\right]\right)$ is
$\Vect\left[L\dsum \pdual{L}\right]$ with braiding induced by the quadratic form $\ev$.
In \Cref{sec:pointed_Z} we claim that the Drinfeld center of $\Vect\left[L\right]^\tau$ is
still pointed as long as the homomorphism $T_{\left(-\right)}$ defined in \eqref{eqn:T}
(which depends on $\tau$) is the trivial homomorphism. 
In this case, \Cref{prop:beta_tau} explicitly describes the braiding on
$\cZ\left(\Vect\left[L\right]^\tau\right)$ as a modification of the braiding corresponding
to $\ev$.

Assume we are in this case, i.e. for all $\ell \in L$
the $2$-cocycle $T_\ell$ defines the trivial cohomology class. Therefore one can choose a
trivialization $t_\ell$ of $T_\ell$.
I.e. this consists of a group homomorphism
\begin{equation}
t_{\left(-\right)} \colon L \to C^1\left(L , \units{\bk}\right)
\label{eqn:t}
\end{equation}
(where $C^1$ denotes group $1$-cochains on $L$) satisfying:
\begin{equation}
d \left(t_\ell\right) =  T_\ell 
\label{eqn:dt}
\end{equation}
where $d$ denotes the differential for the group cohomology of $L$. 

\begin{prop}
Let $L$ be a finite abelian group, and let $\tau$ be a $3$-cocycle on $L$ defining the
trivial homomorphism $T_{\left(-\right)}$ in \eqref{eqn:T}. 
Let $t_{\left(-\right)}$ be as in \eqref{eqn:t} satisfying \eqref{eqn:dt}.
Then $\cZ\left(\Vect\left[L\right]^\tau\right)$ is classified as in \cite{JS}
by the pair $\left(A , q\right)$ where $A = L\dsum \pdual{L}$ and the quadratic function
is:
\begin{equation*}
q\left(\ell , \chi\right) = b_\tau \left( \left(\ell , \chi\right) , \left(\ell , \chi\right)\right)
\end{equation*}
where 
\begin{equation}
b_\tau \left(
\left(\ell_1 , \chi_1\right), \left(\ell_2 , \chi_2\right)
\right) \ceqq 
b_0\left(
\left(\ell_1 , \chi_1\right), \left(\ell_2 , \chi_2\right)
\right)
t_{\ell_1}\left(\ell_2\right) t_{\ell_2}\left(\ell_1\right)
\end{equation}
and $b_{0}$ is defined in \eqref{exm:b0}.
\label{prop:beta_tau}
\end{prop}

\begin{proof}
The $3$-cocycle $\tau$ on $L$ defines such a $3$-cocycle $a$ on $A$ by precomposing with the projection onto
the first factor $A\simeq L\dsum \pdual{L} \to L$:
\begin{equation*}
a\left(
\left(\ell_1 , \chi_1\right) , \left(\ell_2 , \chi_2\right) , \left(\ell_3 , \chi_3\right)
\right)
= \tau\left(\ell_1 , \ell_2 , \ell_3 \right) \ .
\end{equation*}
Combining this with the above discussion, it is sufficient to show that $\left(a ,
b\right)$ defines an abelian group $3$-cocycle on $A$. 
I.e. we need to show that:
\begin{align}
T_{\ell_1} \left(\ell_2 , \ell_3 \right) &= b_\tau \left(
\left(\ell_1 , \chi_1 \right) , \left(\ell_3 , \chi_3 \right)
\right)
\label{al:hex1}
\\& \qquad
\cdot b_\tau \left(
\left(\ell_1 , \chi_1 \right) , \left(\ell_2 + \ell_3 , \chi_2 \chi_3 \right)
\right)^{-1}  
\label{al:hex2}
\\& \qquad
\cdot b_\tau \left(
\left(\ell_1 , \chi_1\right) , \left(\ell_2 , \chi_2 \right)
\right)
\label{al:hex3}
\\
T_{\ell_3}\left(\ell_1 , \ell_2\right) &= b_\tau \left(
\left(\ell_1 , \chi_1 \right) , \left(\ell_3 , \chi_3 \right)
\right)^{-1}
\label{al:hex4}
\\& \qquad
\cdot b_\tau \left(
\left(\ell_1 + \ell_2 , \chi_1 \chi_2 \right) , \left( \ell_3 , \chi_3 \right)
\right)
\label{al:hex5}
\\& \qquad
\cdot b_\tau \left(
\left(\ell_2 , \chi_2\right) , \left(\ell_3 , \chi_3 \right)
\right)^{-1}
\label{al:hex6}
\end{align}
By the definition of $b_\tau$, the RHS of the first relation is:
\begin{equation*}
t_{\ell_1}\left(\ell_3\right) t_{\ell_3} \left(\ell_1\right) 
t_{\ell_1}\left(\ell_2 + \ell_3\right)^{-1} t_{\ell_2 + \ell_3}\left(\ell_1\right)^{-1} 
t_{\ell_1}\left(\ell_2\right) t_{\ell_2}\left(\ell_1\right) 
= d\left(t_{\ell_1}\right)\left(\ell_2 , \ell_3\right)
\end{equation*}
since $t_{\left(-\right)}$ is a group homomorphism.
Then the relation follows from \eqref{eqn:dt}:
\begin{equation*}
d\left(t_{\ell_1}\right)\left(\ell_2 , \ell_3\right)
= T_{\ell_1}\left(\ell_2 , \ell_3\right) \ .
\end{equation*}
The second relation follows from a similar argument.
\end{proof}

\begin{rmk}
Note that \Cref{prop:beta_tau} says that the braiding on
$\cZ\left(\Vect\left[L\right]^\tau\right)$ differs from the braiding on
$\cZ\left(\Vect\left[L\right]\right)$ (i.e. the one corresponding to $\ev$) by a 
phase written solely in terms of $t_{\left(-\right)}$.
\end{rmk}

\begin{defn}
A pointed braided fusion category is \emph{polarizable} if it is the Drinfeld center of
another fusion category $\cC$.
\end{defn}

By \Cref{prop:beta_tau}, the usual classification of pointed braided fusion categories
can be restricted to give a classification of polarizable pointed braided fusion
categories.

\begin{cor}
Pointed polarizable Drinfeld centers are classified by 
pairs $\left(A , \left[\tau , b\right]\right)$, where $\left[\tau , b\right]\in
H^3_\ab\left(A , \units{\bk}\right)$ is the class of 
an abelian group $3$-cocycle $\left(\tau , b\right)$, satisfying:
\begin{itemize}
\item $A$ is a finite abelian group such that $A\cong L\dsum \pdual{L}$ for another finite
abelian group $L$, 
\item $\tau$ is induced by a $3$-cocycle on $L$.
\end{itemize}
Equivalently, braided fusion categories which are Drinfeld centers are classified by pairs 
$\left(A , q\right)$ where $q\left(a\right) = b\left(a , a\right)$.
\label{cor:class_Z}
\end{cor}

\begin{proof}
Every \emph{pointed} polarizable braided fusion category is the Drinfeld center of a
\emph{pointed} fusion category. Pointed fusion categories are classified by finite groups
$L$ equipped with $3$-cocycles $\tau$ in the group cohomology of $L$:
they are all of the form $\Vect\left[L\right]^\tau$. 
Then the result follows from \Cref{prop:beta_tau}.

This is equivalent to the classification by $A$ equipped with $q\left(a\right) = b\left(a
,a\right)$ follows from the equivalence between $H^3_\ab\left(A , \units{\bk}\right)$ 
and quadratic forms $A \to \units{\bk}$ in 
\cite[Theorem 26.1]{EM:2} (see \Cref{sec:brfus_class}).
\end{proof}

\begin{rmk}
Unpacking \Cref{cor:class_Z}, a pointed polarizable Drinfeld center can be built from the
data of the finite abelian group $L$, a $3$-cocycle $\tau$ on $L$, and a map
\begin{equation*}
b \colon \left(L\dsum \pdual{L}\right)^2 \to \units{\bk}
\end{equation*}
satisfying \Cref{al:hex1,al:hex2,al:hex3,al:hex4,al:hex5,al:hex6}
(where $T_{\left(-\right)}$ is associated to $\tau$ as in
\eqref{eqn:T}).
The fact that $\tau$ is a $3$-cocycle encodes the pentagon axiom, and the two relations
between $T_{\left(-\right)}$ and $b$ 
(\Cref{al:hex1,al:hex2,al:hex3,al:hex4,al:hex5,al:hex6})
encode the hexagon axioms.
\end{rmk}

\subsection{Obstruction theory}
\label{sec:obstruction}

We give an executive summary of the obstruction theory developed in \cite{ENO:homotopy}.

\subsubsection{Postnikov and Whitehead towers}

Recall the definition of the Postnikov and Whitehead towers of a space, and the associated
$k$-invariants. 
See \cite{Postnikov,W:elements}, or \cite[Chapter 3]{MP:more_concise} for references.

The $3$-type $B\Aut_\Fus\left(\cC\right)\simeq  B\Pic\left(\cZ\left(\cC\right)\right)$ has
Postnikov and Whitehead towers:
\begin{equation}
\begin{tikzcd}
* \ar{r}&
B^3 \pi_3 \ar{r}\ar{d}&
B^2 \cZ\left(\cC\right)^\times \ar{r}\ar{d}&
B\Aut_{\Fus}\left(\cC\right) \ar{d}\\
& * \ar{r}&
B^2 \pi_2 \ar{r}\ar{d}&
B\Aut_{\EqBr}\left(\cZ\left(\cC\right)\right) \ar{d}\\
& & * \ar{r} & B \pi_1\ar{d}\\
&&& *
\end{tikzcd}
\label{eqn:towers}
\end{equation}
Recall the homotopy groups $\pi_i$ from \eqref{eqn:pointed_homotopy} and
\Cref{rmk:Drinfeld_map}.
See \Cref{rmk:q} for more on the space $B^2 \units{\cZ\left(\cC\right)}$.

\subsubsection{Degree three obstruction}
\label{sec:O3}

\begin{qn}
Given a group $G$ and a morphism of groups $f\colon G\to \pi_1$, what is the
obstruction to performing the following lift?
\begin{equation*}
\begin{tikzcd}
& \Aut_{\EqBr}\left(\cA\right) \ar{d} \\
G \ar{r}{f}\ar[dashed]{ur}&
\pi_1
\end{tikzcd}
\end{equation*}
\end{qn}

This question is equivalent to asking if the 
$k$-invariant $k_2 \in H^3\left(B\pi_1 , \pi_2\right)$ of the 
$2$-type $B\Aut_{\EqBr}\left(\cA\right)$ pulls back to something trivializable on $G$.
I.e. the pullback 
\begin{equation}
O_3\left(f\right) \ceqq 
\left(Bf\right)^* k_2 \in H^3\left(G , \pi_2\right)
\label{eqn:O3}
\end{equation}
is the obstruction to lifting:
\begin{equation*}
\begin{tikzcd}
& B \Aut_{\EqBr}\left(\cA\right) \ar{d} &
\\
BG \ar[swap]{r}{Bf}\ar[dashed]{ur}&
B\pi_1 = \pi_{\leq 1} B\Aut_{\EqBr}\left(\cA\right) 
\end{tikzcd}
\end{equation*}

\begin{rmk}
At the level of classifying spaces, \eqref{eqn:Pic_Eq} tells us that 
the truncation of $B\Pic\left(\cA\right)$ to a $2$-type is identified with
$B\Aut_{\EqBr}\left(\cA\right)$. 
Therefore $O_3$ is the same as the obstruction to lifting:
\begin{equation*}
\begin{tikzcd}[column sep=0pt]
& \pi_{\leq 2} B \Pic\left(\cA\right) \ar{dr}
&
\\
BG \ar{r}{Bf}\ar[dashed]{ur}&
B\pi_1 \ar[phantom]{r}{\simeq}&\pi_{\leq 1} B\Pic\left(\cA\right) 
\end{tikzcd}
\end{equation*}
\label{rmk:O3_Pic}
\end{rmk}

\begin{rmk}
When $\cA = \cZ\left(\cC\right)$ for some fusion category $\cC$, 
\eqref{eqn:BrPic_Pic} allows us to describe this obstruction in terms of the Brauer-Picard
$3$-type $B\Aut_{\Fus}\left(\cC\right)$.
In particular, it is the same as the obstruction to lifting:
\begin{equation*}
\begin{tikzcd}[column sep=0pt]
& \pi_{\leq 2} B\Aut_{\Fus}\left(\cC\right) \ar{dr} 
&
\\
BG \ar{r}{Bf} \ar[dashed]{ur} &
B\pi_1
\ar[phantom]{r}{\simeq}&
\pi_{\leq 1} B\Aut_{\Fus}\left(\cC\right)
\end{tikzcd}
\end{equation*}
\label{rmk:O3_BrPic}
\end{rmk}

\subsubsection{Degree four obstruction in the split case}
\label{sec:O4}

Assume that the obstruction
$O_3\left(f\right)$ is trivializable, and
let $s$ be the section of 
$B\Aut_{\EqBr}\left(\cZ\left(\cC\right)\right)$ over $B\pi_1$ corresponding to a fixed 
trivialization of $O_3\left(f\right)$. 
Recall the trivializations of $O_3\left(f\right)$ form a torsor over $H^2\left(BG ,
\pi_2\right)$.
By \Cref{rmk:O3_Pic}, the upshot of the existence of $s$ is that we have an associated map
$Bf_s : BG \to \pi_{\leq 2} B\Pic\left(\cA\right)$.

\begin{qn}
What is the obstruction to performing the following lift?
\begin{equation*}
\begin{tikzcd}
& B\Pic\left(\cA\right) \ar{d} \\
BG \ar{r}{Bf_s}\ar[dashed]{ur}&
\pi_{\leq 2} B\Pic\left(\cA\right)
\end{tikzcd}
\end{equation*}
\end{qn}

Just as before, the obstruction comes from the $k$-invariant. Now the relevant
$k$-invariant is
\begin{equation}
k_3 \in H^4\left(\pi_{\leq 2} B\Pic\left(\cA\right) , \pi_3\right)
\ .
\label{eqn:k3}
\end{equation}
The pullback
\begin{equation*}
O_4\left(f , s\right) \ceqq \left(Bf_s\right)^* k_3 \in H^4\left(BG ,
\pi_3\right)
\end{equation*}
is the obstruction to lifting.

\begin{rmk}
Again, if $\cA \cong \cZ\left(\cC\right)$ for some fusion category $\cC$ then, 
by \Cref{rmk:O3_BrPic}, we have an associated map $Bf_s : BG \to \pi_{\leq 2}
B\Aut_\Fus\left(\cC\right)$.
By \eqref{eqn:BrPic_Pic}, the above question is equivalent to the following.

\begin{qn}
What is the obstruction to performing the following lift?
\begin{equation*}
\begin{tikzcd}
& B\Aut_\Fus\left(\cC\right) \ar{d} \\
BG  \ar{r}{Bf_s}\ar[dashed]{ur}&
\pi_{\leq 2} B\Aut_\Fus\left(\cC\right)
\end{tikzcd}
\end{equation*}
\end{qn}

The obstruction to lifting this is still $O_4\left(f , s\right) \in H^4\left(BG ,
\pi_3\right)$.
\end{rmk}

As it turns out, e.g. from \cite[Proposition 7.2, 7.3]{ENO:homotopy}:
\begin{equation*}
\pi_3 \left(B \Aut_\Fus\left(\cC\right)\right)
\cong \units{\bk} \cong
\pi_3 \left(B \Pic\left(\cA\right)\right) \ .
\end{equation*}
Therefore the obstruction is classified by familiar $\units{\bk}$-cohomology:
\begin{equation}
O_4\left(f , s\right)
\in H^4\left(BG , \units{\bk}\right)
\ .
\label{eqn:O4}
\end{equation}

\subsubsection{Full degree four obstruction}
\label{sec:k3}

In \Cref{sec:O4}, we assume that the map into 
$\pi_{\leq 2} B\Pic\left(\cA\right)\simeq B\Aut_{\EqBr}\left(\cA\right)$ 
is given by a trivialization of $O_3\left(f\right)$ for $f$
some map from a group to $\pi_1$. 
There is a slightly more general obstruction one can consider.

\begin{qn}
Given a $2$-group $G_2$, and a map
\begin{equation*}
f_2 \colon BG_2 \to \pi_{\leq 2} B\Pic\left(\cA\right)\simeq
B\Aut_{\EqBr}\left(\cA\right) \ ,
\end{equation*}
what is the obstruction to performing the following lift?
\begin{equation*}
\begin{tikzcd}
& B\Pic\left(\cA\right) \ar{d} \\
BG_2 \ar{r}{f_2}\ar[dashed]{ur}&
\pi_{\leq 2} B\Pic\left(\cA\right)
\end{tikzcd}
\end{equation*}
\end{qn}

This is equivalent to asking if the $k$-invariant 
$k_3$ from \eqref{eqn:k3} pulls back
to something trivializable on $G_2$. 

\begin{rmk}
As in the preceding subsections, 
this is equivalent to the analogous lifting question for $B\Aut_{\Fus}\left(\cC\right)$,
by \eqref{eqn:BrPic_Pic}. 
Furthermore, a combination of \eqref{eqn:Pic_Eq} and \eqref{eqn:BrPic_Pic} tell us that
we can consider the universal version of this question:
$G_2 = B\Aut_{\EqBr}\left(\cZ\left(\cC\right)\right)$.
I.e.
\begin{equation}
k_3 \colon B\Aut_{\EqBr}\left(\cZ\left(\cC\right)\right) \to B^4 \units{\bk}
\label{eqn:k3_EqBr}
\end{equation}
is the obstruction to lifting
\begin{equation*}
\begin{tikzcd}
& B\Aut_{\Fus} \left(\cC\right) \ar{d} \\
B\Aut_{\EqBr} \left(\cZ\left(\cC\right)\right) 
\ar{r}{\sim}\ar[dashed]{ur}&
\pi_{\leq 2} B\Aut_{\Fus} \left(\cC\right) 
\end{tikzcd}
\end{equation*}
\end{rmk}

\subsubsection{Obstruction theory in the pointed case}
\label{sec:univ_O4}

Recall the homotopy groups in \eqref{eqn:pointed_homotopy}.
If we take $G = \O\left(A , q\right)$, we get a universal obstruction class:
\begin{equation*}
O_3 \in H^3\left(B\O\left(A , q\right) , A \right)
\ .
\end{equation*}

Let $\abs{A}$ be odd. Then, as discussed in \cite[\S 6]{EG:reflection} 
(and \cite[\S 5]{CGPW}), the obstruction 
$O_3$ vanishes, and in fact has a canonical splitting:
\begin{equation*}
s \colon B\O\left(A , q\right) \to \pi_{\leq 2} B\Pic\left(\cA\right) \ .
\end{equation*}
Therefore we have a well-defined universal version of the obstruction class
in \eqref{eqn:O4}:
\begin{equation}
O_4\left(A , q\right) \in H^4\left(B\O\left(A , q\right) , \units{\bk}\right)
\ .
\label{eqn:pointed_O4}
\end{equation}

\begin{rmk}
In other words, $s^* B\Aut_\Fus\left(\cC\right)$ defines a $B^3 \units{\bk}$-bundle over 
$B\O\left(A,q\right)$. 
At the level of $3$-groups, this is actually a 
\emph{central} extension of $\O\left(A , q\right)$ by $B^2 \units{\bk}$ since, for
Brauer–Picard $3$-groups of fusion categories, $\pi_1$ automatically acts trivially on
$\pi_3$.
\label{rmk:central_extension}
\end{rmk}

\begin{exm}
Let $G = \ZZ / 2$. Then $H^2\left(BG , \units{\bk}\right) = 0$ so the class $O_4$
necessarily vanishes, and so the anomaly vanishes. 
There are two nonequivalent trivializations classified by
$H^3\left(BG , \units{\bk}\right) = \ZZ / 2$.

This is used in \cite[Example 9.4]{ENO:homotopy} to reproduce the classification in
\cite{TY} of $\ZZ / 2$-graded fusion categories.
\end{exm}

\begin{rmk}
Recall that the trivializations of $O_3\left(f\right)$ form a torsor over $H^2\left(BG ,
\pi_2\right)$. 
I.e. given two trivializations $s$ and $s'$ of $O_3\left(f\right)$, there exists $L\in
H^2\left(BG , \pi_2\right)$ such that $s' = Ls$.
By \cite[Proposition 8.15]{ENO:homotopy} we have that
\begin{equation*}
O_4\left(f , s'\right) / O_4\left(f , s\right) = \PW\left(L\right) \ ,
\end{equation*}
where $\PW$ denotes the \emph{Pontrjagin-Whitehead quadratic function} from \cite[\S
8.7]{ENO:homotopy}.

Also see 
\cite[Proposition 7.3]{J:2ext} and 
\cite[Proposition 8]{CGPW} for a 
concrete formula for the  Pontrjagin-Whitehead quadratic function.
It is shown to vanish in some examples in \cite{GJ:orb}.
\end{rmk}

\subsubsection{Vanishing of the obstruction over a finite field}
\label{sec:ffield_O4}

Let $p$ be an odd prime, and let $A = V$ be a $2n$-dimensional vector space over a finite
field.
Let $q_{\splitt}$ be the quadratic form of signature $\left(n,n\right)$.

\begin{obs}
The orthogonal group $\O\left(V , q_{\splitt}\right)$ is the split orthogonal group
over $\FF_p$.
\label{obs:O}
\end{obs}

Note that, if $L$ is a vector space over $\FF_p$ of dimension $n$, then 
\begin{equation*}
\cZ\left(\Vect\left[L\right] , *\right) \simeq 
\left(\Vect\left[V\right] , * , \beta_q \right)
\ .
\end{equation*}
The obstruction $O_4\left(A , q_{\splitt}\right)$ vanishes in this case by
\cite[Theorem 6.1]{EG:reflection}.

Recall the analogy between this categorical representation of $\O\left(V , q\right)$ and
the Weil representation of the metaplectic group in \Cref{rmk:heis}. 
This vanishing can be thought as an analogue of the fact that the Weil representation
splits over a finite field \cite{GH:finite,GH:categorical}.

\subsubsection{Braided categories with a prescribed obstruction class}

\begin{prop}
For any finite group $G$ and group $4$-cocycle $\pi$ there exists a braided fusion
category $\cB$ and monoidal functor 
\begin{equation*}
\rho \colon G \to \Aut_{\EqBr}\left(\cB\right)
\end{equation*}
such that $\left[\pi\right] = O_4\left(\pi_0\comp \rho\right)$.
\label{prop:nonzero_O4}
\end{prop}

\begin{proof}
Recall that every fusion $2$-category is Morita-equivalent to a connected fusion
$2$-category \cite[Theorem 4.2.2]{Dec:Z}.
Therefore the fusion $2$-category $2\Vect\left[G\right]^\pi$ is Morita equivalent to the
$2$-category of module categories over some braided fusion category $\cB$. 

In fact, we obtain more:
As a special case of the classification of fusion $2$-categories \cite{2Fus_class}, the
connected ones are classified by
a nondegenerate braided fusion category $\cB$ equipped with a monoidal functor $\rho
\colon G\to \EqBr\left(\cB\right)$.
The fact that $O_4\left(\pi_0 \comp \rho\right) = \left[\pi\right]$ follows from the
reconstruction of a connected fusion $2$-category from $\cB$ and $\rho$, as in \cite[\S
4.4]{2Fus_class}.
\end{proof}

\subsection{Extensions of the finite orthogonal group}
\label{sec:extensions}

Recall, from \Cref{rmk:central_extension}, that we can interpret $O_4 = O_4\left(A , q\right)$ (from
\Cref{sec:univ_O4}) as encoding a central extension of $B\O\left(A , q\right)$ by $B^3
\units{\bk}$.
In particular, $\hofib\left(O_4\right)$ is a $B^3 \units{\bk}$-bundle over $B\O\left(A
, q\right)$:
\begin{equation*}
\begin{cd}
B^3 \units{\bk} \ar{d} \ar{r}&
\hofib\left(O_4\right) \ar{d} \\
* \ar{r} & B\O\left(A , q\right) 
\end{cd}
\end{equation*}

\begin{defn}
Let $\LLip\left(A , q\right)$ denote the $3$-group of loops in $\hofib\left(O_4\right)$.
\label{defn:lip}
\end{defn}

\begin{rmk}
The name comes from the analogy in \Cref{sec:analogy}. Namely, 
$\LLip\left(A , q\right)$ is a $3$-group analogue of the Lipschitz group (a.k.a. Clifford
group) of a quadratic vector space.
\end{rmk}

Let $l = \abs{L}$.
The order of the class $O_4\left(A , q\right)$ is shown in \cite[Theorem
8.16]{ENO:homotopy} to divide $l^4$.
In other words, it is the image of some class $c \in H^4\left(B\O\left(A, q\right) ,
\mu_{l^4}\right)$. 

\begin{defn}
Define the $3$ group $\PPin\left(A , q\right)$ as loops in 
$\hofib\left(c\right)$.
\label{defn:pin}
\end{defn}

We will henceforth write this class $c$ as:
\begin{equation}
c\left(\PPin\right) \in 
H^4\left(B\O\left(A , q\right) , \mu_{l^4}\right)\ .
\label{eqn:c_pin}
\end{equation}
By definition, $c\left(\PPin\right)$ classifies $B\PPin$ over $B\O\left(A , q\right)$.
We have a diagram 
\begin{equation}
\begin{tikzcd}
&
1 \ar{d} &
1 \ar{d} &
1 \ar{d} &
\\
1 \ar{r} &
B^3 \mu_{l^4}
\ar{d}\ar{r}&
B\PPin\left(A , q\right)
\ar{r}\ar{d}&
B\O\left(A , q\right)
\ar{r}\ar[equal]{d}&
1\\
1 \ar{r}&
B^3 \units{\bk}\ar{d}{B^3 \left(-\right)^{l^4}}\ar{r}&
B \LLip\left(A , q\right) \ar{r}
\ar{d}{N_{\left(A,q\right)}}
&
B\O\left(A,q\right) \ar{r}
\ar{d}{\triv}
&
1\\
1 \ar{r}&
B^3 \left(\left(\units{\bk}\right)^{l^4}\right)\ar{r}\ar{d}&
B^3 \units{\bk}\ar{r}&
B^3\left( \units{\bk} / \left(\units{\bk}\right)^{l^4}\right) \ar{r}&
1\\
&1&&&
\end{tikzcd}
\label{eqn:3_norm}
\end{equation}
where the rows and columns are exact.
The map $N_{\left(A , q\right)}$ is defined uniquely (up to homotopy) on $\pi_1$ and
$\pi_2$. On $\pi_3$ this map sends a scalar to its $l^4$-power.

Recall $\SO\left(A  ,q\right)$ from \eqref{eqn:SO}. 
The class $c\left(\PPin\right)$ in \eqref{eqn:c_pin} pulls back to a class on $\SO\left(A ,
q\right)$, which we suggestively write as
\begin{equation}
c\left(\SSpin\right) \in H^4\left(B\SO\left(A , q\right) , \mu_{l^4} \right)
\ .
\label{eqn:c_spin}
\end{equation}

\begin{defn}
Define the $3$ group $\SSpin\left(A , q\right)$ as loops in
$\hofib\left(c\left(\SSpin\right)\right)$. 
\label{defn:spin}
\end{defn}

The following follows from the definition of $\PPin$, $\SSpin$, and the universal property
of the homotopy fiber. 

\begin{prop}
The $3$ group $\SSpin\left(A , q\right)$ is the pullback:
\begin{equation*}
\begin{tikzcd}
\SSpin\left(A , q\right) \ar[dashed]{r}\ar[dashed]{d}&
\PPin\left(A , q\right) \ar{d} \\
\SO\left(A , q\right) \ar{r}&
\O\left(A , q\right)
\end{tikzcd}
\end{equation*}
\label{prop:spin}
\end{prop}

\begin{rmk}
Note that if $O_4\left(A , q\right)$ vanishes, then 
\Cref{defn:lip,defn:pin,defn:spin} are trivial (i.e. split) extensions of
$\O\left(A , q\right)$ (resp. $\SO\left(A , q\right)$).
E.g. if $A$ is a vector space over $\FF_p$ and $q$ is the split quadratic form as
discussed in \Cref{sec:ffield_O4}.
\end{rmk}

\subsection{Interlude: spinors}
\label{sec:spin}

We collect some well-known facts about the Clifford algebra and spinors.
We refer the reader to \cite{Del:spinors} for a more in-depth reference.

Let $\left(V , q\right)$ be nondegenerate quadratic vector space over a field $\bk$ with
$\chara \bk \neq 2$.
The \emph{Clifford algebra} is following quotient of the tensor algebra:
\begin{equation*}
\cliff\left(V , q\right) \ceqq T\left(V\right) / \left(v\tp v = q\left(v\right)\right)
\ .
\end{equation*}
One can calculate the anti-commutators in $\cliff\left(V , q\right)$ to be:
\begin{equation*}
uv + vu = b_q \left(u , v\right) ,
\end{equation*}
where $b_q$ is the symmetric bihomomorphism corresponding to $q$ (recall $\chara k \neq 2$).
Note that $\cliff\left(V , q\right)$ is $\ZZ / 2$-graded, with the image of $V$ being odd. 
Write $p\left(a\right)$ for the parity of a homogeneous element.
The tensor algebra has an anti-automorphism sending 
\begin{equation*}
v_1 \ext \cdots \ext v_N \mapsto v_N \ext \cdots \ext v_1
\end{equation*}
in degree $N$. The ideal we quotient out by to define $\cliff$ is preserved by this
anti-automorphism, so we obtain an antiautomorphism of $\cliff$. 
Write $a\mapsto a^T$ for this antiautomorphism.

Define the group $\Gamma$ to be the subgroup of $\units{\cliff\left(V , q\right)}$
consisting of homogeneous elements which normalize 
the copy of $V$ inside of $\cliff\left(V , q\right)$.
This is sometimes called the \emph{Clifford group} or the \emph{Lipschitz group}.
Let $g\in \Gamma$ act on $V$ by sending:
\begin{equation*}
v \mapsto \left(-1\right)^{p\left(g\right)} g v g^{-1} \ .
\end{equation*}
This defines a map $\Gamma \to \O\left(V , q\right)$.
Since $\chara k\neq 2$, $\O\left(V , q\right)$ is generated by reflections (by the
Cartan–Dieudonn\'e theorem)
so this map is
onto. The kernel is given by scalars, so we have a short exact sequence
\begin{equation*}
1 \to \units{\bk} \to \Gamma \to \O\left(V , q\right) \to 1 \ .
\end{equation*}

We cut the coefficients down to $\left\{\pm 1\right\} \inj \units{\bk}$ as follows.
The spinor norm is the map
\begin{equation*}
\begin{tikzcd}[row sep = 0pt]
\Gamma \ar{r}{N}&
\units{k} \\
g \ar[mapsto]{r}& g g^T
\end{tikzcd}
\end{equation*}
Define the group:
\begin{equation*}
\Pin\left(V , q\right) \ceqq \ker\left(N\right) \ .
\end{equation*}
There is a unique map $N_O \colon \O\left(V\right) \to \units{k} /
\left(\units{k}\right)^2$,
which is also sometimes called the spinor norm, such that 
\begin{equation*}
\begin{tikzcd}
\Gamma \ar{r} \ar{d}{N}&
\O\left(V\right) \ar{d}{N_O} \\
\units{k} \ar{r}&
\units{k} / \left(\units{k}\right)^2
\end{tikzcd}
\end{equation*}
This fits into the following diagram
\begin{equation}
\begin{tikzcd}
1 \ar{r}&
\left\{\pm 1\right\} \ar{d}\ar{r}&
\Pin\left(V,q\right) \ar{d}\ar{r}&
\ker\left(N_O\right) \ar{d}\ar{r}&
1 \\
1 \ar{r}&
\units{\bk} \ar{r}\ar{d} &
\Gamma \ar{r}\ar{d}{N}&
\O\left(V\right)\ar{r}\ar{d}{N_O}&
1 \\
1 \ar{r}&
\left(\units{\bk}\right)^2 \ar{r}&
\units{\bk}\ar{r}&
\units{\bk} / \left(\units{\bk}\right)^2 \ar{r}&
1
\end{tikzcd}
\label{eqn:norm}
\end{equation}
where the rows are short exact sequences. 
The spinor norm $N_O$ is the obstruction to 
$\Pin\left(V\right)$ being a double cover of 
$\O\left(V\right)$: 
$\Pin\left(V\right)$ is automatically a double cover of $\ker\left(N_O\right)$, but 
if $N_O$ is trivial then $\ker\left(N_O\right) = \O\left(V\right)$.

\begin{exm}
If $\bk$ is algebraically closed then $\units{\bk} / \left(\units{\bk}\right)^2$ is
trivial, so the spinor norm is necessarily trivial.
If $\bk = \RR$, then the spinor norm vanishes if $q$ is positive definite.
Therefore $\Pin\left(V\right)$ is a double cover of $\O\left(V\right)$ in these cases.
\end{exm}

\begin{exm}
Let $\dim V = 2$ with basis $\left\{e_1 , e_2\right\}$ and consider the hyperbolic form
sending $e_1 \mapsto 1$ and $e_2 \mapsto -1$.
The spinor norm evaluated on a reflection about $v\in V$ agrees with $q\left(v\right)$.
Therefore, the spinor norm $N_O$ in this case is nontrivial, and the kernel is $\ZZ / 2$,
generated by the reflection about $e_1$.
\end{exm}

Define $\Spin$ to be the restriction of this extension to $\SO \subset \O\left(V\right)$. 
I.e. the following pullback.
\begin{equation*}
\begin{tikzcd}
1 \ar{r}&
\left\{\pm 1\right\} \ar[equal]{d}\ar{r}&
\Pin\left(V\right) \ar{r}&
\O\left(V\right) \ar{r}&
1 \\
1 \ar{r}&
\left\{\pm 1\right\} \ar{r}&
\Spin\left(V\right) \ar{u}\ar{r}&
\SO\left(V\right) \ar[hook]{u}\ar{r}&
1
\end{tikzcd}
\end{equation*}

For ease of exposition, assume $\dim V = 2n$ and $\bk = \CC$ (or $\bk = \RR$ with $q$
of signature $\left(n,n\right)$)
In this case the algebra $\cliff\left(V,q\right)$ is the matrix algebra
$M_{2^n}\left(\bk\right)$, and therefore has a unique simple module $S$ up to
isomorphism.
Note that the automorphism group of the modules $S$ is $\units{\CC}$.
Automorphisms of $\cliff\left(V\right)$ only lift to automorphisms of the module $S$
up to isomorphism, i.e. up to scalar.
In particular this makes $S$ a projective $\O\left(V , q\right)$-representation.

There exists an identification $V\simeq L\dsum \pdual{L}$, 
and a model for the simple module $S$ is given by the exterior algebra 
$\ext^\dott \pdual{L}$.
The algebra $\cliff\left(V\right)$ can be identified with 
$\End\left(\ext^\dott \pdual{L}\right)$ by sending 
$\left(\ell , 0\right)\in V$ 
to $\ell \ext\left(-\right)$, and 
$\left(0 , \phi\right)\in V$ 
to the interior product $\io_{\phi}$.

\subsection{A detailed analogy with \texorpdfstring{$\Spin$}{Spin}}
\label{sec:analogy}

See \Cref{sec:spin} where we collect some basic facts and notation involving the
well-known spin-representation. 
In \Cref{tab:analogy}, we present an analogy between various objects in the theory of the
spin-representation, and various objects introduced throughout \Cref{sec:fusion}.

The Clifford algebra associated to a quadratic vector space $\left(V , q\right)$ is
analogous to the braided fusion category associated to a metric group $\left(A ,
q\right)$.
One should think that this is a double-categorification: $\EE_1$-algebras are being
replaced with $\EE_2$-categories.
The Clifford algebra can be viewed as an associative deformation of the exterior algebra:
the anticommutator in the Clifford algebra $\left\{x,y\right\} = q\left(x,y\right)$
becomes the usual multiplication on the exterior algebra (up to sign) for trivial $q$.
Similarly, we can think that the braided fusion category $\cA$ is a braided deformation of
the symmetric fusion category $\Vect\left[A\right]$  with convolution: The braiding 
\begin{equation*}
\b_q \colon \bk_a * \bk_b \lto{b_q\left(a,b\right) \id} \bk_b * \bk_a
\end{equation*}
becomes the symmetric braiding for trivial $q$.

The group $\O\left(V , q\right)$ plays the same role as $\O\left(A , q\right)$, and they
both contain their respective special orthogonal subgroups as the kernel of the
determinant. 
The ordinary group of scalars $\units{\bk}$ is replaced by the $3$-group $B^2 \units{\bk}$.
Again we have gone up two categorical levels from an ordinary group to a $3$-group.
The group of automorphisms of $\cliff\left(V\right)$ contains the affine orthogonal group
$V\rtimes \O\left(V , q\right)$.
As long as $A$ is of odd order, 
the entire $2$-group $\Aut_{\EqBr}\left(\cA\right)$ is
a semidirect product of $B A$ and $\O\left(A , q\right)$. 
This is one instance where things become easier upon categorification.
See \Cref{rmk:heis}.

The Clifford/Lipschitz group $\Gamma = \Lip\left(V , q\right)$ is a central extension of
$\O\left(V , q\right)$ by $\units{\bk}$, just as $\LLip\left(A , q\right)$ 
is a central extension of $\O\left(A ,
q\right)$ by $B^2 \units{\bk}$ (\Cref{rmk:central_extension}).
The cohomology class $O_4\left(A , q\right)$ is of order dividing $l^4$ (where $l =
\abs{L} = \sqrt{\abs{A}}$), whereas $\Gamma$ is induced by a double cover. 
I.e. $B^2 \mu_{l^4}$ plays the role of $\mu_2$.
The $3$-group $\PPin$ (resp. $\SSpin$) is the analogue of $\Pin$ (resp. $\Spin$).
This is a difference between the two sides of the analogy:
$\PPin$ and $\SSpin$ are not higher double covers (i.e. extensions by $B^2 \mu_2$) even
though $\Pin$ and $\Spin$ are double covers. 
Note that the spinor norm $N_O$ in \eqref{eqn:norm} is analogous to $N_{\left(A ,
q\right)}$ in \eqref{eqn:3_norm}.

\begin{rmk}
Another aspect of the spin group is that (in signature $\left(m , n\right)$ with either $m$
or $n$ being $\leq 1$) it is the universal cover of $\SO$. 
However, outside of these cases the $\Spin$ groups have fundamental group $\ZZ / 2$.
In particular, the split case $\Spin\left(n , n\right)$ is not simply-connected, 
and this is the case which is most directly analogous to the $3$-dimensional setting: 
$q = \ev$ (for $A$ a vector space over $\FF_p$) has signature $\left(n , n\right)$.
\end{rmk}

The module $S = \ext^\dott \pdual{L}$ is analogous to the fusion category
$\cC = \Vect\left[\pdual{L}\right]$ with convolution.
This is again a double-categorification: We have replaced $\EE_0$-algebras with
$\EE_1$-categories.
Just as $S$ is a module over $\cliff\left(V\right)$, $\cC$ is a module over $\cA$, since
$\cA \cong \cZ\left(\cC\right)$.

The $3$-group $\Pic\left(\cA\right)$ is identified with invertible $\cC$-bimodules,
\begin{equation*}
\Pic\left(\cA\right) \simeq \Aut_\Fus\left(\cC\right) \ ,
\end{equation*}
in an analogous way in which the
Clifford algebra itself is identified with the endomorphisms of spinors.
Therefore $\Gamma$ (resp. $\Spin$) acts on $S$ just as 
$\LLip$ (resp. $\SSpin$) automatically acts on $\cC$.

\begin{rmk}
This suggests that the analogue of $\Pic\left(\cA\right)$ should be the invertible part of
the Clifford algebra. This is not entirely the case:
an important fact about $\Pic\left(\cA\right)$ is that it is an extension of 
$\Aut_{\EqBr}\left(\cA\right)$ by $B^2 \units{\bk}$.
There is a more analogous object in the setting of metaplectic quantization, rather than
spin quantization. 
There is a well-known analogy between these two settings:
instead of starting with a quadratic vector space, we start with a symplectic vector
space.
The Clifford algebra is replaced by the Weyl algebra, and the analogue of the $\Spin$
double-cover of $\SO$ is the double cover $\Mp$ of $\Sp$.
See \cite{Del:spinors} for a more in-depth explanation of the analogy. 

The analogue of $\Aut_{\EqBr}\left(\cA\right)$ in the metaplectic setting is what is
sometimes called the \emph{affine symplectic group}:
\begin{equation*}
A\Sp\left(V , \om\right) = V\rtimes \Sp\left(V , \om\right)
\ .
\end{equation*}
Then the \emph{extended symplectic group} $E\Sp$ is an extension of $A\Sp$ by
$\U\left(1\right)$, and in fact 
\begin{equation*}
E\Sp\left(V , \om\right) = \heis \rtimes \Mp\left(V , \om\right)
\ ,
\end{equation*}
where $\heis$ is the Heisenberg Lie group associated to $\left(V , \om\right)$.
I.e. we have a diagram:
\begin{equation*}
\begin{tikzcd}
\U\left(1\right) \ar{r}&
\heis \ar{r}\ar{d}&
V \ar{d} \\
\U\left(1\right) \ar{r}&
E\Sp\left(V , \om\right) \ar{r}\ar{d}&
A\Sp\left(V , \om\right) \ar{d} \\
\ZZ / 2 \ar{r}&
\Mp
\ar{r}&
\Sp
\end{tikzcd}
\end{equation*}
with exact rows and columns.
This is analogous to the diagram:
\begin{equation*}
\begin{tikzcd}
B^3 \units{\bk} \ar{r}&
B^2 \units{\cA} \ar{r}\ar{d}&
B^2 A \ar{d} \\
B^3 \units{\bk}\ar{r}&
B \Pic\left(\cA\right)\ar{r}\ar{d}&
B\Aut_{\EqBr}\left(\cA\right) \ar{d} \\
B^3 \mu_{l^4} \ar{r}&
B\PPin \ar{r}&
B\O\left(A , q\right)
\end{tikzcd}
\end{equation*}

There is an analogue of the Heisenberg Lie algebra in the $\Spin$ setting: It is the super
Lie algebra generated by $V$ with super-bracket determined by the quadratic form. 
The analogue of $\heis$ should be some super group integrating this super Lie algebra.

Note that, even though $V\rtimes \Sp\left(V\right)$ acts on the Weyl algebra, it is not
necessarily the whole automorphism group. 
I.e. see \cite{BKK:weyl} where the question of identifying the full automorphism group is
discussed.
\label{rmk:heis}
\end{rmk}

\section{Projective \texorpdfstring{$3$}{3}-dimensional TQFTs}
\label{sec:3d_theories}

\subsection{Projective \texorpdfstring{$3$d}{three-dimensional} TQFTs}
\label{sec:3d_proj}

In order to describe the projectivity captured by the anomalies in
\Cref{thm:k3,thm:pin_anom,cor:spin_anom} in terms of projective theories, 
as in \Cref{thm:anom_symm,cor:anom_symm}, 
we consider the projectivization (\Cref{def:P}) of $\Fus$ discussed in \Cref{sec:PFus}.

\begin{defn}
Projective $3$-dimensional TQFTs with tangential structure $\left(X , \z\right)$ 
are functors 
\begin{equation*}
\Bord_3^{\left(X , \z\right)} \to \PP \Tens \ .
\end{equation*}
We will primarily restrict our attention to theories factoring through the subcategory
$\PP\Fus$.
The \emph{projectivity} of $\overline{F}$ is the theory $\al$:
\begin{equation*}
\begin{cd}
\Bord_3 \ar{r}{\overline{F}}\ar[bend right]{rr}{\al}&
\PP \Fus \ar{r}{S}&
\units{\BrFus}
\ .
\end{cd}
\end{equation*}
\end{defn}

Let $\cC$ be an object of the Morita $3$-category of fusion categories. 
As in \eqref{eqn:C}, this classifies a 
framed theory $F \colon \pt \mapsto \cC$.

\begin{cor}
As in \Cref{thm:k3}, write 
$X = B\Aut_{\EqBr}\left(\cZ\left(\cC\right)\right)$, and let $k_3$ be the cocycle in
\eqref{eqn:k3}.

The anomalous theory $F_{k_3}$ in 
\eqref{eqn:Fk3} is equivalent to some projective theory:
\begin{equation*}
\overline{F} \colon \Bord_3^{X} \to \PP\Fus 
\end{equation*}
with underlying theory $F$, in the sense of 
\Cref{defn:underlying}.
\label{cor:proj_k3}
\end{cor}

\Cref{cor:proj_k3} follows from \Cref{cor:anom_symm}\ref{cor_proj} in the context of  \Cref{thm:k3}.

Let $c\left(\PPin\right)$ and $c\left(\SSpin\right)$ be the cocycles defined in 
\eqref{eqn:c_pin} and \eqref{eqn:c_spin}, which classify the anomaly theories
$\al_{c\left(\PPin\right)}$ and $\al_{c\left(\SSpin\right)}$ from
\eqref{eqn:anom_pin} and \eqref{eqn:anom_spin}.

\begin{cor}
Let $\cZ\left(\cA\right)$ be pointed, with underlying pre-metric group $\left(A ,
q\right)$.
The anomalous theory $F_{c\left(\PPin\right)}$ from \eqref{eqn:F_pin} is equivalent to
some projective theory:
\begin{equation*}
\overline{F} \colon \Bord_3^{B\O\left(A, q\right)} \to \PP\Fus 
\end{equation*}
with underlying theory $F$, in the sense of 
\Cref{defn:underlying}.

The same result holds when $F_{c\left(\PPin\right)}$ is replaced with
$F_{c\left(\SSpin\right)}$ from \eqref{eqn:F_spin}, and $\O\left(A\right)$ is replaced
with $\SO\left(A\right)$.
\label{cor:proj_pin}
\end{cor}

\Cref{cor:proj_pin} follows from \Cref{thm:anom_symm}\ref{proj} in the context of
\Cref{thm:pin_anom,cor:spin_anom}.

As in the general discussion in \Cref{sec:proj} (namely \Cref{prop:linearize} and
\Cref{defn:projective_TQFT})
a trivialization of any of the anomaly theories $\al_{k_3}$, $\al_{c\left(\PPin\right)}$, or
$\al_{c\left(\SSpin\right)}$
determines a linearization of $\overline{F}$, which tautologically agrees with the
trivialized anomalous theory $1\lto{\sim} \al \lto{F_c} 1$.
This is \Cref{thm:anom_symm}\ref{triv_proj} (or \Cref{cor:anom_symm}\ref{triv_cor_proj})
in this context. 

On the other hand, nontrivial cocycles on finite groups \emph{always} describe anomalies
of some $3$-dimensional TQFT by the following. 

\begin{cor}
For any nontrivial $\pi \colon BG \to B^4 \units{\bk}$,
there exists a $\pi$-finite space $X$ and nontrivial anomaly TQFT 
\begin{equation*}
\al_\pi \colon \Bord^{X}_4 \to \BrFus 
\end{equation*}
such that there exists a map $f \colon BG\to X$ such that 
the pullback of the cocycle classifying $\al$ 
along $f$ agrees with the cohomology class of $\pi$. 

Furthermore, there exists a nondegenerate braided fusion category $\cB$ such that
$\al_\pi$ is an $X$-anomaly of the Reshetikhin-Turaev theory associated to $\cB$.
\label{cor:nontriv_O4anomaly}
\end{cor}

\begin{proof}
\Cref{prop:nonzero_O4} provides us with a braided fusion category $\cB$ associated to
$\pi$. 
Setting the $\pi$-finite space $X$ to be $B\Aut_{\EqBr}\left(\cB\right)$, 
the result follows from \Cref{prop:nonzero_O4}.
\end{proof}

\subsection{Dualizability and invertibility}
\label{sec:dual_inv}

It was shown in \cite{DSPS:dualizable} that 
the category of fusion categories $\Fus$ has duals.\footnote{Technically our definition of
$\Fus$ differs from the setting of finite semisimple abelian categories by an
$\Ind$-completion. However, as discussed in \Cref{rmk:fusion}, we can nontheless apply the
theorems of \cite{DSPS:dualizable,ENO:homotopy} in our context.}
Therefore the Cobordism Hypothesis tells us that, given a fusion category $\cC$, 
there is a uniquely defined fully-extended topological field theory 
\begin{equation}
\begin{tikzcd}[row sep=-2pt]
\Bord_3^\fra \ar{r}{F}&
\Fus \\
\pt \ar[mapsto]{r}&
\cC 
\end{tikzcd}
\label{eqn:C}
\end{equation}
This is sometimes called the 
\emph{Turaev-Viro theory} associated to a fusion category \cite{TV,FT:gapped}.

It was shown in \cite{BJS:dualizable} that the category of braided fusion categories
$\BrFus$ has duals.
Therefore, given a braided fusion category $\cA$, the Cobordism Hypothesis
provides a fully-extended topological field theory
$\Bord_4^\fra \to \BrFus$ sending the point to $\cA$. 
This is a framed version of the \emph{Crane-Yetter theory} associated to $\cA$
\cite{CY:4d_TQFT,FT:gapped}.

\begin{rmk}
For ease of discussion, let $\cA$ be the braided fusion category attached 
to a metric group $\left(A , q\right)$. 
Consider the space of invertible objects in $\cZ$, written $\cZ^\times$.
The fact that $\cZ$ was braided means we are able to deloop twice, to obtain a $3$-type
$B^2 \units{\cA}$.
The homotopy groups are all trivial except $\pi_2 = A$ and $\pi_3 = \units{\bk}$.
The only remaining information needed to specify this space is the $k$-invariant in 
$H^4\left(B^2 A , \units{\bk}\right)$.
It is a theorem of Eilenberg-MacLane 
\cite[Theorem 26.1]{EM:2}
that quadratic forms $q \colon A \to  \units{\bk}$ are classified by 
$\tau_q \in H^4\left(B^2 A , \units{\bk}\right)$.

Assuming \Cref{hyp:sigma}, we can construct the theory associated to the $\pi$-finite
space $B^2 A$ twisted by a cocycle representing $\tau_q$, which turns out to be precisely
the framed fully-extended Crane-Yetter theory from above sending the point to $\cA$.
This is discussed in \Cref{exm:sigma_BG} \ref{sigma_B2G}.
\label{rmk:q}
\end{rmk}

\subsubsection{Invertibitlity}
\label{sec:inv_3}

Invertibility of a braided tensor category was shown in \cite{BJSS:invertible} to be
equivalent to checking three conditions. 
When the braided tensor category is fusion, these conditions become equivalent so we only
have to check one. One of the conditions, \emph{non-degeneracy}, asks if the 
\emph{M\"uger-center}\footnote{The M\"uger center of a braided tensor category consists
of the objects which braid trivially with all other objects.} is trivial.

If $\cC$ is a fusion category such that $\cA \cong \cZ\left(\cC\right)$, then 
the M\"uger-center is known to be trivial \cite{Mug:Z2Z1,DGNO}, so 
the TQFT sending the point to $\cZ\left(\cC\right)$ is in fact invertible.

\subsection{WRT theories}
\label{sec:WRT}

Recall the description of the projective target category $\PP\Fus$ from \Cref{sec:PFus},
and the corresponding notion of a projective $3$-dimensional TQFT from \Cref{sec:3d_proj}.
Let $\cA$ be any nondegenerate braided fusion category. 
Note that the regular module $\cA_\cA$ defines a $1$-morphism in $\BrFus$ from $\cA$ to
$1$, and it is fully-dualizable as an object of the arrow category \cite{Ben:unit}.
Since $\cA$ is fusion and nondegenerate, it is also an invertible object of $\BrFus$
\cite{BJSS:invertible}.
In other words, the $1$-morphism $\cA_\cA \colon \cA \to 1$ is in the subcategory $\PP\Fus
\subset \BrFus^{\down 1}$.

Summarizing this discussion, the cobordism hypothesis implies that $\cA$ defines 
a fully-extended projective TQFT:
\begin{equation*}
\WRT_\cA \colon \Bord^\fra_3 \to \PP \Fus \ .
\end{equation*}
This is fully-extended anomalous Witten-Reshetikhin-Turaev (WRT)\footnote{
Oftentimes WRT theory denotes the linear TQFT defined out of the extended bordism category 
as in \Cref{rmk:resolve_anom}. As is explained there, these are equivalent.}
theory, with a caveat: it is
not oriented. This would be the data of a homotopy $\SO\left(3\right)$-fixed point
structure on this object of $\PP\Fus$.

The anomaly of $\WRT_{\cA}$ is, by definition, the TQFT:
\begin{equation*}
\CY_{\cA} \colon \Bord^\fra_4 \to \BrFus 
\end{equation*}
which is classified by sending the point to $\cA$. 
Recall this is a framed version of the \emph{Crane-Yetter theory} associated to $\cA$
\cite{CY:4d_TQFT,FT:gapped}.

As it turns out, $\WRT_\cA$ has a canonical projective action of the $2$-group
of braided autoequivalences of $\cA$.

\begin{theorem}
Let $\cA$ be a nondegenerate braided fusion category, and consider a monoidal functor
$\Phi \colon G\to \Aut_{\EqBr}\left(\cA\right)$. 
There is an anomalous TQFT 
\begin{equation*}
\WRT_\cA^G \colon \Bord^{BG}_3 \to \PP \Fus 
\end{equation*}
which agrees with $\WRT_\cA$ upon restriction to trivial $G$-bundles:
\begin{equation*}
\begin{tikzcd}
\Bord^{BG} \ar{dr}{\WRT_\cA^G} &
\\
\Bord^\fra\ar{u}\ar{r}{\WRT_\cA}
& \PP\Fus
\end{tikzcd}
\end{equation*}

Furthermore, the anomaly theory of $\WRT_\cA^G$, 
\begin{equation*}
\CY_\cA^G \colon \Bord_4^{BG} \to \BrFus 
\ ,
\end{equation*}
agrees with $\CY_\cA$ upon restriction to trivial $G$-bundles.
\label{thm:WRT}
\end{theorem}

\begin{proof}
Symmetric monoidal functors from $\Bord_3^{B\Aut_{\EqBr}\left(\cA\right)}$ to $\PP\Fus$
which restrict to $\WRT_\cA$ are classified by maps:
\begin{equation*}
B\Aut_{\EqBr}\left(\cA\right) \to B\Aut_{\PP\Fus} \left(1 \lto{\cA_\cA} \cA\right) \ .
\end{equation*}
There is a canonical functor:
\begin{equation*}
\Aut_{\EqBr}\left(\cA\right) \to \Aut_{\BrFus}\left(\cA\right)
\end{equation*}
given by sending a braided autoequivalence $\varphi$ to the identity bimodule twisted by
$\varphi$, written ${}_\cA \cA_{\varphi\left(\cA\right)}$.
Precomposing with the given monoidal functor $\Phi \colon G \to
\Aut_{\EqBr}\left(\cA\right)$ defines an action:
\begin{equation*}
\overline{\Phi} \colon G \to \Aut_{\BrFus}\left(\cA\right) \ .
\end{equation*}

Now notice that $\Phi$ and $\overline{\Phi}$ assemble together into a functor
\begin{equation}
G \to \Aut_{\PP \Fus}\left(\cA_\cA\right)
\end{equation}
which on objects sends $g\in G$ to the lax square
\begin{equation*}
\begin{tikzcd}
\cA \ar{d}{\cA_\cA} \ar{r}{\overline{\Phi}\left(g\right)}&
\cA \ar{d}{\cA_\cA} \\
\Vect \ar[equal]{r}\ar{ur}{\Phi\left(g\right)} &
\Vect
\end{tikzcd}
\end{equation*}
Note that this classifies a functor
\begin{equation*}
\WRT_\cA^G \colon \Bord^BG_3 \to \PP \Fus \ ,
\end{equation*}
and its restriction to trivial bundles sends the point to the regular module, and
therefore agrees with $\WRT_\cA$ by the cobordism hypothesis.
\end{proof}

\begin{rmk}
Since $\cA$ is nondegenerate, it is invertible in the $4$-category $\BrFus$, and therefore
we have equivalences:
\begin{equation*}
\Aut_{\BrFus}\left(\cA\right) \simeq \Aut_{\BrFus}\left(1 \right)\simeq \units{\Fus}
\simeq B^3 \units{\bk} \ .
\end{equation*}
The delooping of this map pulls back along the monoidal functor $G\to
\Aut_{\EqBr}\left(\cA\right)$ to a $4$-cocycle on $G$, which classifies the anomaly theory
$\CY_\cA^G$ as a functor out of $\Bord^{BG}_4$.
\end{rmk}

\begin{rmk}
Recall from \Cref{rmk:resolve_anom} that, given an anomalous TQFT, there is an extended
bordism category over which the anomaly can always be trivialized. 
As is explained in \cite{Wal:TQFT,Fre:aspects,BDSPV,Ben:WRT_CY}, 
if you lift from oriented $4$-dimensional bordism category to the bordism category of
oriented $4$-manifolds with $p_1$-structure, then the Crane-Yetter theory can be
trivialized. 
Another option is equipping the $4$-manifolds with a signature structure. 
Choices of the latter structure correspond to square roots of the central charge, and the
choices of the former correspond to sixth roots, so these relate to one another by a
factor of three \cite{BDSPV}.
\label{rmk:WRT_p1}
\end{rmk}

\begin{rmk}
If the anomaly theory happens to be trivializable as a functor out of $\Bord^{BG}$, then
the trivializations $1\lto{\sim} \CY^G$ will form a torsor over invertible theories of one
dimension lower. We will save a more detailed discussion of trivializations for the more
restricted examples of $3$-dimensional theories in the coming sections. 
\end{rmk}

\subsection{The anomaly and the center}
\label{sec:anom_Z}

Consider the following theory $\z$.
\begin{equation*}
\begin{tikzcd}[row sep=-2pt]
\Bord_4^\fra \ar{r}{\zeta}&
\BrFus \\
\pt \ar[mapsto]{r}&
\cZ\left(\cC\right)
\end{tikzcd}
\end{equation*}
Note that $\cC$ always defines a module ${}_{\cZ} \cC$ over its own center $\cZ\left(\cC\right)$.
By \cite[Theorem 7.15]{JFS}, this morphism ${}_{\cZ}\cC \colon 1 \to \cZ\left(\cC\right)$
in $\BrFus$ classifies a relative theory, 
i.e. a lax natural transformation (\Cref{rmk:lax}):
\begin{equation*}
F_\z \colon 1 \to \zeta \ .
\end{equation*}
By \Cref{sec:inv_3}, $\zeta$ is invertible, so we can think of $F_\z$ as an anomalous
theory. 

In fact, the Drinfeld center of a fusion category $\cC$ is trivial in $\BrFus$:
$\cC$ defines an equivalence between $1$ and $\cZ\left(\cC\right)$.
See the proof of \cite[Theorem 4.2]{BJSS:invertible}. 
For a general braided fusion category $\cA$, an $\cA$-central fusion category $\cC$ was
shown to give an equivalence between $\cA$ and $\Vect$ if and only if the natural map $\cA
\to \cZ\left(\cC\right)$ is an equivalence in \cite[Theorem 2.23]{JMPP}.

\begin{rmk}
Recall $\Pic\left(\cZ\left(\cC\right)\right)$ from \Cref{sec:higher_groupoids}.
Note that these modules are not $1$-morphisms in $\BrFus$ (they are only modules, not
central modules).
They are however $2$-morphisms in $\BrFus$.
In the language of defects, this is saying that $\Pic$ does not consist of domain walls,
but rather certain codimension-$2$ defects. 
In particular, consider a domain wall between the trivial theory and $\zeta$ labelled by
$\cZ\left(\cC\right)$ as a module over itself. 
Now the objects of $\Pic$ label self-interfaces between this domain wall.
\label{rmk:defects}
\end{rmk}

\subsection{Projective action of the braided automorphisms of the center}
\label{sec:anom_k3}

Fix a TQFT $F$ associated to $\cC \in \Fus$ as in \eqref{eqn:C}.
We will apply \Cref{cor:anom_symm} to this theory. 

Recall the notation from \Cref{sec:k3}. In particular, recall the class 
from \eqref{eqn:k3_EqBr}:
\begin{equation*}
k_3 \colon B\Aut_{\EqBr}\left(\cZ\left(\cC\right)\right) \to B^4 \units{\bk} \ .
\end{equation*}
In the notation of \Cref{sec:proj_anom}, we will let $X =
B\Aut_{\EqBr}\left(\cZ\left(\cC\right)\right)$ and $k = k_3$, so
\begin{equation*}
\widetilde{X} = B\Aut_{\Fus}\left(\cC\right) \ .
\end{equation*}
The upshot of this is that $\cC$ is a Schur object (\Cref{defn:ss}), and therefore
\Cref{cor:anom_symm} applies.

By the general discussion of \Cref{sec:proj_anom}, there is an anomaly theory, i.e. a once
categorified $3$-dimensional TQFT
\begin{equation*}
\al_{k_3} \colon \Bord^{B\Aut\left(\cZ\left(\cC\right)\right)}_3 \to 
\BrFus \ ,
\end{equation*}
classified by $k_3$.

\begin{rmk}
This is a situation where the anomaly theory $\al_{k_3}$ actually extends to a $\left(d+1\right)  =
4$-dimensional theory, i.e. a functor out of
$\Bord_4^{B\Aut\left(\cZ\left(\cC\right)\right)}$, as is discussed in \Cref{rmk:alt_anom}. 
\end{rmk}

\begin{theorem}
The framed TQFT $F$ has an $X= B\Aut_{\EqBr} \left(\cZ\left(\cC\right)\right)$-anomaly as in \Cref{defn:X_anom}. 
In particular, the fusion category $\cC$ itself defines an anomalous theory:
\begin{equation}
F_{k_3} \colon \alpha_{k_3} \to 1_{B\Aut\left(\cZ\left(\cC\right)\right)} \ .
\label{eqn:Fk3}
\end{equation}
Furthermore, if $\cZ\left(\cC\right)$ is pointed, and the cohomology class classifying the braiding of
$\cZ\left(\cC\right)$ (as in \Cref{rmk:q}) is nontrivial, then the anomaly $\al_{k_3}$ is nontrivial. 
\label{thm:k3}
\end{theorem}

\begin{proof}
As remarked above, $\cC$ is a Schur object (\Cref{defn:ss}) of $\Fus$, so 
\Cref{cor:anom_symm} applies.
The fact that $F \colon \pt \mapsto \cC$ has an $X$-anomaly is then 
\Cref{cor:anom_symm}\ref{cor_anom}.

Recall from \Cref{rmk:q} that the class classifying the braiding is the same as the
$k$-invariant of $B^2\units{\cZ}$.
Also recall the Postnikov and Whitehead towers of $B\Aut_\Fus\left(\cC\right)$
from \eqref{eqn:towers}. The composition of $k_3$ with the map from the universal cover
agrees with the $k$-invariant of $B^2\units{\cZ}$ over $B^2 \pi_2$:
\begin{equation*}
\begin{tikzcd}
B^2 \cZ\left(\cC\right)^\times \ar{r}\ar{d}&
B\Aut_{\Fus}\left(\cC\right) \ar{d}&\\
B^2 \pi_2 \ar{r} \ar[swap]{drr}{\tau_q}&
B\Aut_{\EqBr}\left(\cZ\left(\cC\right)\right) \ar{dr}{k_3}&\\
&& B^4 \units{\bk}
\end{tikzcd}
\end{equation*}
where we write $\tau_q$ for the class corresponding to $q$ as in \Cref{rmk:q}.
From this diagram we see that trivial $k_3$ implies trivial $\tau_q$.
Therefore if $\tau_q$ is nontrivial, then $k_3$ is nontrivial, so $\al_{k_3}$ is
nontrivial. 
\end{proof}

\subsubsection{Trivializing the anomaly}

Recall the setting of \Cref{sec:k3}.
If we can trivialize $k_3$ over 
\begin{equation*}
f_2 \colon G_2 \to
B\Aut_{\EqBr}\left(\cZ\left(\cC\right)\right)\ ,
\end{equation*}
then the trivializations form a torsor
over $H^3\left(BG_2 , \units{\bk}\right)$, and a single trivialization determines
\begin{equation*}
1_{BG_2} \lto{\sim} \al_{f_2 \comp c} \to 1_{BG_2} \ .
\end{equation*}
which is a (non-anomalous) $G_2$-theory by \Cref{prop:1_X}. This is 
\Cref{cor:anom_symm}\ref{triv_cor_anom}. The equivalence with
\Cref{cor:anom_symm}\ref{triv_cor_proj} and \ref{triv_cor_symm} in this example will be
discussed in \Cref{sec:3d_proj,sec:3d_symm}.

\subsection{Projective action of the orthogonal group}
\label{sec:anom_O4}

Still fix $F \colon \pt \mapsto \cC$ \eqref{eqn:C}, 
as in \Cref{sec:dual_inv,sec:anom_Z,sec:anom_k3}.
Although now we will assume that we are in the setting of \Cref{sec:O4}.
I.e. we are assuming that the obstruction $O_3\left(f\right)$ from \eqref{eqn:O3},
associated to a map
\begin{equation*}
f \colon G\to \pi_1 B\Aut_\Fus\left(\cC\right) \ ,
\end{equation*}
is trivializable, and that we have chosen a trivialization $s$. 
Recall the class 
\begin{equation*}
O_4\left(f , s\right) \in H^4\left(BG , \units{\bk}\right)
\end{equation*}
from \eqref{eqn:O4}.

We will apply \Cref{thm:anom_symm} to this setting. 
I.e. let $\cT = \BrFus$, as in \Cref{sec:anom_k3}.
Set $X = BG$ and $c = O_4\left(f , s\right)$.
By the general discussion of \Cref{sec:proj_anom}, there is an anomaly theory
\begin{equation}
\al_{O_4\left(f,s\right)} \colon \Bord^{BG}_3 \to \BrFus
\label{eqn:anom_O4fs}
\end{equation}
classified by $O_4\left(f,s\right)$.

By definition, the space classified by $O_4\left(f,s\right)$ is the pullback:
\begin{equation*}
\begin{tikzcd}
\hofib\left(O_4\left(f,s\right)\right) \ar{r}\ar{d}&
B\Aut_\Fus\left(\cC\right) \ar{d} \\
BG \ar{r}{s}&
\pi_{\leq 2} B\Aut_{\Fus}\left(\cC\right)
\end{tikzcd}
\end{equation*}
The inclusion of the connected component corresponding to $\cC$ defines a map:
\begin{equation*}
B\Aut_{\Fus} \left(\cC\right) \to \Fus
\ ,
\end{equation*}
so composition with the map defined by the pullback gives us:
\begin{equation*}
\hofib\left(O_4\left(f,s\right)\right)
\to B\Aut_\Fus\left(\cC\right) \to \Fus \ ,
\end{equation*}
which classifies a TQFT:
\begin{equation*}
\Bord^{\hofib\left(O_4\left(f,s\right)\right)}_3 \to \Fus \ .
\end{equation*}
This theory is the extended TQFT associated to the anomalous theory 
\begin{equation*}
F_{O_4\left(f,s\right)} \colon \al_{O_4\left(f,s\right)} \to 1_{BG}
\end{equation*}
by \Cref{prop:alpha_c}.

\subsubsection{Pointed center}
\label{sec:anom_pointed}

Continuing the investigation of TQFTs of the form \eqref{eqn:C}, we now assume that
$\cZ\left(\cC\right)$ is pointed. 
As in \Cref{cor:class_Z}, this is classified by its group of simple objects and a
quadratic form $q\colon A \to \units{\bk}$ which describes the braiding.

Recall, from \Cref{sec:univ_O4}, that as long as $\abs{A}$ is odd, there is 
a canonical splitting $s$ of $B\Aut_{\EqBr}\left(\cZ\left(\cC\right)\right)$. 
Therefore there is a universal class defined in \eqref{eqn:pointed_O4}:
\begin{equation*}
O_4 = 
O_4\left(A , q\right) \in H^4\left(B\O\left( A , q\right) , \units{\bk}\right)
\ .
\end{equation*}
Write $\O\left( A \right) = \O\left( A , q\right)$ for
brevity. 

By the general discussion in \Cref{sec:proj_anom}, there is an anomaly classified by
$O_4$:
\begin{equation}
\al_{O_4} \colon \Bord^{B\O\left(A\right)} \to \BrFus \ .
\label{eqn:anom_O4}
\end{equation}

Recall the $3$-group $\LLip\left(A\right)$ defined in \Cref{sec:extensions}.
By definition, the space classified by $O_4$ is the pullback:
\begin{equation*}
\begin{tikzcd}
B\LLip \left(A\right)= 
\hofib\left(O_4\right) \ar{r}\ar{d}&
B\Aut_\Fus\left(\cC\right) \ar{d} \\
B\O\left(A \right) \ar{r}{s}&
\pi_{\leq 2} B\Aut_{\Fus}\left(\cC\right)
\end{tikzcd}
\end{equation*}
where $s$ is the canonical splitting 
\begin{equation*}
s \colon 
B\O\left(A\right) \to 
B\Aut_{\EqBr}\left(\cZ\left(\cC\right)\right)\simeq \pi_{\leq 2}
B\Aut_{\Fus}\left(\cC\right)
\ .
\end{equation*}
The inclusion of the connected component corresponding to $\cC$ defines a map:
\begin{equation*}
B\Aut_{\Fus} \left(\cC\right) \to \Fus
\ ,
\end{equation*}
so composition gives us a map:
\begin{equation}
B\LLip\left(A\right)
\to B\Aut_\Fus\left(\cC\right) \to \Fus \ ,
\label{eqn:class_anom_lip}
\end{equation}
which classifies a TQFT:
\begin{equation*}
F_{\LLip} \colon \Bord^{B\LLip\left(A\right)}_3 \to \Fus
\end{equation*}
by the Cobordism Hypothesis \eqref{eqn:CH_X}. 
This is the extended theory corresponding to the anomalous theory
\begin{equation*}
F_{O_4} \colon \al_{O_4} \to 1_{BG}
\end{equation*}
by \Cref{prop:alpha_c}.

The same holds when we restrict to the $3$-groups $\PPin\left(A\right)$ and $\SSpin\left(A\right)$ defined in
\Cref{sec:extensions} as follows.
In particular, recall the cohomology classes \eqref{eqn:c_pin} and \eqref{eqn:c_spin}:
\begin{align*}
c\left(\PPin\right) \in 
H^4\left(B\O\left(A , q\right) , \mu_{l^4}\right)
&&
c\left(\SSpin\right) \in H^4\left(B\SO\left(A , q\right) , \mu_{l^4} \right) \ .
\end{align*}
These classify anomaly theories, by the general discussion in \Cref{sec:proj_anom}:
\begin{align}
\al_{c\left(\PPin\right)} &\colon \Bord^{B\O\left(A\right)}_3 \to \BrFus 
\label{eqn:anom_pin}\\
\al_{c\left(\SSpin\right)} &\colon \Bord^{B\SO\left(A\right)}_3 \to \BrFus 
\label{eqn:anom_spin}
\ .
\end{align}
We have a map, e.g. from \eqref{eqn:3_norm}, $\PPin \left(A\right)\to \LLip\left(A\right)$.
We also have a map $\SSpin \left(A\right)\to \PPin\left(A\right)$ from \Cref{prop:spin}, which we can compose with the
map $\PPin \left(A\right)\to \LLip\left(A\right)$ to obtain a map $\SSpin
\left(A\right)\to \LLip\left(A\right)$. 
The upshot of this, is that we can compose these maps with \eqref{eqn:class_anom_lip} to
obtain functors:
\begin{align*}
\PPin \left(A\right)\to \Fus && \SSpin \left(A\right)\to \Fus
\end{align*}
which classify a $\PPin\left(A\right)$-TQFT and a $\SSpin\left(A\right)$-TQFT:
\begin{align*}
F_{\PPin} &\colon \Bord^{B\PPin\left(A\right)} \to \Fus  \\
F_{\SSpin} &\colon \Bord^{B\SSpin\left(A\right)} \to \Fus  
\end{align*}
by the Cobordism Hypothesis \eqref{eqn:CH_X}. 
These are the extended theories associated to the anomalous theories:
\begin{align*}
F_{c\left(\PPin\right)} &\colon \al_{c\left(\PPin\right)} \to 1_{\PPin\left(A\right)}  \\
F_{c\left(\SSpin\right)} &\colon \al_{c\left(\SSpin\right)} \to 1_{\SSpin\left(A\right)}
\end{align*}
by \Cref{prop:alpha_c},
where $1_{\PPin\left(A\right)}$ (resp. $1_{\SSpin\left(A\right)}$) denotes the trivial $\Fus$-valued TQFT defined on
$\Bord^{B\PPin\left(A\right)}$ (resp. $\Bord_3^{B\SSpin\left(A\right)}$).

Summarizing this discussion, we have the following. 

\begin{theorem}
Let $\cC$ be a fusion category with pointed Drinfeld center.
Recall this is classified by a polarized
metric group $\left(A , q , L\right)$ as in \Cref{cor:class_Z}.
The nonanomalous framed theory 
\begin{equation*}
F \colon \Bord^{\fra}_3 \to \Fus 
\end{equation*}
sending the point to $\cC$ has an $\O\left(A , q\right)$-anomaly in the sense of
\Cref{defn:X_anom}.
I.e. there is an anomaly theory
\begin{equation*}
\al_{c\left(\PPin\right)} \colon\Bord^{B\O\left(A , q\right)}_3 \to \BrFus \ ,
\end{equation*}
and $\cC$ canonically defines an anomalous $\O\left(A , q\right)$-TQFT:
\begin{equation}
F_{c\left(\PPin\right)} \colon \al_{c\left(\PPin\right)} \to 1_{B\O\left(A , q\right)} \ .
\label{eqn:F_pin}
\end{equation}
\label{thm:pin_anom}
\end{theorem}

\begin{rmk}
As in \Cref{sec:anom_k3}, the anomalies $\al_{O_4\left(f,s\right)}$, and therefore
$\al_{O_4}$, $\al_{c\left(\PPin\left(A\right)\right)}$, and $\al_{c\left(\SSpin\right)}$, 
defined in \Cref{eqn:anom_O4fs,eqn:anom_O4,eqn:anom_pin,eqn:anom_spin} extend to 
$\left(d+1\right)  = 4$-dimensional theories, i.e. functors out of
$\Bord_4^{BG}$ (for $G = \O\left(A\right)$ or $\SO\left(A\right)$) as is discussed in
\Cref{rmk:alt_anom}.
\label{rmk:anom_extends}
\end{rmk}

\begin{cor}
Restricting the $\O\left(A,q\right)$-anomaly of \Cref{thm:pin_anom} to $\SO\left(A ,
q\right)$ we obtain an $\SO\left(A , q\right)$-anomaly of $F$ 
\begin{equation*}
\al_{c\left(\SSpin\right)} \colon \Bord^{B\SO\left(A , q\right)}_4 \to \BrFus \ ,
\end{equation*}
and an anomalous $\SO\left(A , q\right)$-equivariant theory 
\begin{equation}
F_{c\left(\SSpin\right)} \colon \al_{c\left(\SSpin\right)} \to 1_{B\SO\left(A , q\right)} \ .
\label{eqn:F_spin}
\end{equation}
\label{cor:spin_anom}
\end{cor}

\subsubsection{Trivializing the anomaly}
\label{sec:triv_O4}

As in \Cref{sec:triv_anom}, we ask if we can trivialize the anomaly from
\Cref{thm:pin_anom} (resp. \Cref{cor:spin_anom}) after pulling back along a map 
$f \colon G \to \O\left(A\right)$ (resp. $f\colon G \to \SO\left(A\right)$) for $G$ a
finite group.

The map $f$ defines a new anomaly theory $\al_{f^* c\left(\PPin\right)}$:
\begin{equation*}
\begin{tikzcd}
\Bord_3^{BG} \ar{d}{f} \ar{dr}{\al_{f^* c\left(\PPin\right)}} & \\
\Bord_3^{B\PPin\left(A\right)} \ar[swap]{r}{\al_{c\left(\PPin\right)}}&
\BrFus
\ .
\end{tikzcd}
\end{equation*}
Trivializing the class $f^* c\left(\PPin\right)$ is equivalent to trivializing the anomaly
$\al_{f^* c\left(\PPin\right)}$, which yields 
a (non-anomalous) $G$-theory by \Cref{prop:1_X}
\begin{equation*}
1_{BG} \to \al_{f^*c\left(\PPin\right)} \to 1_{BG} \ .
\end{equation*}

In addition to a fusion category $\cC$, the data defining a theory $\Bord^{BG}_3\to \Fus$
includes of a coherent action of $G$ on $\cC$. 
This means each element $g\in G$ is assigned a bimodule $M_g$, along with equivalences 
between the composition $\cM_h \tp \cM_g$ and $\cM_{gh}$ (giving rise to vanishing $O_3$
as in \Cref{sec:O3}) and associators giving rise to vanishing $O_4$. (See
\cite[\S 7,8]{ENO:homotopy} for details.)
This makes 
\begin{equation}
\bdsum_{g\in G} M_g
\label{eqn:G_ext}
\end{equation}
itself a fusion category, which is called a $G$-extension of $\cC$.

\begin{rmk}
\cite[Theorem 1.3]{ENO:homotopy} associates a $G$-extension of $\cC$ to a trivialization
of $O_3$ and $O_4$.
If we use this trivialization to trivialize $\al$, the $G$-extension of
\cite{ENO:homotopy} is precisely \eqref{eqn:G_ext} from the $G$-theory 
$1_{BG} \to \al_{O_4} \to 1_{BG}$.
\label{rmk:ENO}
\end{rmk}

\subsection{Functorial assignment of TQFTs to a vector space over a finite field}
\label{sec:Fp_quant}

Let $L \simeq \FF_p^n$ be a vector space over a finite field $\FF_p$
($p \neq 2$).
The fusion category $\cC = \left(\Vect\left[L\right] , *\right)$ defines a theory:
\begin{equation*}
F \colon \Bord_3^\fra \to \Fus\ .
\end{equation*}
Note that 
\begin{equation*}
\cZ\left(\cC\right) \simeq \Vect\left[\FF_p^{2n}\right]
\end{equation*}
with braiding induced by the quadratic form $q_{\splitt}$ on $V \simeq \FF_p^{2n}$ of
signature $\left(n,n\right)$.
Write $\O\left(n,n;\FF_p\right)$ for the split orthogonal group over $\FF_p$. 
Recall this agrees with $\O\left(V , q_{\splitt}\right)$ (\Cref{obs:O}).

Recall from \Cref{sec:ffield_O4} that 
the obstruction $O_4\left(\FF_p^{2n} , q_{\splitt}\right)$ is shown to vanish in this
case in \cite[Theorem 6.1]{EG:reflection}.
Therefore we have the following by \Cref{thm:pin_anom}.

\begin{cor}
There is a canonically defined $H^3\left(B\O \left(n,n;\FF_p\right) ,
\units{\bk}\right)$-torsor of $\O \left(n,n;\FF_p\right)$-TQFTs
\begin{equation*}
F_{O_{n,n}} \colon
\Bord^{B\O \left(n,n; \FF_p\right)}_4 \to \Fus
\ .
\end{equation*}
\label{cor:Fp}
\end{cor}

Recall the analogy between this categorical representation of $\O \left(n,n;\FF_p\right)$
and the Weil representation of the metaplectic group in \Cref{rmk:heis}. 
\Cref{cor:Fp} can be thought as an analogue of the fact that the Weil representation
splits over a finite field \cite{GH:finite,GH:categorical}.

\subsection{Module structures on \texorpdfstring{$3$d}{three-dimensional} TQFTs}
\label{sec:3d_symm}

To describe the projectivity captured by the anomalies in
\Cref{thm:k3,thm:pin_anom,cor:spin_anom} in terms of twisted module structures as in
\Cref{thm:anom_symm,cor:anom_symm}, we need to assume \Cref{hyp:sigma} for our
$4$-dimensional target $\cT$.

There is latitude in the choice of target. 
One option is to use $\cT = \BrFus$. This target is sufficient to support our anomaly
theories 
$\al_{k_3}$, $\al_{c\left(\PPin\right)}$, and 
$\al_{c\left(\SSpin\right)}$ from \Cref{thm:k3,thm:pin_anom,cor:spin_anom}.

Whenever a $4$-dimensional target $\cT$ satisfies:
\begin{enumerate}[label = (\alph*)]
\item $\Om \cT = \End_\cT\left(1\right) \simeq \Fus$, and 
\label{om}
\item there is a fully-faithful functor $\BrFus \to \cT$, 
\label{ff}
\end{enumerate}
we can regard the anomaly theories $\al_{k_3}$, $\al_{c\left(\PPin\right)}$, and
$\al_{c\left(\SSpin\right)}$ from \Cref{thm:k3,thm:pin_anom,cor:spin_anom} as valued
in $\cT$. 

\begin{exm}
As is remarked in \cite{DR:2Fus}, one can define an $\left(\infty , 4\right)$-category of
algebras in the $3$-category of finite semisimple $2$-categories (with bimodules, linear
$2$-functions, natural transformations, and modifications).
It is conjectured in \cite{DR:2Fus} that $2$-fusion categories are the fully-dualizable
objects of this $\left(\infty , 4\right)$-category.

Finite nonabelian topological $4$-dimensional $G$-gauge theory is simply classified by
sending the point to the fusion $2$-category $2\Vect\left[G\right]$ of $G$-graded
$2$-vector spaces \cite[Construction 2.1.13]{DR:2Fus}.
\end{exm}

Let $\cC$ be an object of the Morita $3$-category of fusion categories. 
As in \eqref{eqn:C}, this classifies a 
framed theory $F \colon \pt \mapsto \cC$.

\begin{cor}
Assuming \Cref{hyp:sigma} for either $\cT = \BrFus$ or another target $\cT$ satisfying
\ref{om} and \ref{ff}, 
the framed theory $F$ has a canonical
$\left(\sigma_{B\Aut_\EqBr\left(\cZ\left(\cC\right)\right), k_3}^4 ,
\rho\right)$-module structure.
\label{cor:symm_k3}
\end{cor}

\begin{proof}
On account of the anomalous theory \eqref{eqn:Fk3}, 
\Cref{cor:anom_symm} \ref{cor_symm} implies the result. 
\end{proof}

Let $G_2$ be a finite $2$-group (i.e. $BG_2$ is a $\pi$-finite $2$-type), and assume we
have chosen a trivialization of $f^* k_3$, where $f \colon B G_2 \to X$.
Assuming \Cref{hyp:sigma}, by \Cref{cor:anom_symm}\ref{triv_cor_symm}, we can reduce 
(in the sense of \Cref{defn:reduction}) the canonical 
$\left(\sigma_{B\Aut_{\Fus}}\left(\cC\right) , \rho\right)$-module structure on $F$
in \Cref{cor:symm_k3}
to a $\left(\sigma_{BG_2} , \rho_{BG_2}\right)$-module structure.
This in particular yields a $G_2$-theory by \Cref{prop:sigma_mod}, which agrees with 
the $BG_2$ theory obtained by trivializing the anomaly.

\begin{cor}
Assuming \Cref{hyp:sigma} for either $\cT = \BrFus$ or another target $\cT$ satisfying
\ref{om} and \ref{ff}, 
there is a $\left(\sigma_{B\PPin\left(A\right)} ,
\rho_{B\PPin\left(A\right)}\right)$-module structure on $F$.
Similarly, there is a $\left(\sigma_{B\SSpin\left(A\right)} ,
\rho_{B\SSpin\left(A\right)}\right)$-module structure on $F$.
\label{cor:symm_O4}
\end{cor}

\begin{proof}
On account of the anomalous theory \eqref{eqn:F_pin},
\Cref{thm:anom_symm}\ref{symm} implies the result.
The $\SSpin\left(A\right)$ version follows from
\Cref{thm:anom_symm}\ref{symm}, for the anomalous theory \eqref{eqn:F_spin}.
\end{proof}

Let $f \colon BG \to \O\left(A\right)$ for a finite group $G$.
Given a trivialization of $f^* c\left(\PPin\right)$,
the $\left(\sigma_{B\PPin\left(A\right)} , \rho\right)$-module structure on $F$ can be reduced to a
$\left(\sigma_{BG} , \rho_{BG}\right)$-module structure as in \Cref{prop:reduction_Y}.
This in particular yields a theory on $\Bord^{BG}_3$ by
\Cref{prop:sigma_mod} which agrees with
the theory obtained by trivializing the anomaly by \Cref{thm:triv_anom_symm}\ref{symm}.
The same is true when $\PPin\left(A\right)$ is replaced by $\SSpin\left(A\right)$ and
$f \colon G \to \O\left(A\right)$ is replaced by $f \colon G\to \SO\left(A\right)$.

\begin{cor}
Assuming \Cref{hyp:sigma} for $\cT = 2\Fus$,
there is an $H^3\left(B\O \left(n,n;\FF_p\right) ,
\units{\bk}\right)$-torsor of 
$\left(\sigma_{B\O \left(n,n;\FF_p\right)} , \rho_{B\O
\left(n,n;\FF_p\right)}\right)$-module structure on $F = \sigma_{\FF_p^n}^3$.
\label{cor:symm_Fp}
\end{cor}

\begin{proof}
By \Cref{cor:anom_symm}\ref{triv_cor_symm}, the trivialization defining the 
$\O \left(n,n;\FF_p\right)$-theory in \Cref{cor:Fp} (with underlying framed theory $F$)
defines a $\left(\sigma_{B\O\left(n , n ; \FF_p\right)} , \rho_{B\O\left(n , n ;
\FF_p\right)}\right)$-module structure on $F = \sigma_{\FF_p^n}^3$. 
\end{proof}

\appendix
\section{TQFT and category theory}
\label{sec:TQFT}

The axioms for functorial quantum field theory originated in \cite{At:TQFT,S:CFT}.
In this work, we will use TQFT to refer to 
\emph{fully-extended functorial topological quantum field theories}, which first
appeared in \cite{BD:CH}.
These were further studied in \cite{L:CH}.
Other references include \cite{F:TQFT,Wal:TQFT,Tel:TQFT,SP:dual,Kap:TQFT}.

\subsection{The Cobordism Hypothesis}
\label{sec:cob_hyp}

The Cobordism Hypothesis was formulated by Baez-Dolan in \cite{BD:CH}.
A detailed sketch of the proof was given in \cite{L:CH}, and 
an approach for a proof using factorization homology is in \cite{AF:CH}.
A proof in two dimensions is in \cite{SP:CH}.
See \cite{F:CH,SP:dual} for other references.

Let $d \in \ZZ^{\geq 0}$.
Given a symmetric-monoidal $\left(\infty , d+1\right)$-category $\cT$, which will be the
target for our theories, write $\cT^\sim$ for the groupoid obtained by discarding all non-invertible
morphisms at all levels.
We will assume $\cT$ has duals.\footnote{Otherwise replace $\cT$ with the subcategory consisting of
the fully-dualizable objects of $\cT$.}
We will also assume that $\Om^{d+1} \cT = \bk$ for $\bk$ an algebraically closed field of
characteristic zero.
Throughout the paper, we will write $\Vect$ for the $\bk$-linear category of
finite-dimensional vector spaces over $\bk$.

\begin{rmk}
Besides having duals and satisfying $\Om^{d+1} \cT = \bk$, we will occasionally assume
that $\cT$ satisfies an additional hypothesis:
When we discuss twisted module structures, in the sense of \cite{FMT} (as yet
another avatar of anomalous theories) we will assume that the target $\cT$ is sufficiently
additive to support the construction of $\pi$-finite spaces in the sense of \cite[\S 8]{FHLT}
and \cite[\S A]{FMT} (this is \Cref{hyp:sigma}).
\end{rmk}

The (framed) Cobordism Hypothesis, \cite[Theorem 2.4.6]{L:CH}, asserts that, since $\cT$
has duals, the functor given by evaluation on the point is an equivalence:
\begin{equation*}
\Fun^\tp \left(\Bord^\fra_{d+1} , \cT\right)
\lto{\sim} \cT^\sim \ .
\end{equation*}

\begin{rmk}
The category $\Fun^\tp\left(\Bord , \cT\right)$
consists of symmetric monoidal functors with morphisms given by strong natural
transformations (in the sense of \cite{JFS}).
A priori this is an $\left(\infty , d+1\right)$-category, but the statement of the
Cobordism Hypothesis is that it is actually an $\left(\infty , 0\right)$-category, and in
particular it is equivalent to the $\left(\infty , 0\right)$-category $\cT^\sim$.

We will occasionally use the variant of $\Fun^\tp\left(\Bord , \cT\right)$ which has
the same objects, however the $1$-morphisms are lax natural transformations. We will write
this as $\Fun^{\lax} \left(\Bord , \cT\right)$. See
\Cref{rmk:lax}.
\label{rmk:fun}
\end{rmk}

\subsubsection{The Cobordism Hypothesis for \texorpdfstring{$X$-theories}{theories with
background fields}}
\label{sec:X_theories}

There is a refined version of the Cobordism Hypothesis obtained as follows.
We follow \cite{L:CH}.
Let $X$ be a topological space with real rank $d+1$ vector bundle $\z$, and let 
$M$ be a manifold of dimension $m\leq d+1$.
An $\left(X , \z\right)$-structure on $M$ consists of 
\begin{enumerate}[label = (\roman*)]
\item a continuous map $f\colon M \to X$, and 
\item an isomorphism of bundles $TM \dsum \RR^{d+1-m} \simeq f^* \z$.
\end{enumerate}
Write $\Bord_{d+1}^{\left(X , \z\right)}$ for the category of bordisms equipped with an
$\left(X , \z\right)$-structure.
This is \cite[Definition 2.4.17]{L:CH}.
We will refer to symmetric monoidal functors out of $\Bord_{d+1}^{\left(X , \z\right)}$ as 
$\left(d+1\right)$-dimensional $\left(X , \z\right)$-TQFTs, or just $\left(X ,
\z\right)$-theories.
The Cobordism Hypothesis for $\left(X , \z\right)$-theories is
\cite[Theorem 2.4.18]{L:CH}. It is the equivalence
\begin{equation}
\Fun^\tp \left(\Bord^{\left(X , \z\right)}_{d+1} , \cT\right) \lto{\sim}
\Hom\left(\widetilde{X} , \cT^\sim\right)
\ ,
\label{eqn:CH_tilde}
\end{equation}
where $\widetilde{X}$ is the associated principal $\O\left(d+1\right)$-bundle of orthonormal frames in
$\z$.

Let $G$ be a topological group with continuous group homomorphism to the orthogonal group
$\chi \colon G\to \O\left(d+1\right)$. 
Particularly important cases of $\left(X , \z\right)$-structures are given by $X = BG$
and $\z = \left(\RR^{d+1} \times EG\right) / G$ defined by $\chi$.
In this case, we sometimes call theories defined on $\Bord^{BG}_{d+1}$
\emph{$G$-theories}.

\begin{exm}
Let $G$ be trivial and $\chi = \io$ be the inclusion. A $\left(BG , \z_{\io}\right)$-structure is a framing.
\end{exm}

\begin{exm}
Let $G = \SO\left(d+1\right)$ and $\chi = \io$ be the inclusion. A
$\left(B\SO\left(d+1\right) , \z_{\io}\right)$-structure is an orientation.
\label{exm:Bord_SO}
\end{exm}

We will write $\Bord^X_{d+1}$ for the category of bordisms with $\left(X ,
\z\right)$-structure with $\z$ trivial.
E.g. $X = BG$ and $\chi \colon G\to \O\left(d+1\right)$ trivial.
In spite of $\fra$ being removed from the notation, one should think of $\Bord^X_{d+1}$ as
consisting of framed bordisms with a map to $X$.
We will refer to symmetric monoidal functors $\Bord^X_{d+1} \to \cT$ as \emph{$X$-TQFTs}
(or just $X$-theories).
These should be thought of as families of framed TQFTs over $X$.
In this case, the Cobordism Hypothesis \eqref{eqn:CH_tilde} reduces to:
\begin{equation}
\Fun^\tp \left(\Bord_{d+1}^X ,\cT\right) \lto{\sim} \Hom \left(X , \cT^\sim \right)
\ .
\label{eqn:CH_X}
\end{equation}

\begin{exm}
When $X = BG$ with $\chi = \triv \colon G\to \O\left(d+1\right)$ trivial, then the Cobordism Hypothesis reduces
to \cite[Theorem 2.4.26]{L:CH}:
\begin{equation*}
\Fun^\tp \left(\Bord_{d+1}^{BG} ,\cT\right) \lto{\sim} \left(\cT^\sim\right)^{hG}
\end{equation*}
where $\left(\cT\right)^{hG}$ denotes the homotopy fixed points of $\cT$.

We will sometimes call symmetric monoidal functors $\Bord^{BG}_{d+1} \to \cT$
\emph{$BG$-theories}, or \emph{$G$-theories}.
\label{exm:Bord_BG}
\end{exm}

Endomorphisms of the trivial TQFT are equivalent to theories of one lower dimension.
By the Cobordism Hypothesis this is equivalent to endomorphisms of the identity in the
target being decategorification.
\Cref{prop:1_X} is this result for $X$-theories.

\begin{lem}
If $\mathcal{A}$ and $\mathcal{B}$ are monoidal $\left(\infty , d+1\right)$-categories, then
\begin{equation*}
\Om \Fun^{\tp} \left(\mathcal{A} , \mathcal{B} \right)
\simeq \Fun^{\tp} \left( \mathcal{A} , \Om \mathcal{B} \right)
\ .
\end{equation*}
\label{lem:Om}
\end{lem}
\begin{proof}
The LHS consists of natural transformations:
\begin{equation*}
\begin{tikzcd}
\mathcal{A} \ar{r}\ar{d}&
\cat{1}
\ar{d}
\\
\cat{1} \ar{r}
\ar[Rightarrow]{ur}
& \mathcal{B}
\end{tikzcd}
\end{equation*}
where $\cat{1}$ denotes the trivial monoidal category. 
The category $\Om \mathcal{B}$ is the pullback of the diagram
$\cat{1} \to \mathcal{B} \from \cat{1}$, so the result follows from the universal property.
\end{proof}

\begin{prop}
If $1_X^{d+1}$ is the trivial functor $\Bord^X_{d+1} \to \cT$, then
\begin{equation*}
\Fun^\tp \left(\Bord^X_{d} , \Om \cT\right)
\simeq \End\left(1_X^{d+1}\right)
\ ,
\end{equation*}
where $\Om \cT \ceqq \End_{\cT}\left(1\right)$ denotes the looping of $\cT$.
\label{prop:1_X}
\end{prop}
\begin{proof}
This follows from \Cref{lem:Om} for $\mathcal{A} = \Bord^X_{d}$ and $\mathcal{B} = \cT$.
\end{proof}

\subsection{Relative theories}
\label{sec:relative_theories}

A symmetric monoidal functor
\begin{equation*}
\al \colon \Bord_{d}^X \to \cT
\end{equation*}
is a \emph{once-categorified $d$-dimensional $X$-TQFT}, 
where $\cT$ is the fixed target from the beginning of \Cref{sec:cob_hyp}. Let 
\begin{equation*}
1_X \colon \Bord_{d}^X \to \cT
\end{equation*}
denote the trivial once-categorified $d$-dimensional $X$-TQFT.
Recall the notion of a relative theory \cite{FT:relative}.
These are also called \emph{twisted theories} \cite{ST:twisted,JFS}.

\begin{defn*}
A theory defined right-relative to $\al$ is a lax natural transformation (in the sense
of \cite{JFS})
\begin{equation*}
\al \to 1
\ .
\end{equation*}
A theory defined left-relative to $\al$ is a lax natural transformation
\begin{equation*}
1 \to \al \ .
\end{equation*}
\end{defn*}

\begin{rmk}
Recall the definition of lax (resp. oplax) natural transformations from \cite{JFS}. 
Consider the arrow categories $\cT^\down$ and $\cT^\to$, and the
source and target functors $s,t \colon \cT^* \to \cT$
for $* = \down , \to$.
Following \cite{JFS}, a lax (resp. oplax) natural transformation $\al \to 1$ is a functor 
\begin{align}
F_\al \colon \Bord^X_d \to \cT^{\down}
&&
\left(
\text{resp. }
F_\al \colon \Bord^X_d \to \cT^{\to}
\right)
\end{align}
such that $s\comp F_\al = \al$, and 
$t\comp F_\al = 1$.

Throughout, we will use the lax version, as written in the above definition of relative
theories. 
The reason we use the lax version, as noted in \cite[Example 7.3]{JFS}, that the
\emph{lax} natural transformations
from the trivial theory to itself consist of theories of dimension lower (\cite[Theorem
7.4]{JFS}) whereas the same is not true when lax is replaced with oplax.
We need the analogous result for $X$-theories (\Cref{prop:1_X}), in particular to
establish a trivialized anomalous theory $1\lto{\sim}\al\to 1$ as an honest theory of
one dimension lower.

Also noted in \cite[Example 7.3]{JFS}, is the fact that oplax natural transformations are
``elements'' $F_{\sigma}\left(M\right) \colon 1 \to \al\left(M\right)$ for $M$ a
\emph{closed} bordism of any codimension, which is for example the point of view taken in
\cite{FT:relative}.
Oplax natural transformations are also used in \cite{FT:gapped}.
Besides \Cref{prop:1_X}, and the results depending on it, 
the remaining constructions and facts in this paper hold for the oplax version, 
obtained my replacing all lax natural transformations with oplax ones, and replacing
$\cT^{\down}$ with $\cT^{\to}$.
\label{rmk:lax}
\end{rmk}

\begin{rmk}
Often the once-categorified $d$-dimensional theory $\al$
extends to a $\left(d+1\right)$-dimensional theory:
\begin{equation*}
\begin{tikzcd}
& \cT \\
\Bord_d^X \ar[hook]{r}\ar{ur}{\al}&
\Bord_{d+1}^X \ar[dashed]{u}
\end{tikzcd}
\end{equation*}
In this case, the relative theory $\al \to 1$ is upgraded to what is called a
\emph{boundary theory}.
Boundary theories are defined as functors out of the extended bordism category
$\Bord^{X , \p}_{d+1}$, described in \cite[\S 4.3]{L:CH}.
See also \cite{Stew:thesis}.
The connection with the notion defined here is made in \cite[Theorem 7.15]{JFS}.
\label{rmk:boundary}
\end{rmk}

\subsection{Invertible theories}
\label{sec:invertible}

Let $\abs{\Bord_{d+1}^{X}}$ denote the completion, i.e. the $\left(\infty ,
0\right)$-category obtained by manually inverting all the arrows, 
and let $\units{ \cT }$ denote the underlying sub-groupoid, i.e. the 
$\left(\infty , 0\right)$-category obtained by discarding all
non-invertible objects and non-invertible morphisms at all levels.

A TQFT $\al \colon \Bord_{d+1}^{X} \to \cT$ is \emph{invertible} if it factors as:
\begin{equation}
\begin{tikzcd}
\Bord_{d+1}^{X} \ar{r}{\alpha}\ar[two heads]{d} & 
\cT \\
\abs{\Bord_{d+1}^{X}} \ar{r}{\widetilde{\al}}&
\units{ \cT } \ar[hook]{u}
\end{tikzcd}
\label{eqn:invertible}
\end{equation}
I.e. it assigns all objects and morphisms in the source to invertible
objects and morphisms in the target. 

\begin{rmk}
Note that \eqref{eqn:invertible} reduces the functor $\al$ to a map of spectra
$\widetilde{\al}$.
This allows us to study it as a 
cohomology class (in the theory determined, via Brown representability, by the target) on
the source.
\label{rmk:map_of_spectra}
\end{rmk}

\section{Topological symmetry}
\label{sec:top_symm}

\subsection{TQFTs associated to \texorpdfstring{$\pi$}{pi}-finite spaces}
\label{sec:sigma}

Let $X$ be a space (i.e. higher groupoid) which is (connected, pointed, and)
$\pi$-finite\footnote{This means $X$ has finitely many homotopy groups, each of which is
finite.}.
There is a recipe for constructing a TQFT using $X$, which was
introduced in \cite{K:finite} and studied further in \cite{Q:finite,T:finite}.
We will consider the fully local case introduced in \cite{F:higher} and studied in
\cite[\S 3,8]{FHLT} and \cite[\S A]{FMT}.

In the $1$-dimensional case, the finite path-integral has mathematical rigorous
foundations in the $\infty$-categorical setting \cite{Harp}.
Related work in a higher setting is the subject of an upcoming work of Claudia Scheimbauer
and Tashi Walde.

\subsubsection{The summation map}

We will proceed heuristically, following \cite[\S 3]{FHLT} and \cite[\S A.2]{FMT}, to fix
notation and describe expectations which will eventually be stated and assumed in
\Cref{hyp:sigma}.

Let $\Fam_{d+1}$ denote the category with objects finite $\left(d+1\right)$-groupoids, 
$1$-morphisms given by correspondences of $\pi$-finite spaces, 
$2$-morphisms given by correspondences of correspondences, and so on until level
$\left(d+1\right)$. 
(Two $\left(d+1\right)$-morphisms are regarded as identical if they are equivalent.)

Let $\cT$ be the arbitrary symmetric monoidal target with duals, fixed in
\Cref{sec:cob_hyp}.
Let $Y$ be an object of $\Fam_{d+1}$.
A local system on $Y$ valued in $\cT$ is a functor $Y \to \cT$. 
Write $\Fam_{d+1}\left(\cT\right)$ for the category of $\pi$-finite spaces equipped with a
local system valued in $\cT$.

For example, $\Fam_{d+1}\left(B^{d+1} \units{\bk}\right)$ has objects given by pairs
$\left(Y , \tau\right)$, where $\tau$ is a cocycle 
\begin{equation*}
\tau \colon Y \to B^{d+1} \units{\bk}
\end{equation*}
representing a class in $H^{d+1}\left(Y , \units{\bk}\right)$.
A morphism is a correspondence:
\begin{equation}
\begin{cd}
& \left(E , \mu\right) \ar{dr}{p_2}\ar{dl}{p_1} & \\
\left(Y_1 , \tau_1\right)  && \left(Y_2 , \tau_2\right)
\end{cd}
\label{eqn:corr_mor}
\end{equation}
where $\mu \colon E \to B^d \units{\bk}$ satisfies
\begin{equation}
d \mu = \left( p_1^* \tau_1 \right)^{-1}  \cdot 
\left( p_2^* \tau_2 \right)
\ .
\label{eqn:triv}
\end{equation}
$2$-morphisms are correspondences between correspondences with a similar condition on the
cocycles, and so on to define morphisms up to level $\left(d+1\right)$.

Recall we assumed $\Om^{d+1}\cT = \bk$. Therefore there is a natural functor:
$B^{d+1} \units{\bk} \to \cT$ inducing a functor:
\begin{equation}
\Fam_{d+1}\left(B^{d+1} \units{\bk}\right) \to \Fam_{d+1}\left(\cT\right) \ .
\label{eqn:twist_loc_sys}
\end{equation}

We assume that the following holds.

\begin{custom}{Hypothesis Q}
For $\cT$ appropriately additive\footnote{As is explained in \cite[\S 3]{FHLT}, for $X$ an
ordinary groupoid, then $\cT$ must be additive in the sense that the colimit $\dirlim_{x\in
X}\tau\left(x\right)$ in $\cT$ exists, and agrees with the limit $\invlim_{x\in X}
\tau\left(x\right)$, where we are regarding $\tau$ as defining a $\cT$-valued local system
on $X$ by \eqref{eqn:twist_loc_sys}.},
there is a ``quantization''\footnote{See \cite[Remark A.7.1]{FMT} where it is explained how
$\Sum_{d+1}\left(X\right)$ (and therefore $\sigma_X^{d+1}$) can be obtained by 
integrating over fluctuating fields.}
functor
\begin{equation}
\Sum_{d+1} \colon \Fam_{d+1}\left(B^{d+1} \units{\bk}\right) \to \cT
\label{eqn:sum}
\end{equation}
such that there is an invertible natural transformation between 
\begin{equation*}
\left(X , \tau\right) \mapsto \Hom_\cT \left(1 , \Sum_{d+1}\left(X , \tau\right)\right)
\end{equation*}
and 
\begin{equation*}
\left(X , \tau\right) \mapsto 
\Hom\left(\hofib\left(\tau\right) , \Om \cT\right) 
\end{equation*}
viewed as functors out of $\Fam_{d+1}\left(B^{d+1}\units{\bk}\right)$. 
\label{hyp:sigma}
\end{custom}

\begin{rmk}
Unpacking the existence of the natural transformation in \Cref{hyp:sigma}, we see that it
ensures that we have equivalences 
\begin{equation}
\Hom_{\cT}\left(1 , \Sum_{d+1}\left(X, \tau\right)\right) \simeq
\Hom\left(\hofib\left(\tau\right) , \Om \cT\right)
\label{eqn:hyp:sigma}
\end{equation}
for all objects $\left(X , \tau\right)$ of $\Fam_{d+1}\left(B^{d+1}\units{\bk}\right)$,
and given a morphism from $\left(X_1 , \tau_1\right)$ to $\left(X_2 , \tau_2\right)$ in
$\Fam_{d+1}\left(B^{d+1} \units{\bk}\right)$, we have a commuting diagram:
\begin{equation}
\begin{tikzcd}
\Hom_\cT\left(1 , \Sum_{d+1}\left(X_1 , \tau_1 \right)\right)
\ar[phantom]{r}{\simeq}\ar{d}&
\Hom\left(\hofib\left(\tau_1\right) , \Om \cT\right) \ar{d} \\
\Hom_\cT\left(1 , \Sum_{d+1}\left(X_2 , \tau_2 \right)\right)
\ar[phantom]{r}{\simeq} & 
\Hom\left(\hofib\left(\tau_2\right) , \Om \cT\right)
\end{tikzcd}
\label{eqn:commuting_square}
\end{equation}
This generalizes the classical fact about modules over the group algebra and
$G$-representation, as in \Cref{exm:sigma_BG} \ref{sigma2_BG}.
\label{rmk:hyp:sigma}
\end{rmk}

\begin{rmk}
In \cite[\S 8.2]{FHLT} the map $\Sum_n$ is constructed at the level of objects
and shown to be a functor up to $2$-morphisms.
In particular, the cases in \Cref{exm:sigma_BG}
are worked out in \cite[\S 8.1,8.3]{FHLT}.
\end{rmk}

\begin{exm}
\begin{enumerate}[label = (\roman*)]
\item Let $\cT = \Vect$. 
Then $\Sum_1\left(X\right) = \bk\left(\pi_0 X\right)$ is the vector space of $\bk$-valued functions on
$\pi_0 X$. \Cref{hyp:sigma} is satisfied, since 
the natural transformation \eqref{eqn:hyp:sigma} required in \Cref{hyp:sigma} is:
\begin{equation*}
\Hom_{\Vect}\left(1 , \Sum_1\left(X\right)\right) \simeq \Hom\left(X , \bk\right) \simeq 
\Hom\left(\pi_0 X , \bk\right) \ .
\end{equation*}

\item\label{sigma2_BG} Set $X = BG$ for a finite group $G$, and let $\cT = \Alg$ be the Morita $2$-category
of algebras.
Let $\Sum_2\left(BG\right) = \bk\left[G\right]$ be the group algebra.
\Cref{hyp:sigma} is satisfied, since 
the natural transformation \eqref{eqn:hyp:sigma} required in \Cref{hyp:sigma} is:
\begin{equation*}
\Hom_\Alg\left(1 , \Sum_2\left(BG\right)\right) = \lMod{\bk\left[G\right]} \simeq
\Rep\left(G\right) \simeq \Hom\left(BG , \Vect\right) \ .
\end{equation*}

We can equip $BG$ with a $2$-cocycle $\mu$ (i.e. a group $2$-cocycle)
which twists the convolution structure on the group algebra, resulting in 
$\Sum_2\left(BG , \tau\right)$.

\item Set $X = BG$.
Let $d = 2$, and take $\cT$ to be the Morita $3$-category of fusion categories.
Then $\Sum_3\left(BG\right)$ is $\Vect\left[G\right]$, the fusion category of
vector bundles on $G$ with convolution. 
We can equip $BG$ with a $3$-cocycle $\al$ which twists the fusion structure on
$\Vect\left[G\right]$, yielding $\sigma_{BG,\al}^3\left(\pt\right)$.
This is \cite[Example A.65]{FMT}.

\item\label{sigma_B2G} Set $X = B^2 G$.
Let $d = 3$ and take $\cT$ to be the Morita $4$-category of braided monoidal
categories, $\BrFus$ (see \Cref{sec:fusion_prelim}).
Consider a cocycle $\tau \colon B^2 G \to B^4\units{\bk}$.
It is a theorem of Eilenberg-Maclane \cite[Theorem 26.1]{EM:2}
that cohomology classes in $H^4\left(B^2 G , \units{\bk}\right)$ correspond to quadratic
forms $G \to  \units{\bk}$.
Write $q_\tau$ for the form corresponding to $\left[\tau\right]$.
This defines a symmetric bicharacter on $G$:
\begin{equation*}
\lr{g , h}_\tau \ceqq \frac{q_\tau\left(g+h\right)}{q_\tau\left(g\right) q_\tau\left(h\right)}
\ .
\end{equation*}
Then $\Sum_4\left(B^2 G\right)$ is $\Vect\left[G\right]$ with convolution, and with
braiding specified on simples by:
\begin{equation*}
\beta_\tau \colon\bk_g * \bk_h = \bk_{gh} \lto{\lr{g,h} \id_{\bk_{gh}}} \bk_{gh} =
\bk_{hg} = \bk_h * \bk_g \ .
\end{equation*}

\item Let $d = 3$ and $\cT$ be a $3$-category of monoidal $2$-categories 
(e.g. $2\Fus$ \cite{DR:2Fus}). 
For any finite group $G$, the fusion $2$-category $\Sum_4\left(BG\right)$ is the
collection of $G$-graded $2$-vector spaces \cite[Construction 2.1.13]{DR:2Fus}.

As far as the author is aware, dualizability in the Morita $4$-category of monoidal
$2$-categories has not been extensively studied.
However this fusion $2$-category is expected define a fully extended $4$-dimensional TQFT
\cite{DR:2Fus}.
See \Cref{sec:3d_symm} where we discuss this example in more detail.
\label{sigma4_BG}
\end{enumerate}
\label{exm:sigma_BG}
\end{exm}

\subsubsection{TQFTs from groupoids}

The upshot of assuming \Cref{hyp:sigma} is that,
for a fixed object $\left(X , \tau\right)$ of $\Fam_{d+1}\left(B^{d+1}\units{\bk}\right)$, 
\eqref{eqn:sum} can be composed with the mapping space functor\footnote{Note that
\eqref{eqn:twist_loc_sys} allows us to construct a $\cT$-valued local system from the
cocycle $\tau$.} to obtain the theory 
\begin{equation*}
\sigma_{X , \tau}^{d+1} \colon 
\Bord^{\fra}_{d+1} \lto{\pi_{\leq d+1} \Map\left(- , X\right)}
\Fam_{d+1}\left(\cT\right)\lto{\Sum_{d+1}} \cT \ .
\end{equation*}

\begin{rmk}
In \cite[\S 3,8]{FHLT}, the theories $\sigma_X^{d+1}$ are studied for $\cT =
\Alg\left[d\right]$ the Morita $\left(d+1\right)$-category of ``$d$-algebras'', discussed
in \cite[\S 7]{FHLT}.
\end{rmk}

\begin{rmk}
As is remarked in \cite[\S A]{FMT} and \cite[\S 3,8]{FHLT},
$\sigma_{X , \tau}^{d+1}$ can be upgraded to an oriented theory, and if $\tau$ is trivial
then it can even be upgraded to an unoriented theory. 
We will work with the underlying framed theories in this paper.
\end{rmk}

\begin{prop}[\cite{FMT}]
Assuming \Cref{hyp:sigma}, a 
morphism from $\left(X_1 , \tau_1\right)$ to $\left(X_2 , \tau_2\right)$ 
in $\Fam_{d+1}\left(B^{d+1}\units{\bk}\right)$ (i.e. correspondence as in
\eqref{eqn:corr_mor}) induces a bimodule (i.e. domain wall)
\begin{equation*}
\sigma^{d+1}_{X_1 , \tau_1} \to \sigma^{d+1}_{X_2 , \tau_2} \ .
\end{equation*}
\label{prop:corr_mor}
\end{prop}

\begin{proof}
By the Cobordism Hypothesis, the functor $\Sum_{d+1}$ in \Cref{hyp:sigma} provides a functor:
\begin{equation*}
\begin{tikzcd}[row sep = 0pt]
\Fam_{d+1}\left(B^{d+1} \units{\bk}\right) \ar{r}&
\Fun^\tp\left(\Bord_{d+1}^{\fra} , \cT\right) \\
\left(X , \tau\right) \ar[mapsto]{r}&
\sigma_{X,\tau}^{d+1}
\end{tikzcd}
\end{equation*}
\Cref{prop:corr_mor} follows from the fact that the correspondence 
is a morphism in the category $\Fam_{d+1}\left(B^{d+1}\units{\bk}\right)$.
\end{proof}

\begin{rmk}
\Cref{prop:corr_mor} can be shown directly (i.e. without the Cobordism Hypothesis) since a
correspondence of spaces $X\from C \to Y$ defines a correspondences of mapping spaces
$\Map\left(M , X\right) \from \Map\left(M , C\right) \to \Map\left(M , Y\right)$, for any
bordism $M$. 
This is the perspective taken in \cite{FMT}.
\end{rmk}

Recall the motivation in \Cref{rmk:hyp:sigma} for the existence of the natural
transformation in \Cref{hyp:sigma}.
\Cref{prop:sigma_mod} writes this in terms of the theories $\sigma_X^{d+1}$.

\begin{prop}
Assuming \Cref{hyp:sigma}, every boundary theory for $\sigma_{X , \tau}^{d+1}$ is
classified by a symmetric-monoidal functor:
\begin{equation*}
\Bord_{d}^{\hofib\left(\tau\right)} \to \Om \cT
\ .
\end{equation*}
If $\tau$ is trivial, then the boundary theories are classified by symmetric-monoidal
functors:
\begin{equation*}
\Bord_d^{X} \to \Om \cT \ .
\end{equation*}
\label{prop:sigma_mod}
\end{prop}

\begin{proof}
By \cite[Theorem 7.15]{JFS}, the boundary theories for $\sigma_{X,\tau}^{d+1}$
are classified by 
\begin{equation*}
\Hom_{\cT}\left(1 , \sigma_{X,\tau}^d\left(\pt\right)\right)
\end{equation*}
which is the LHS of \eqref{eqn:hyp:sigma}.
Similarly, by the Cobordism Hypothesis \eqref{eqn:CH_X}, the symmetric-monoidal functors 
from $\Bord_d^{\hofib\left(\tau\right)}$ to $\Om \cT$ are 
classified by $\Hom\left(\hofib\left(\tau\right) , \Om \cT\right)$, which is also the RHS
of \eqref{eqn:hyp:sigma}.
If $\tau$ is trivial, then the RHS of \eqref{eqn:hyp:sigma} is equivalent to functors from
$X$ to $\Om \cT^\sim$. 
\end{proof}

\subsection{Module structures}
\label{sec:mod_str}

We summarize the material used in our construction from \cite{FMT}.
See \cite[\S 3]{FMT} for a more detailed discussion of these definitions.

Let
\begin{equation*}
\sigma \colon \Bord_{d+1}^\fra \to \cT
\end{equation*}
be a $\left(d+1\right)$-dimensional TQFT valued in the fixed target $\cT$ from the
beginning of \Cref{sec:cob_hyp}.
Recall the notion of a boundary theory from \Cref{rmk:boundary}.
Recall the following definition from \cite{FMT}.

\begin{defn*}
A \emph{$d$-dimensional quiche} is a pair $\left(\sigma , \rho\right)$ in which 
$\rho$ is a right topological boundary theory (or \emph{right $\sigma$-module}), which we will write as 
$\rho\colon \sigma \to 1$.
\end{defn*}

\begin{rmk}
All of the quiches considered in this paper will be of the form
$\left(\sigma_{X ,\tau} , \rho_{X , \tau}\right)$, where $\sigma_{X,\tau}$ is 
the theory associated to $\left(X , \tau\right)$ as in \Cref{sec:sigma}, and 
the (right) boundary theory $\rho_{X,\tau}$ is the natural transformation induced by the
correspondence diagram:
\begin{equation*}
\begin{cd}
& \pt \ar{dr}\ar{dl} & \\
\left(X , \tau\right) && \pt 
\end{cd}
\end{equation*}
as in \Cref{prop:corr_mor}.
Given a pointed space $X$, we will always write $\rho_{X , \tau}$ for this boundary theory. 
\label{rmk:rho_X}
\end{rmk}

Recall, from \cite[Corollary 7.7]{JFS}, there is an equivalence of $\left(\infty ,
d+1\right)$-categories between the following. 
\begin{enumerate}
\item The $\left(\infty , d+1\right)$-category $\Fun^{\lax}\left(\Bord_{d+1}^{\fra} ,
\cT\right)$ of framed TQFTs with $1$-morphisms given by relative\footnote{Recall from
\Cref{rmk:lax} we are using lax natural transformations/relative theories in this paper.}
field theories between them, and $k$-morphisms given by what Johnson-Freyd-Scheimbauer
call $k$-times-twisted field theories.

\item The $\left(\infty , d+1\right)$-subcategory of $\cT$ consisting of fully dualizable
objects, and with morphisms which are $\left(d+1\right)$-times left-adjunctible. 
\label{C_adj}
\end{enumerate}

\begin{defn}
A morphism between quiches $\left(\sigma , \rho\right) \to \left(\sigma' , \rho'\right)$ 
is a $2$-morphism between the theories $\rho$ and $\rho'$ as $1$-morphisms in
$\Fun^{\lax}\left(\Bord_{d+1}^{\fra} , \cT\right)$.
I.e. a twice-twisted theory as in \cite{JFS}. 
\label{defn:quiche_morphism}
\end{defn}

Given any quiche $\left(\sigma ,\rho\right)$, 
the boundary theory $\rho$ is equivalent to a relative theory 
$\tau_{\leq d} \sigma \to 1$ by \cite[Theorem 7.15]{JFS}, and is classified 
by a $1$-morphism 
\begin{equation*}
\abs{\rho} \colon
\sigma\left(\pt\right)\to 1
\end{equation*}
in $\cT$.

\begin{prop}
A morphism of quiches $\left(\sigma , \rho\right) \to \left(\sigma' , \rho'\right)$
(\Cref{defn:quiche_morphism}) is equivalent to 
a $2$-morphism $\rho\left(\pt\right) \to \rho'\left(\pt\right)$ which is
invertible. 
\label{prop:quiche_morphism}
\end{prop}

\begin{proof}
By \cite[Corollary 7.7]{JFS}, such a morphism of quiches 
is equivalent to a morphism $\rho\left(\pt\right) \to \rho'\left(\pt\right)$ which is 
$\left(d+1\right)$-times left-adjunctible. 
The result follows from the fact that 
any $k$-morphism in the $\left(\infty , d+1\right)$-category $\cT$ which is forced to be
$\left(d+1\right)$-times left-adjunctible is necessarily invertible. 
\end{proof}

A quiche is an abstract symmetry datum, in the same sense as an algebra.
The following definition, from \cite{FMT}, is the analogue of a module, i.e. a realization
of the quiche as symmetries of a given theory. 

\begin{defn*}
Let $\left(\sigma , \rho\right)$ be a $d$-dimensional quiche, and let $F$ be a
$d$-dimensional TQFT.
A $\left(\sigma , \rho\right)$-module structure on $F$ is the pair $\left(F_\sigma ,
\theta\right)$ where $F_{\sigma}$ is a (left) boundary theory
$F_\sigma \colon 1 \to \sigma$ which is equipped with an isomorphism of
$d$-dimensional theories:
\begin{equation*}
\theta \colon \rho\tp_\sigma F_\sigma \lto{\sim} F \ .
\end{equation*}
\end{defn*}

\subsection{Reduction of topological symmetry}
\label{sec:reduction}

Let $X$ be a (pointed, connected) $\pi$-finite space, 
and consider a cocycle $c \colon X \to B^{d+1} \units{\bk}$.
Assuming \Cref{hyp:sigma}, and given 
a $\left(\sigma_{X , c}^{d+1} , \rho_{X , c} \right)$-module
structure on a theory $F$, we might wonder what extra data is needed to 
``reduce'' this to a $\left(\sigma_X^{d+1} , \rho_X\right)$-module structure on $F$.
\Cref{prop:induced}, which holds more generally, will tell us that a trivialization of $c$
is sufficient to perform such an operation, which we will define to be a \emph{reduction}
of topological symmetry in \Cref{defn:reduction}.

Recall the analogy in \cite{FMT} between modules over an algebra (or linear
representations of a Lie group) and field theories.
Under this analogy, \Cref{prop:induced} is the analogue of the fact that a bimodule
induces a functor between the categories of modules. 

\begin{theorem}
Assume \Cref{hyp:sigma}.
Given a $\left(\sigma_{1}, \rho_{1}\right)$-module structure on a
$d$-dimensional theory $F$,  then
a morphism of quiches (\Cref{defn:quiche_morphism}) 
from $\left(\sigma_1 , \rho_1\right)$ to $\left(\sigma_2 , \rho_2\right)$ 
canonically defines a $\left(\sigma_2 , \rho_2\right)$-module structure on $F$.
\label{prop:induced}
\end{theorem}

\begin{proof}
Write the given $\left(\sigma_1 , \rho_1\right)$-module structure on $F$ as
$\left(F_{\sigma_1} , \theta_1\right)$ (as in \Cref{sec:mod_str}).
A morphism of quiches (\Cref{defn:quiche_morphism}) is equivalent to an invertible
$2$-morphism:
\begin{equation*}
\begin{tikzcd}
\sigma_1\left(\pt\right) \ar{r}{\delta}
\ar[swap]{d}{\abs{\rho_1}}
&
\sigma_2\left(\pt\right)
\ar{d}{\abs{\rho_2} }
\\
1\ar[swap]{r}{\id_1}
\ar[Rightarrow]{ur}{s}
&
1 
\end{tikzcd}
\qquad\qquad
\begin{tikzcd}
\sigma_1\left(\pt\right) \ar{r}{\delta}
\ar[swap]{d}{\abs{\rho_1}}
&
\sigma_2\left(\pt\right)
\ar{d}{\abs{\rho_2} }
\ar[Rightarrow,swap]{dl}{s^{-1}}
\\
1\ar[swap]{r}{\id_1}
&
1 
\end{tikzcd}
\end{equation*}
by \Cref{prop:quiche_morphism}.
The morphism $\delta$ induces a $\left(\sigma_2 , \sigma_1\right)$-bimodule, i.e. a 
domain wall $D\colon\sigma_1 \to \sigma_2$.
Define the $\left(\sigma_2 , \rho_2\right)$-module structure $\left(F_{\sigma_2} ,
\theta_2\right)$ as follows.
Define the boundary theory $F_{\sigma_2}$ to be 
\begin{equation*}
F_{\sigma_2} = D \tp_{\sigma_1} F_{\sigma_1}
\ .
\end{equation*}
Define the equivalence
\begin{equation*}
\theta_2 \colon \rho_2 \tp_{\sigma_2} F_{\sigma_2} 
= \rho_2 \tp_{\sigma_2} D \tp_{\sigma_1} F_{\sigma_2} \lto{\sim} F
\end{equation*}
to be induced by the following composition:
\begin{equation*}
\begin{tikzcd}[row sep=40pt,column sep = 40pt]
1 \ar{r}{\id_1}
\ar{d}{\abs{F_{\sigma_1}}}
\ar[bend right = 100,swap,  ""{name=F}]{dd}{F}
&
1 \ar{d}{\abs{F_{\sigma_2}}}
\ar[Rightarrow]{dl}{\id}
\\
\sigma_1\left(\pt\right) \ar[swap]{r}{\delta}
\ar{d}{\abs{\rho_1} }
\ar[Rightarrow,to=F,"\abs{\theta_1}"]
&
\sigma_2\left(\pt\right)
\ar{d}{\abs{\rho_2} }
\ar[Rightarrow]{dl}{s^{-1}}
\\
1\ar[swap]{r}{\id_1}
&
1 
\end{tikzcd}
\end{equation*}
where $\abs{\theta_1}$ is the $2$-morphism in $\cT$ classifying the equivalence:
\begin{equation*}
\theta_1 \colon \rho_1 \tp_{\sigma_1} F_{\sigma_1} \lto{\sim} F \ .
\end{equation*}
\end{proof}

\begin{exm}
The inclusion of a subgroup $H\subset G$ defines a correspondence
\begin{equation*}
\begin{tikzcd}
& BH \ar{dr}\ar{dl} & \\
BG && BH
\end{tikzcd}
\end{equation*}
which induces a morphism of quiches
\begin{equation*}
\left(\sigma_{BG}^{d+1} , \rho_{BG} \right) \to \left(\sigma_{BH}^{d+1} , \rho_{BH}\right)
\ .
\end{equation*}
If a framed $d$-dimensional TQFT $F$ has a $\left(\sigma_{BG}^{d+1} ,
\rho_{BG}\right)$-module structure, then the induced $\left(\sigma_{BH}^{d+1} ,
\rho_{BH}\right)$-module structure from \Cref{prop:induced} is the restriction of the
original module structure along the inclusion of $H$. 
In \Cref{exm:sigma_BG} \ref{sigma2_BG} this is literally the restriction of a
$G$-representation to an $H$-representation.
\end{exm}

Now we return to $X$ a $\pi$-finite space with $c \colon X \to B^{d+1}\units{\bk}$.

\begin{prop}
A trivialization of the class $c$ induces a 
$\left(\sigma_X^{d+1} , \rho_X\right)$-module structure on any theory with a 
$\left(\sigma_{X , c}^{d+1} , \rho_{X , c}\right)$-module structure.
\label{prop:reduction_X}
\end{prop}

\begin{proof}
A trivialization $t$ of $c$ determines a correspondence as in \eqref{eqn:corr_mor}:
\begin{equation}
\begin{cd}
& \left(X , t\right) \ar{dr} \ar{dl} & \\
\left(X,c\right)
&& X 
\end{cd}
\label{eqn:triv_corr}
\end{equation}
By \Cref{prop:corr_mor}, this induces a domain wall $D$ from $\sigma^{d+1}_{X , c}$ to
$\sigma_X^{d+1}$.
This tautologically induces a morphism of quiches (\Cref{defn:quiche_morphism}) from
$\left(\sigma_{X,c}^{d+1} , \rho_{X , c}\right)$ to 
\begin{equation*}
\left(\sigma_X ,
\rho_{X,c}\tp_{\sigma_{X ,c}^{d+1}} D \right)
\ .
\end{equation*}
But now notice that 
$\rho_{X,c}\tp_{\sigma_{X ,c}^{d+1}} D$ is induced by the following 
composition of correspondences:
\begin{equation*}
\begin{cd}
&& \pt \ar{dr}\ar{dl}&& \\
& X\ar{dl}{s}\ar{dr}{\id_X} && \pt\ar{dl}\ar{dr} \\
\left(X , c\right) && X && \pt
\end{cd}
\end{equation*}
But this is just the pointing of $\left(X , c\right)$, i.e. it induces $\rho_{X , c}$. 
Then \Cref{prop:induced} implies the result.
\end{proof}

We can pull $c$ back along a map $f \colon Y \to X$ so, even if it is not trivializable on
$X$, it may be on $Y$. 

\begin{prop}
A trivialization of $f^* c$ induces a 
$\left(\sigma_Y^{d+1} , \rho_Y\right)$-module structure on any theory with a 
$\left(\sigma_{X , c}^{d+1} , \rho_{X,c}\right)$-module structure.
\label{prop:reduction_Y}
\end{prop}

\begin{proof}
A trivialization $t$ of $f^* c$ determines a correspondence
\begin{equation*}
\begin{tikzcd}
& \left(Y , t\right) \ar{dr}{\id_Y}\ar{dl}{f} & \\
\left(X , c\right) && Y
\end{tikzcd}
\end{equation*}
and the rest of the proof is the same as that of \Cref{prop:reduction_X}.
\end{proof}

\begin{defn}
Let $f \colon Y\to X$ be a map between $\pi$-finite spaces. 
Given a trivialization of $f^* c \colon Y \to B^{d+1} \units{\bk}$, we will refer to the induced module
structure from \Cref{prop:reduction_Y} as a \emph{reduction to $Y$} of the original module
structure. 
\label{defn:reduction}
\end{defn}

\begin{exm}
Let $\widetilde{X} = B\widetilde{G} \to BG = X$ for a group extension $\widetilde{G} \to G$.
I.e. $d = 1$, and the class $c$ classifies this central extension of $G$ as usual.
Theories with a $\left(\sigma_{B\widetilde{G}} , \rho_{B\widetilde{G}}\right)$-module
structure are the same as representations of $\widetilde{G}$, i.e. projective
representations of $G$ with projectivity cocycle $c$. 
\Cref{prop:reduction_Y} then says that splittings of $\widetilde{G}$ over $H \to G$
determine reductions of the linear representations of $\widetilde{G}$ to linear
representation of $H$, as usual. 
See \Cref{exm:proj_repr} and \cite[\S 1]{F:what_is}.
\end{exm}


\bibliographystyle{alphaurl}
\bibliography{references}

\end{document}